\documentclass[reqno]{amsart}

\usepackage{amssymb}
\usepackage{amsfonts}
\usepackage{amsmath}
\usepackage{amsthm} 
\usepackage{xcolor}
\usepackage{graphicx}
\usepackage{enumitem}
\setlist[itemize]{leftmargin=6mm} 

\usepackage{mathptmx}
\usepackage[mathscr]{eucal}

\usepackage{wrapfig}
\usepackage[labelformat=empty]{subcaption}

\usepackage[margin=3.0cm]{geometry}
\usepackage{comment}

\usepackage[colorlinks=true,citecolor=green,urlcolor=blue,linkcolor=blue]{hyperref}

\numberwithin{equation}{section}

\theoremstyle{plain}
\newtheorem{theorem}{Theorem}[section]
\newtheorem{proposition}[theorem]{Proposition} 
\newtheorem{lemma}[theorem]{Lemma} 
\newtheorem{corollary}[theorem]{Corollary}
\newtheorem{conjecture}{Conjecture} 

\theoremstyle{definition}
\newtheorem{definition}[theorem]{Definition} 
\newtheorem{remark}[theorem]{Remark} 
\newtheorem{example}[theorem]{Example}

\renewcommand{\S}{\mathbb{S}}
\newcommand{\R}{\mathbb{R}}

\newcommand{\scalarp}[2]{\langle#1,#2\rangle}
\newcommand{\cH}{\mathcal{H}}

\newcommand{\cA}{\mathcal{A}}
\newcommand{\cF}{\mathcal{F}}
\newcommand{\cK}{\mathcal{K}}

\newcommand{\be}{\mathbf{e}}

\newcommand{\sH}{\mathscr{H}}

\renewcommand{\tilde}{\widetilde}

\renewcommand{\div}{\mathrm{div}}
\newcommand{\norm}{\varphi}

\newcommand{\p}{\partial}

\newcommand{\dist}{\mathrm{dist}}
\newcommand{\Lip}{\mathrm{Lip}}
\newcommand{\intpart}[1]{\left\lfloor #1 \right\rfloor} 
\newcommand{\cl}[1]{\overline{#1}}
\renewcommand{\bar}{\overline}

\newcommand{\black}{\color{black}}

\date{\today}

\title[Crystalline elastic flow]{Crystalline elastic flow 
of polygonal curves: long time behaviour and convergence to stationary solutions}

\author[G. Bellettini]{Giovanni Bellettini} 
\address[G.Bellettini]{University of Siena, Via Roma 56, 53100 Siena, Italy 
\& 
International Centre for
Theoretical Physics (ICTP), Strada Costiera 11, 34151 Trieste, Italy}
\email{giovanni.bellettini@unisi.it}

\author[Sh. Kholmatov]{Shokhrukh Yu. Kholmatov}
\address[Sh. Kholmatov]{University of Vienna,Oskar-Morgenstern-Platz 1, 
1090  Vienna, Austria 
\&  Samarkand State University, University boulevard 15, 
140104 Samarkand, Uzbekistan}
\email{shokhrukh.kholmatov@univie.ac.at}

\author[M. Novaga]{Matteo Novaga}
\address[M. Novaga]{University of Pisa, Largo Bruno Pontecorvo 5, 56127  
Pisa, Italy}
\email{matteo.novaga@unipi.it}

\begin{document}

\begin{abstract}
Given a planar crystalline anisotropy, we study the crystalline elastic flow of immersed 
polygonal curves, possibly also unbounded.
Assuming that the segments evolve by parallel translation 
(as it happens in the standard crystalline curvature flow), 
we prove that a unique regular flow exists
 until a maximal time when some segments having zero crystalline curvature 
disappear. Furthermore, for 
closed polygonal curves, we analyze the behaviour 
at the maximal time, and show that 
it is possible
to restart  the flow finitely many times, 
yielding a globally in time evolution,
 that preserves the index of the curve. Next, we investigate the 
long-time properties of the flow using a Lojasiewicz-Simon-type inequality, 
and show that, as time tends to infinity, the flow fully 
converges to a stationary curve. We also provide a 
complete classification of  the stationary solutions and a 
partial classification of the translating solutions in the case of 
the square anisotropy. 


\end{abstract}



\maketitle


\section{Introduction}

Geometric evolution equations driven by curvature, especially flows governed by anisotropic energies, play a central role in the analysis of interface motion and shape optimization and have been widely studied in the context of material science (see e.g. \cite{BBBP:1997, Cahn:1991, CHT:1992, GN:2000, Herring:1951, Herring:1999, KL:2001, Mullins:1956}). When the anisotropy is crystalline -- i.e., the interfacial energy density is piecewise linear -- the corresponding geometric flows 
exhibit non-smooth and non-local behaviors, and the evolution is 
usually restricted to the family of the so-called admissible polygonal curves, whose segments 
translate in 
the normal direction with velocity depending on their crystalline curvature  
(see e.g. \cite{Andrews:2002,CHT:1992,ELM:2021,GGM:1998, GP:2022, IS:1998, Taylor:1978} and the references therein).

In the current paper  we focus on the gradient flow of the  crystalline elastic energy of planar polygonal curves, namely the gradient flow associated to the energy functional 
$$
\cF_\alpha(\Gamma) = \int_\Gamma 
\left(1+\alpha\left(\kappa_\Gamma^\norm\right)^2
\right)\norm^o(\nu_\Gamma)\,d\sH^1, 
$$
including both an anisotropic length and a $\norm$-curvature term. 
Here  $\alpha>0$, $\norm$ is a crystalline anisotropy in $\R^2,$ $\norm^o$ is its dual, $\Gamma$ is $\norm$-admissible polygonal curve (possibly having self-intersections), $\nu_\Gamma$ is the unit normal of $\Gamma,$ defined $\sH^1$-a.e. and $\kappa_\Gamma^\norm$ is the crystalline curvature.

When $\alpha=0,$ the functional $\cF_0$ gives the anisotropic length,
whose gradient flow is well-studied 
in the literature (see e.g.  \cite{Bellettini:2004_ha, BM:2009, GP:2022}). As formulated originally by Taylor \cite{Taylor:1992}, the segments of $\Gamma$ (locally admissible with the Wulff shape $W^\norm=\{\norm\le1\},$ Section \ref{subsec:admissible_curve}), translate in 
the normal direction and a general form of the evolution  equation is 
\begin{equation}\label{izuv7bn}
V_S = g(\nu_S,\kappa_S^\norm) \quad \text{on segment $S.$}
\end{equation}
Here $V_S$ is the translation velocity and  $g:\S^1\times\R^2\to\R$ is an appropriate function. A natural choice $g(\nu_S,\kappa_S^\norm)=-\kappa_S^\norm$ is developed  in many papers, but unless $W^{\norm^o}$ is cyclic, i.e., inscribed to a circle, the 
corresponding flow starting from $W^\norm$ does not shrink self-similarly. 
To obtain the natural self-shrinking property of $W^\norm$ one has to choose $g(\nu_S,\kappa_S^\norm) =-\norm^o(\nu_S)\kappa_S^\norm$ which in some sense suggests that segments translate not in the normal, rather in the intrinsic direction of the Cahn-Hoffman vector on $S$ (see e.g. \cite{BP:1996} in the case of 
a strictly convex smooth, i.e., regular,
 anisotropy). All in all, the evolution \eqref{izuv7bn} is uniquely represented by a system of ODEs. 
Interestingly, a similar characterization of crystalline curvature flow does
 not exist in higher dimensions mainly due to the facet-breaking  phenomenon\footnote{One may observe facet-breaking or facet-bending phenomena also in the planar case if we add an external force to equation \eqref{izuv7bn} \cite{GR:2008}.} \cite{BNP:1999}.  

In the Euclidean setting, i.e., when $\norm$ is Euclidean, the evolution equation solved by the elastic flow $\{\Gamma(t)\}$ of curves, parametrized by a smooth family of immersions $\gamma:[0,T)\times\S^1\to\R^2,$ 
reads as 
\begin{equation}\label{Xhct6bb}
\begin{cases}
\p_t\gamma = -2\alpha \p_{ss}\vec \kappa -  2\alpha|\vec\kappa|^2\vec\kappa + \vec\kappa,\\
\gamma(0,\cdot) = \gamma_0,
\end{cases}
\end{equation}
where $\vec\kappa= \kappa\nu$ is the curvature vector of the curve $\gamma(t,\cdot)$ at time $t,$ $\gamma_0$ is the 
initial immersion and $\p_{ss}\vec\kappa$ is second derivative of the curvature vector in an arclength parametrization (see \cite{DKSch:2002, Mantegazza:2002, MP:2021_cvpde}). 
Observe that, since the equation is of fourth order, 
no general comparison principles are expected for this elastic flow. 
It is known \cite{MP:2021_cvpde} that the evolution starting from a closed curve 
globally exists and, as time tends to infinity, the flow  stays in a compact region of $\R^2$ and  converges to a stationary solution of \eqref{Xhct6bb}. Also, several results are known for unbounded curves (see 
e.g. the recent survey \cite{MPP:2021}). On the other hand,
various problems are still open, such as 
the general shapes of stationary solutions.  
An interesting open problem,
to our knowledge set forth by G. Huisken,
is whether or not the flow starting from a curve sitting
in the upper half-plane may be, 
at some time during the evolution, completely contained in the lower half-plane, which somehow resembles translating solutions like grim reaper in the mean curvature flow.

Without zero-order terms, the evolution equation \eqref{Xhct6bb} becomes 
\begin{equation*}
V = - \p_{ss}\kappa,
\end{equation*}
which is called a \emph{surface diffusion flow}. This equation is often considered in the anisotropic setting as 
\begin{equation*}
V = -\p_s (M_0(\nu)\p_s\kappa^\norm),
\end{equation*}
where now $\norm$ is a regular anisotropy, $\kappa^\norm$ is the scalar anisotropic curvature and $M_0>0$ is a mobility function \cite{CT:1994, FFLM:2011, GG:2023}. 
This equation can be defined also in the crystalline setting after restricting the evolution to a specific class of polygons 
(sometimes called admissible) and after reducing the evolution to a 
system of ODEs, see e.g. \cite{CRCT:1995}.
Unlike the planar crystalline curvature flow, 
planar surface diffusion flow seems to exhibit facet-breaking, see \cite{GG:2023}.  

In the current paper we study 
short and long time properties of the gradient flow 
associated to the functional $\cF_\alpha$ with a fixed $\alpha>0.$ 
As in the crystalline curvature setting, 
the regular flow is defined for curves $\Gamma$ admissible with the Wulff shape 
$W^\norm$;  
during the evolution we make the assumption  that 
segments $S$ of $\Gamma$ translate in the normal direction 
with velocity equal to $-\norm ^o(\nu_S) \delta \cF_\alpha,$ 
where $\delta\cF_\alpha$ is a formal notation 
to indicate
the first variation of $\cF_\alpha$.
Also in this case, 
one gets that Wulff shapes shrink self-similarly
(Example \ref{ex:evol_Wulff}). As usual, the evolution is formulated 
as a system of ODEs that govern the signed Euclidean distances 
(called here also signed heights) of each segment along its normal direction (Definition \ref{def:cryst_elas_flow}).
As already said, 
we are supposing that neither new edges do appear nor segments break or bend
along the flow.
This simplifying assumption
is mainly due to the fact that, in general, we miss 
a stability property of the flow with respect to insertion
of  small segments (possibly having 
zero $\norm$-curvature) near the vertices of a given initial curve.

A main contribution of the present  paper is the rigorous construction and analysis of the unique polygonal  crystalline elastic flow. 
More specifically, we prove short-time existence and uniqueness of the flow (Theorem \ref{teo:existence}), and show that the evolution continues through a restarting mechanism 
after some segments with zero $\norm$-curvature vanish -- 
thus guaranteeing global existence and uniqueness in time
(Theorem \ref{teo:continue_flows}). Segments with nonzero $\norm$-curvature,
instead, do not vanish.
Furthermore, by means of a crystalline version of a 
Lojasiewicz-Simon inequality (Propositions \ref{prop:loja_simon_ineq} and \ref{prop:loja_simon_ineqII}), we investigate the long-time behavior of solutions 
starting from a closed polygonal curve: as in the 
Euclidean case \cite{MP:2021_cvpde}, we prove that the flow converges to a stationary solution as $t\to+\infty$ (Theorems \ref{teo:conver_stationar_sol} and \ref{teo:long_time_general}).

We are also interested in 
a classification of  special solutions, such as stationary and translating ones. 
These problems 
are quite difficult in a general crystalline setting, and 
in the present paper we restrict to consider
the case of a square Wulff shape: here we are able to establish 
a complete classification of stationary solutions and a characterization of translating solutions under 
some appropriate assumptions (Section \ref{sec:square_anisotropy}). 
We remark that, unlike in the Euclidean elastic case, 
our crystalline elastic flow preserves convexity, 
a property that could be probably
 related to our assumptions of facets non-breaking 
and prohibition of spontaneous edge creation near vertices 
(Corollary \ref{cor:convex_evolution}). 

The paper is organized as follows. In Section \ref{sec:preliminary} we 
introduce some preliminaries.
The anisotropic elastic functional and the  crystalline elastic flow are introduced in Section \ref{sec:willmore_funca}. Section \ref{sec:existence} is dedicated to the investigation of existence, uniqueness and 
restart of the crystalline elastic flow.
For 
convenience of the reader, in Section \ref{sec:example}
we exhibit
 some explicit examples 
(evolution of Wulff shapes and grim reaper-type solutions).
Stationary solutions and Lojasiewicz-Simon-type properties of stationary curves are studied in Section \ref{sec:stationar_solutino}
(and are needed in 
Section \ref{sec:long_time_behaviour}), 
where  we prove the full convergence of the flow to a stationary solution 
as time converges to infinity. In Section \ref{sec:square_anisotropy},
 designed for the square anisotropy, 
we provide a classification of some special 
-- stationary and translating -- solutions. 
We conclude the paper pointing out some open problems in Section \ref{sec:open_problems}.

\section{Notation and main definitions}\label{sec:preliminary}

\subsection{Anisotropy}

An anisotropy in $\R^2$ is a positively one-homogeneous convex function $\norm:\R^2\to[0,+\infty)$ satisfying 
\begin{equation}\label{norm_bounds_12120}
c_\norm|x|\le \norm(x)\le C_\norm|x|,\quad x\in\R^2,
\end{equation}
for some $0<c_\norm\le C_\norm<+\infty.$ 
The closed convex set $W^\norm:=\{\norm\le 1\}$ is called the unit $\norm$-ball or sometimes 
the Wulff shape (of $\norm$). Similarly, the set $\{\norm(\cdot-y)\le r\}$ is called the Wulff shape of radius $r$ centered at
$y.$ The anisotropy
$$
\norm^o(x):=\max_{y\in W^\norm} \scalarp{x}{y},\quad x\in\R^2,
$$ 
is called the dual of $\norm$, where $\scalarp{\cdot}{\cdot}$ is the 
Euclidean scalar product in $\R^2.$ When $W^\norm$ is a polygon, we say $\norm$ is crystalline and the boundary segments of $W^\norm$ will be called facets. Note that $\norm$ is  crystalline if and only if so is $\norm^o.$ In what follows we shorthand $\nu^\norm:= \frac{\nu}{\norm^o(\nu)}$ for $\nu\ne0.$

\begin{wrapfigure}[13]{l}{0.35\textwidth}
\vspace*{-3mm}
\includegraphics[width=0.32\textwidth]{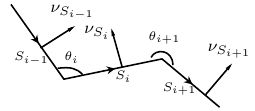}
\caption{\small The angles of a polygonal line. Here: $\theta_i:=\angle(S_{i-1},S_i)\in(0,\pi)$ as $S_{i-1}$ and $S_i$ form a convex cone (whose inner normals are $\nu_{S_{i-1}}$ and $\nu_{S_{i}}$), while
$\theta_{i+1}:=\angle(S_i,S_{i+1})\in(\pi,2\pi)$ since the cone formed by $S_{i}$ and $S_{i+1}$ (whose inner normals are $\nu_{S_i}$ and $\nu_{S_{i+1}}$) is concave.}
\label{fig:angles}
\end{wrapfigure}

\subsection{Curves}
A curve $\Gamma$ in $\R^2$ is the image of a Lipschitz 
continuous function $\gamma:I\to\R^2,$ where $I$ is one of $[0,1],$ $[0,1)$ or $(0,1)$, 
depending on whether $\Gamma$ is bounded,  bounded by one end,
or  having two unbounded ends. The function $\gamma$ is called a \emph{parametrization} of $\Gamma.$ When  $\gamma(0) = \gamma(1),$ we say $\Gamma$ is \emph{closed}. Our
curves may have self-intersections.  If $\gamma$ is $C^1$ (resp. Lipschitz) and $|\gamma'| > 0$ in $[0,1]$ (resp. a.e. in $[0,1]$), it is called a \emph{regular parametrization} of $\Gamma.$ A curve
$\Gamma$ is $C^{k+\alpha}$ for some $k\ge0$ and  $\alpha\in[0,1]$, $k+\alpha \geq 1$, if it admits a regular $C^{k+\alpha}$-parametrization. The tangent line 
to $\Gamma$ at a point $p\in \Gamma$ is denoted $T_p\Gamma$ (provided it exists).
The (Euclidean) unit tangent vector to $\Gamma$ at $p$ is denoted $\tau_\Gamma(p)$ and the unit normal vector is $\nu_\Gamma(p) = \tau_\Gamma(p)^\perp,$ where ${}^\perp$ is the counterclockwise $90^o$ rotation.
When there is no risk of confusion, we simply write $\tau$ and 
$\nu$ in place of $\tau_\Gamma$ and $\nu_\Gamma$. If $p= \gamma(x)$ and $\gamma$ is differentiable at $x$, then 
$$
\hspace*{5cm}
\tau(p) = \frac{\gamma'(x)}{|\gamma'(x)|}
\qquad\text{and}\qquad 
\nu(p) = \frac{\gamma'(x)^\perp}{|\gamma'(x)|}.
$$
Unless otherwise stated, we choose tangent vectors in the direction of the parametrization  and closed curves are oriented in the clockwise order,
 so that,
when $\Gamma$ is bounded and embedded, the unit  normal 
points outside the bounded region enclosed by $\Gamma.$ 
The same convention is taken for the orientation of
the  boundary of the Wulff shape $W^\norm.$

A closed curve $\Gamma$ is \emph{polygonal} if it is a finite union of segments. We frequently represent  $\Gamma$ as a union $\cup_{i=1}^N S_i$ of its segments, counted in the 
order. 
Sometimes we consider triplets $S_{i-1},$ $S_i,$ $S_{i+1}$ of consecutive segments with the conditions that $S_0:=S_N$ and $S_{N+1}:=S_1.$ Similar notation 
is used for all quantities involving indexation over $i=1,\ldots,N.$
The angle $\angle(S_{i-1},S_{i})$ between the segments $S_{i-1}$ and $S_{i}$ is denoted by $\theta_i;$  for convenience in computations, we choose $\theta_i\in(0,\pi)$ if $S_{i-1}$ and $S_{i}$ form a convex cone for which $\nu_{S_{i-1}}$ and  $\nu_{S_{i}}$ are interior, otherwise we choose $\theta_i\in(\pi,2\pi),$ see e.g. Fig. \ref{fig:angles}.

We say an unbounded curve $\Gamma$ is 
\emph{polygonal} provided that there exists $r_0>0$ such that for any $r>r_0,$ $B_r(0)\cap \Gamma$ is a polygonal curve and $\Gamma\setminus B_r$  is a union of two half-lines. If $\Gamma=\cup_{i=1}^nS_i,$ the angles $\theta_i:=\angle (S_{i-1},S_{i})$ for $2\le i\le n$ are defined as in the closed case with the convention $\theta_1=\theta_{n+1}=0.$ 

A curve $\Gamma$ is \emph{rectifiable} if $\sH^1(\Gamma)<+\infty.$ 
By definition, any polygonal curve is (locally) rectifiable. By \cite[Lemmas 3.2, 3.5]{Falconer:1985} any rectifiable curve $\Gamma$ admits a unit tangent vector
$\tau$ (and a corresponding unit normal $\nu$) $\sH^1$-a.e. defined.

For shortness, let us call a polygonal curve $\Gamma$ \emph{convex} provided that all its angles either belong to $(0,\pi)$ or to $(\pi,2\pi)$ simultaneously (indeed, in the latter, we can always reorient the curve to reduce to $(0,\pi)$). 

\subsection{Admissible polygonal curves}\label{subsec:admissible_curve}

Let $\norm$ be a crystalline anisotropy. A polygonal curve $\Gamma$ is called  \emph{$\norm$-admissible} (admissible for short) if it admits a parametrization such  that for every segments $S'$ 
and $S''$ of $\Gamma$ having a common vertex there are facets $F'$ and $F''$ of $W^\norm$ having a common vertex such that $\nu_{S'} =\nu_{F'}$ and $\nu_{S''}=\nu_{F''}.$
\begin{figure}[htp!]
\centering 
\includegraphics[width=0.8\textwidth]{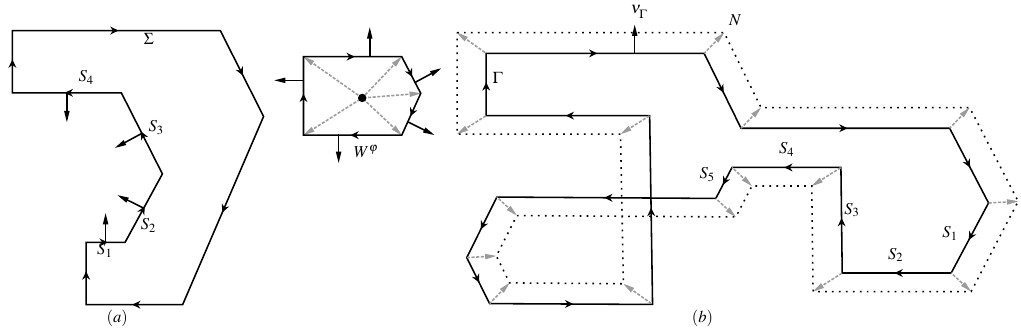}
\caption{\small Non admissible curve $\Sigma$ (left) 
and admissible curve $\Gamma$ with a CH field (right) with
respect to the pentagonal anisotropy $W^\norm$ (in the center). Note that in (a) for the pairs $(S_1,S_2),$ $(S_2,S_3)$ and $(S_3,S_4)$ there are no pairs 
of consecutive facets of $W^\norm$ having the same normals (also because there is no facet of $W^\norm$ with the  unit normal $\nu_{S_2}$ or $\nu_{S_3}$). The curve $\Gamma$ in (b) is locally convex near $S_1$ and $S_2,$
locally concave around $S_4$ and neither concave nor convex near $S_3$ and $S_5.$}\label{fig:phi_regular}
\end{figure}

Note that every (closed or unbounded) admissible polygonal curve $\Gamma$ is Lipschitz $\norm$-regular. Indeed, in this case, the CH field is uniquely defined at the  vertices of $\Gamma$ as the corresponding vertices of
$W^{\norm}$ and then, for instance, interpolated linearly along the segments, see Fig. \ref{fig:phi_regular}.

\subsection{Tangential divergence}

The \emph{tangential divergence} of a vector field $g\in C^1(\R^2; \R^2)$ over a  Lipschitz curve $\Gamma$ is defined as
$$
\div_\tau g(p) = \scalarp{ \nabla g(p)\tau(p) }{ \tau(p)} 
\quad \text{for $\sH^1$-a.e. $p\in\Gamma$}.
$$
The tangential divergence can also be introduced using parametrizations. More precisely, let $\gamma:[0,1]\to \R^2$ be a regular Lipschitz parametrization
of $\Gamma$ and $g:\Gamma\to \R^2$ a Lipschitz vector field along
$\Gamma$, i.e., $g\circ \gamma\in \Lip([0,1]; \R^2)$. Then
\begin{equation*}
\div_\tau\, g(p) = \frac{ \scalarp{[g\circ \gamma]'(x) }{ \gamma'(x)} 
}{|\gamma'(x)|^2},\qquad p =\gamma(x)
\end{equation*}
at points of differentiability. One checks that the tangential divergence is independent of the parametrization.

\subsection{Cahn-Hoffman vector fields}

Let $\Gamma$ be a rectifiable curve, with 
$\sH^1$-almost everywhere defined unit normal $\nu.$ A vector field
$N:\Gamma\to \p W$ is called a \emph{Cahn-Hoffman field} (CH field) if
\begin{equation*}
\scalarp{N }{\nu} = \norm^o(\nu)\,\, \quad 
\text{$\sH^1$-a.e. on $\Gamma,$}
\end{equation*}
namely $N(x)\in \p\norm^o(\nu(x))$ for $\sH^1$-a.e. $x \in \Gamma,$ where $\p$ stands for the subdifferential. 
Notice that reversing the orientation of the curve translates into a change of sign of $\nu$ and of the corresponding CH field, which is always ``co-directed''
as $\nu$. We indicate the collection of all Cahn-Hoffman vector vields along $\Gamma$ by $CH(\Gamma).$

\begin{definition}[\textbf{Lipschitz $\norm$-regular curve}]
We say the curve $\Gamma$ is \emph{Lipschitz $\norm$-regular} ($\norm$-\emph{regular}, for short) if it admits a \emph{Lipschitz} CH field.
\end{definition}

\subsection{Index of a $\norm$-regular curve} 

Classically, the index (an integer number) of a closed smooth planar curve is 
the sum of the positive (counterclockwise)
and negative (clockwise)
full turns 
on $\S^1$ made  by the normal vector field to the curve. The same
definition can be given in our crystalline context,
provided that the normal vector field is replaced by 
a Cahn-Hoffman vector field, and $\S^1$ by the Wulff shape:

\begin{definition}[Index]\label{def:curve_index}
Let $\Gamma$ be a $\norm$-regular curve and $N^0\in CH(\Gamma).$ We define the index of $\Gamma$ as the signed number of complete turns of $N^0$ over $\p W^\norm.$
\end{definition}

One checks that the index is independent of  $N^0$ and is well-defined also for unbounded curves. 
Moreover, if $\Gamma$ has the reversed orientation, then the 
index changes sign.

\subsection{Crystalline curvature}
Given a crystalline anisotropy $\norm$ and a $\norm$-regular curve $\Gamma,$ one can readily check that the minimum problem
$$
\inf_{N\in CH(\Gamma)} \int_\Gamma \norm^o(\nu_\Gamma) (\div_\tau N)^2d\sH^1
$$
admits a solution. Clearly, even though minimizers are not unique, their tangential divergence is always the same. For a minimizer $N^0\in CH(\Gamma),$ 
at every point $p\in\Gamma$ where $N^0$ is differentiable, the number 
$$
\kappa_\Gamma^\norm(p):= \div_\tau N^0(p)
$$
is called the $\norm$-curvature of $\Gamma$ at $p.$
When $\Gamma$ is an admissible polygonal curve, the minimizer $N^0$ is uniquely defined: on a vertex it coincides with the corresponding vertex of $W^\norm$ and then it is linearly interpolated along segments and constantly 
extended along 
half-lines. In particular, $\kappa_\Gamma^\norm$ is a constant $\kappa_{S_i}^\norm$  on each segment/half-line $S_i$ of $\Gamma.$ One 
checks
that
\begin{equation}\label{curva_definiton}
\kappa_\Gamma^\norm = \frac{c_i\sH^1(F_i)}{\sH^1(S_i)}\quad\text{on $S_i,$}
\end{equation}
where $F_i$ is the facet of $W^\norm$ with $\nu_{F_i}=\nu_{S_i}$ and the sometimes called \emph{transition number} $c_i$
is equal to $+1$ if $\Gamma$ is locally ``convex'' around $S_i$ (see Fig. \ref{fig:sdists} (c)), to $-1$ if $\Gamma$ is locally ``concave'' around $S_i$ (see Fig. \ref{fig:sdists} (a)) and to $0$ if $\Gamma$ is neither
concave, nor convex near $S_i$ (see Fig. \ref{fig:sdists} (b)) or if $S_i$ is a half-line.

\subsection{Parallel polygonal curves and signed height}

In this section we introduce the notion of parallel poly\-gonal curves,   distance vectors and signed heights.
With respect to the usual case, some care is necessary since we are considering curves with possible self-intersections.

\begin{definition}[\textbf{Parallel curves}] 
Let $\Gamma:=\cup_{i=1}^N S_i$ be a polygonal curve consisting of a consecutive union of $N\geq 1$ segments/half-lines $S_1,\ldots,S_N.$ A polygonal 
curve $\bar \Gamma$ is called \emph{parallel} to $\Gamma$ provided that:

\begin{itemize}
\item $\bar \Gamma:=\cup_{i=1}^N \bar S_i$ is a consecutive union of $N$ (nondegenerate, i.e. of positive length) segments/half-lines $\bar S_1,\ldots,\bar S_N,$
 
\item each $S_i$ is parallel to $\bar S_i$ with the same orientation (so that $\nu_{S_i} = \nu_{\bar S_i}$),

\item if $S_i$ is a half-line, then so is $\bar S_i$ with bounded $S_i\Delta \bar S_i.$
\end{itemize}
\end{definition}

Clearly, the corresponding angles of parallel curves are equal, i.e., $\angle(S_i,S_{i+1})=\angle(\bar S_i, \bar S_{i+1}).$
The orientation of segments of curves in the parallelness is  important, see Fig. \ref{fig:nonparallel}.

\begin{figure}[htp!]
\includegraphics[width=0.8\textwidth]{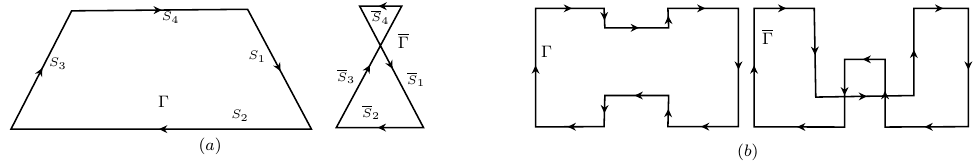} 
\caption{\small Nonparallel and parallel curves. Note that in (a) the segments $S_4$ and $\bar S_4$ have different orientations. }\label{fig:nonparallel}
\end{figure}

Note that if $\norm$ is crystalline and $\Gamma$ is an admissible polygonal curve, then every polygonal curve $\bar \Gamma$ parallel to $\Gamma$ is also admissible. Moreover,
$$
\text{if}\quad \kappa_\Gamma^\norm = \frac{c_i\sH^1(F_i)}{\sH^1(S_i)}\quad \text{on $S_i,$}  \quad\text{then}\quad
\kappa_{\bar \Gamma}^\norm = \frac{c_i\sH^1(F_i)}{\sH^1(\bar S_i)}
\quad \text{on $\bar S_i.$}
$$
So, the transition numbers of segments do not change in parallel networks.

\begin{definition}[\textbf{Distance vectors}] 
Let $S$ and $T$ be two parallel segments and let $\ell_S$ and $\ell_T$ be the straight lines passing through $S$ and $T.$ A vector $H(S,T)\in\R^2,$ orthogonal to both $\ell_S$ and $\ell_T$ and satisfying $\ell_T = \ell_S +
H(S,T),$ is called a \emph{distance vector} from $S$ to $T.$
When $S$ is oriented by $\nu_S,$ in what follows  we frequently refer to the number 
$$
h:=\scalarp{H(S,T)}{ \nu_{S} }
$$
as the (Euclidean) \emph{signed height} from $S$ to $T.$ Note that $H(S,T) = h\nu_S$.
\end{definition}

Let us obtain the length of the segments of parallel curve by means of  the length of the corresponding segments and signed heights.
\begin{figure}[htp!]
\includegraphics[width=0.95\textwidth]{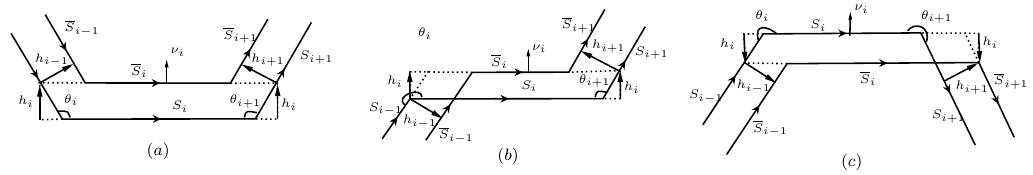}
\caption{\small } \label{fig:sdists}
\end{figure}
Let $\Gamma:=\cup_{i=1}^NS_i$ be a closed polygonal curve and let $\bar\Gamma:=\cup_{i=1}^N\bar S_i$ be a closed polygonal curve, parallel to $\Gamma.$ For any $i\in\{1,\ldots,N\},$ set $h_i:=\scalarp{H(S_i,\bar
S_i)}{\nu_{S_i}}$ and let $\theta_i:=\angle(S_{i-1},S_{i})\in(0,2\pi)\setminus\{\pi\}$ be the angle between $S_i$ and $S_{i+1},$ see Fig. \ref{fig:sdists}. Then elementary geometric arguments show that
\begin{equation}\label{length_changed_hhh}
\sH^ 1(\bar S_i) = \sH^1(S_i) - 
\Big(\frac{h_{i-1}}{\sin\theta_i} + h_i[\cot\theta_i + \cot\theta_{i+1}] +
\frac{h_{i+1}}{\sin\theta_{i+1}}\Big),
\end{equation}
see also \cite[Sec. 3]{GP:2022}.

\subsection{Construction of a parallel curve from the given heights}

Given parallel polygonal curves 
$\Gamma=\cup_{i=1}^nS_i$ and $\bar \Gamma:=\cup_{i=1}^n \bar S_i,$
the $n$-tuple $\{h_1,\ldots,h_n\}$ of the signed heights 
$h_i:=\scalarp{H(S_i,\bar S_i)}{\nu_{S_i}}$ is uniquely defined. Let
us check whether the converse assertion is also true.

\begin{lemma}\label{lem:construc_paral_curve}
Let $\Gamma=\cup_{i=1}^nS_i$ be a polygonal curve and let $\{\theta_i\}_{i=1}^n$ be the corresponding angles of $\Gamma$ (as mentioned earlier, with the assumptions $\theta_1=\theta_{n+1}=0$ if $\Gamma$ is unbounded). Let an $n$-tuple $\{h_1,\ldots,h_n\}$ of real  numbers be such that $h_i=0$ if $S_i$ is a half-line and
\begin{equation}\label{paspaslf}
\max_{1\le i\le n} |h_i| < \min_{1\le i\le n} 
\frac{\sH^1(S_i)}{\frac{1}{|\sin \theta_i|} + |\cot\theta_i +
\cot\theta_{i+1}| + \frac{1}{|\sin \theta_{i+1}|}}.
\end{equation}
Then there exists a unique polygonal curve $\bar\Gamma=\cup_{i=1}^n \bar S_i$ parallel to $\Gamma$ with $\scalarp{H(S_i,\bar S_i)}{\nu_{S_i}} = h_i.$
\end{lemma}

Notice that \eqref{paspaslf} is not sufficient for  ensuring the embeddedness of $\bar\Gamma$ even if $\Gamma$ is embedded.

\begin{proof}
First assume $\Gamma$ is closed. For any $i\in\{1,\ldots,n\}$ 
let $\ell_i$ be the unique straight line parallel to $S_i$ such that $\scalarp{H(S_i,\ell_i)}{\nu_{S_i}} = h_i.$
As $S_i\cap S_{i+1}\ne\emptyset,$ the lines $\ell_i$ and $\ell_{i+1}$
intersect at a unique point $G_i.$ Consider the closed curve $\bar\Gamma:=\overline{G_1G_2\ldots G_N G_1},$ obtained by connecting these intersection points with  segments $\bar S_i:=[G_{i-1}G_i]$ consecutively. Clearly, $\bar\Gamma$ is uniquely defined and, by construction,
$\scalarp{H(S_i,\bar S_i)}{\nu_{S_i}} = h_i$ for any $i\in\{1,\ldots,n\}.$

We claim that $\bar\Gamma$ is parallel to $\Gamma.$ 
Indeed, for any $t\in[0,1]$ let $\bar \Gamma(t)$ be the polygonal curve associated to the $n$-tuple $\{th_1,\ldots,th_n\},$
constructed as above. Since $n$ is finite and $t\mapsto \Gamma(t)$ is Kuratowski continuous,  $\bar\Gamma(t)$ is parallel to $\Gamma$ for small $t>0.$ Let $\bar
t\in(0,1]$ be the first time for which $\bar\Gamma(\bar t)$ is 
not parallel to $\Gamma.$ Since each segment is continuously translating along its normal direction,  the only way to fail the 
parallelness is that some segment
of $\bar\Gamma(\bar t)$ becomes degenerate (i.e., has $0$-length). 
However, by \eqref{length_changed_hhh} and \eqref{paspaslf}, for any $t\in(0,\bar t)$ and $i\in\{1,\ldots,n\}$  we have
\begin{align*}
\sH^1(S_i(t)) = & \sH^1(S_i) - \Big[ \frac{th_{i-1}}{\sin\theta_i} +
th_i[\cot\theta_i + \cot\theta_{i+1}] +
\frac{th_{i+1}}{\sin\theta_{i+1}}\Big] \\
\ge & \sH^1(S_i) - \Big[\frac{1}{|\sin\theta_i|} +
|\cot\theta_i + \cot\theta_{i+1}| +
\frac{1}{\sin\theta_{i+1}}\Big]\max_i|h_i|>0.
\end{align*}
Thus, by the Kuratowski continuity of $S_i(t),$ letting 
$t\nearrow \bar t$ we conclude $\sH^1(S_i(\bar t))>0$ for all $i\in\{1,\ldots,n\},$ a contradiction. This contradiction shows that such $\bar t\in(0,1]$ does not exist and $\bar \Gamma = \bar\Gamma(1)$ is parallel to $\Gamma.$

In case $\Gamma$ is unbounded, i.e., $S_1$ and $S_n$ are half-lines, as above we define the polygonal curve $\overline{G_1G_2\ldots G_n}$ and then append two half-lines $\bar S_1$ and $\bar S_n$ respectively at $G_1$ and $G_n$ with the property that both $\bar S_1\Delta S_1$ and $\bar S_n\Delta S_n$ are bounded.
\end{proof}

\section{Anisotropic elastic functional}\label{sec:willmore_funca}

Let $\norm$ be an anisotropy. Given a $\norm$-regular curve $\Gamma,$ $\alpha>0$ and open set $D\subseteq\R^2,$
we define the \emph{anisotropic elastic functional}
$$
\cF_\alpha(\Gamma,D):=\int_{D\cap \Gamma} \left(1+\alpha \left(
\kappa_\Gamma^\norm\right)^2\right)\norm^o(\nu_\Gamma)\,d\sH^1,
$$
where for simplicity we suppress the dependence of $\cF_\alpha$ on $\norm.$ We write $\cF_\alpha:=\cF_\alpha(\cdot,\R^2).$ Note
that in general, $\cF_\alpha(\Gamma,D)$ may be infinite.

\begin{wrapfigure}[9]{l}{0.2\textwidth}
\vspace*{-3mm}
\includegraphics[width=0.2\textwidth]{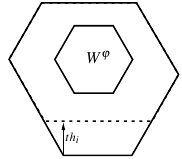}
\caption{\small }\label{fig:first_var}
\end{wrapfigure}
We are interested in the case of a crystalline $\norm$ and an  admissible $\Gamma,$ in which case 
$\cF_\alpha(\Gamma,D)<+\infty$ for any bounded open set $D\subset\R^2.$ Let us compute the first variation of $\cF_\alpha$  in the simplest case where only one segment $S_i$ is translated while keeping others untranslated, but
shortened or elongated when necessary (see Fig. \ref{fig:first_var}). Given a signed height $h_i$ and a real number $t\in\R$ with $|t|<<1,$ let
$\Gamma(t)$ be the polygonal curve, parallel to $\Gamma$ such that $\scalarp{H(S_j,S_j(t))}{\nu_{S_j}}$ is $0$
for $j\ne i$ and is $th_i$ for $j=i.$ Then for any large disc $D$ compactly containing $S_i$ and $\cup_t S_i(t),$
\begin{multline*}
\hspace*{3.4cm}
\cF_\alpha(\Gamma(t),D) - \cF_\alpha (\Gamma,D) = \sum_{j=i-1,i,i+1} \norm^o(\nu_{S_j}) \Big[ \sH^1(D\cap S_j(t))
- \sH^1(D\cap S_j)\Big] \\
+ \alpha\sum_{j=i-1,i,i+1} \norm^o(\nu_{S_j}) c_j^2 \sH^1(F_j)^2 \Big[ \frac{1}{\sH^1(D\cap S_j(t)) } -
\frac{1}{\sH^1(D\cap S_j) }\Big],
\end{multline*}
where $c_i$ are transition numbers, $F_i$ is the unique facet of $W^\norm$ with $\nu_{S_i} = \nu_{F_i}$ and
$\theta_i:=\angle(S_i,S_{i+1})$ are the angles of $\Gamma.$
As $h_j=0$ for $j\ne i,$ by the choice of $D$ and \eqref{length_changed_hhh}, using $h_{i-1}=h_{i+1}=0,$ 
\begin{gather*}
\sH^1(D\cap S_{i-1}(t)) = \sH^1(D\cap S_{i-1}) - \frac{th_i}{\sin \theta_i}, \\
\sH^1(S_i(t)) = \sH^1(S_i) - t h_i(\cot \theta_i+\cot \theta_{i+1}), \\
\sH^1(D\cap S_{i+1}(t)) = \sH^1(D\cap S_{i+1}) - \frac{th_i}{ \sin\theta_{i+1} }.
\end{gather*}
Thus, 
\begin{multline*} 
\frac{d}{dt}\cF_\alpha(\Gamma(t),D) \Big|_{t=0} = -h_i \Big(\tfrac{\norm^o(\nu_{S_{i-1}}) }{\sin\theta_i } 
+ \norm^o(\nu_{S_i})[\cot \theta_i+\cot \theta_{i+1}] 
+ \tfrac{\norm^o(\nu_{S_{i+1}}) }{\sin\theta_{i+1} }\Big) \\
+\alpha h_i \Big(\tfrac{c_{i-1}^2\sH^1(F_{i-1})^2 \norm^o(\nu_{S_{i-1}})}{\sH^1(D\cap S_{i-1})^2\sin\theta_i}
+ \tfrac{c_{i}^2\sH^1(F_{i})^2\norm^o(\nu_{S_{i}}) [\cot \theta_i+\cot \theta_{i+1}]}{\sH^1(S_{i})^2}
+\tfrac{c_{i+1}^2\sH^1(F_{i+1})^2\norm^o(\nu_{S_{i+1}})}{\sH^1(D\cap S_{i+1})^2\sin\theta_{i+1}}\Big),
\end{multline*}
or equivalently
\begin{multline*}
\frac{d}{dt}\cF_\alpha(\Gamma(t),D) \Big|_{t=0} = -\int_{S_i} \tfrac{h_i}{\sH^ 1(S_i)}
\Big(\tfrac{\norm^o(\nu_{S_{i-1}}) }{ \sin\theta_i } + \norm^o(\nu_{S_i})[\cot \theta_i+\cot \theta_{i+1}]+
\tfrac{\norm^o(\nu_{S_{i+1}}) }{ \sin\theta_{i+1} }\Big)\,d\sH^ 1 \\[1mm]
+ \alpha \int_{S_i} \tfrac{h_i}{\sH^1(S_i)} \Big(\tfrac{c_{i-1}^2\sH^1(F_{i-1})^2\norm^o(\nu_{S_{i-1}})
}{\sH^1(D\cap S_{i-1})^2\sin\theta_i} + \tfrac{c_{i}^2\sH^1(F_{i})^2\norm^o(\nu_{S_{i}})[\cot \theta_i+\cot
\theta_{i+1}]}{\sH^1(S_{i})^2}+\tfrac{c_{i+1}^2\sH^1(F_{i+1})^2\norm^o(\nu_{S_{i+1}})}{\sH^1(D\cap
S_{i+1})^2\sin\theta_{i+1}}\Big)\,d\sH^ 1.
\end{multline*}

Recall that $\Gamma$ is admissible and hence, there are unique facets $F_{i-1},$ $F_i$ and $F_{i+1}$ of $W^\norm$ such that
$\nu_{F_j}=\nu_{S_j}$ for $j=i-1,i,i+1.$ Then direct geometric computations show that if $S_{i-1},S_i,S_{i+1}$ are
as in  Fig. \ref{fig:betayin}, then
\begin{equation}\label{ashuc7bnn}
\frac{\norm^o(\nu_{F_{i-1}}) }{ \sin\theta_i } + \norm^o(\nu_{F_i})[\cot \theta_i+\cot \theta_{i+1}]+
\frac{\norm^o(\nu_{F_{i+1}}) }{ \sin\theta_{i+1} } =
-c_i \sH^1(F_i) =
\begin{cases}
-\sH^1(F_i) &  \text{in case (a)}  \\
0 &  \text{in case  (b)} \\
\sH^1(F_i) &  \text{in case (c).}  
\end{cases}
\end{equation}
\begin{figure}[htp!]
\includegraphics[width=0.9\textwidth]{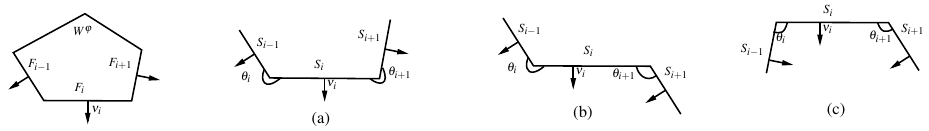}
\caption{\small The three cases in \eqref{ashuc7bnn}.}\label{fig:betayin}
\end{figure}
Indeed, even though this identity is well-known in the crystalline literature \cite{Taylor:1992}, for completeness  we provide a short explanation.  Suppose that we are as in Fig. \ref{fig:betayin} (b). Then necessarily $\nu_{F_{i-1}} = \nu_{F_{i+1}},$ and thus  $S_{i-1}$ and $S_{i+1}$ are parallel and $\theta_i+\theta_{i+1}=2\pi.$ Thus, the sum on the left-hand side of \eqref{ashuc7bnn} is $0.$ 
Next, suppose that we are as in Fig. \ref{fig:betayin} (a).

\begin{wrapfigure}{l}{0.45\textwidth}
\vspace*{-4mm}
\includegraphics[width=0.44\textwidth]{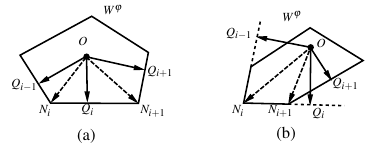}
\caption{\small Two possible Wulff shapes.}\label{fig:mutad}
\end{wrapfigure}
Let $O$ denote the center of $W^\norm,$ $Q_{i-1},Q_i,Q_{i+1}$ be the bases of the heights from $Q$ to facets $F_{i-1},F_i,F_{i+1},$ respectively, and let  $N_i$ and $N_{i+1}$ denote the endpoints of $F_i,$ see Fig. \ref{fig:mutad}. Note that $Q_j$ need not lie on $F_j.$
Let $\beta_{i-1}:=\angle Q_{i-1}ON_i$ and $\beta_i:=\angle N_iOQ_i.$
Clearly, $|OQ_{j}|=|ON_i|\cos\beta_j$ and $\scalarp{\vec{OQ_j}}{\vec{ON_i}} = |OQ_j|\cdot |ON_i|\cos\beta_j$ for $j=i-1,i, $ where $\vec{PQ}$ and $|PQ|$ are the vector $PQ$ (directed from $P$) and length of the segment $PQ.$ On the other hand, 
$$
\hspace*{7cm}
\nu_j = \tfrac{\vec{OQ_j}}{|OQ_j|},\quad \scalarp{\tfrac{\vec{ON_i}}{\norm(\vec{ON_i})} }{\nu_j} = \norm^o(\nu_j),\quad j=i-1,i,
$$
and thus, recalling $\norm(\vec{ON_i})=1,$ we get 
$$
|ON_i|\cos\beta_j = 
\tfrac{|ON_i|\cos\beta_j}{\norm(\vec{ON_i})} = \norm^o(\nu_j),\quad j=1,2.
$$
Now, observing $\angle Q_{i-1}N_iQ_i = 2\pi-\theta_i,$ so that $\beta_{i-1} + \beta_i = \theta_i-\pi,$ we have   
$$
|ON_i|\cos\beta_{i-1} = |ON_i| \cos (\theta_i-\pi-\beta_i) = - |ON_i| \cos \theta_i\cos \beta_i -  |ON_i| \sin  \theta_i\sin \beta_i. 
$$
Therefore, 
$$
|N_iQ_i| = |ON_i| \sin\beta_i = - \frac{\norm^o(\nu_{i-1})}{\sin\theta_i} - \norm^o(\nu_{i}) \cot\theta_i.
$$
Analogously, 
$$
|Q_iN_{i+1}| = 
\begin{cases}
- \frac{\norm^o(\nu_{i+1})}{\sin\theta_{i+1}} - \norm^o(\nu_{i}) \cot\theta_{i+1} & \text{in case Fig. \ref{fig:mutad} (a)},\\[2mm]
\frac{\norm^o(\nu_{i+1})}{\sin\theta_{i+1}} + \norm^o(\nu_{i}) \cot\theta_{i+1} & \text{in case Fig. \ref{fig:mutad} (b)}.
\end{cases}
$$
Now using $\sH^1(F_i) = |N_iN_{i+1}|,$ we deduce \eqref{ashuc7bnn}. 

The case of Fig. \ref{fig:betayin} (c) is analogous.

\begin{lemma}[First variation formula]
Let $\norm$ be a crystalline anisotropy and $\Gamma$ be an admissible polygonal curve. Then for any segment $S_i$ of
$\Gamma$ and a real number $h_i\in\R$ one has
\begin{multline*}
\frac{d}{dt}\cF_\alpha(\Gamma(t),D) \Big|_{t=0} = \int_{S_i} \frac{c_i h_i \sH^1(F_i)}{\sH^ 1(S_i)}\,d\sH^ 1
\\[1mm]
+ \frac{\alpha h_i}{\sH^1(S_i)} \int_{S_i} \Big(\frac{c_{i-1}^2\sH^1(F_{i-1})^2\norm^o(\nu_{S_{i-1}})
}{\sH^1(D\cap S_{i-1})^2\sin\theta_i}
+ \frac{c_{i}^2\sH^1(F_{i})^2\norm^o(\nu_{S_{i}})[\cot \theta_i+\cot \theta_{i+1}]}{\sH^1(S_{i})^2} +
\frac{c_{i+1}^2\sH^1(F_{i+1})^2\norm^o(\nu_{S_{i+1}})}{\sH^1(D\cap S_{i+1})^2\sin\theta_{i+1}}\Big)\,d\sH^ 1,
\end{multline*}
where $\Gamma(t)=\cup_{j=1}^n S_j(t)$ is the polygonal curve parallel to $\Gamma,$ satisfying $\scalarp{H(S_i,S_i(t))}{\nu_{S_i}}=th_i$ for small $|t|$ and  $H_j(S_j,S_j(t))=0$
for $j\ne i.$ 
\end{lemma}

Notice that using \eqref{ashuc7bnn}, we can represent the first variation also as 
\begin{multline*}
\frac{d}{dt}\cF_\alpha(\Gamma(t),D) \Big|_{t=0} = - h_i\Bigg(
\int_{S_i} \tfrac{\norm^o(\nu_{S_{i-1}})}{\sH^1(S_{i})} \Big(1- \tfrac{\alpha c_{i-1}^2 \sH^1(F_{i-1})^2}{\sH^1(D\cap S_{i-1})^2}\Big) dx\\
+
\int_{S_i} \tfrac{\norm^o(\nu_{S_i}) [\cot\theta_i + \cot\theta_{i+1}]}{\sH^1(S_{i})} \Big(1- \tfrac{\alpha c_i^2 \sH^1(F_i)^2}{\sH^1(S_i)^2}\Big)dx
+
\int_{S_i} \tfrac{\norm^o(\nu_{S_{i+1}})}{\sH^1(S_{i})} \Big(1- \tfrac{\alpha c_{i+1}^2 \sH^1(F_{i+1})^2}{\sH^1(D\cap S_{i+1})^2}\Big)dx
\Bigg).
\end{multline*}

\subsection{Gradient flow  associated to the crystalline elastic functional}

Let $\norm$ be a crystalline anisotropy. We study the gradient flow associated to $\cF_\alpha.$ Given an admissible
$\norm$-regular polygonal curve $\Gamma^0$ and $T>0,$ consider a family $\{\Gamma(t)\}_{t\in[0,T)}$ of admissible polygonal curves with
$\Gamma(0)=\Gamma^0$ defined as in the standard crystalline curvature flow: during the evolution 
segments of $\Gamma(t)$ translate along the normal direction with velocity equal to the negative of the first variation of $\cF_\alpha.$ In
particular, $\Gamma(t)$ is expected to be parallel to $\Gamma^0$ for any $t\in[0,T).$ Since the segments $S_i(t)$
translate along the normal, we can uniquely define corresponding signed heights
$h_i(t):=\scalarp{H(S_i^0,S_i(t))}{\nu_{S_i^0}}.$ The derivative of these heights plays the role of the velocity
of the translation, in other words, the normal velocity of $\Gamma(t).$

\begin{definition}[Crystalline elastic flow]\label{def:cryst_elas_flow}
Given an admissible polygonal curve $\Gamma^0:=\cup_{i=1}^n S_i^0$ and $T\in(0,+\infty],$ we say a family
$\{\Gamma(t)\}_{t\in[0,T)}$ is a \emph{crystalline elastic flow} starting from $\Gamma^0$ (shortly
\emph{admissible crystalline elastic flow}) if:
\begin{itemize}
\item $\Gamma(t):=\cup_{i=1}^n S_i(t)$ is a polygonal curve, parallel to $\Gamma^0,$

\item for each segment $S_i,$ the associated signed heights $h_i(t):=\scalarp{H(S_i^0, S_i(t))}{\nu_{S_i^0}}$ belong  to $C^0([0,T))\cap
C^1(0,T)$ with $h_i(0)=0$ and satisfy the system of ODEs
\begin{equation}\label{elastic_ode123}
\frac{h_i'}{\norm^o(\nu_{S_i})} = -\frac{c_i\sH^1(F_i)}{\sH^ 1(S_i)}
- \frac{\alpha}{\sH^1(S_i)} \Big(\frac{c_{i-1}^2\delta_{i-1}}{\sH^1(S_{i-1})^2\sin\theta_i}
+ \frac{c_{i}^2\delta_i [\cot \theta_i+\cot \theta_{i+1}]}{\sH^1(S_{i})^2}
+\frac{c_{i+1}^2 \delta_{i+1}}{\sH^1(S_{i+1})^2\sin\theta_{i+1}}\Big)
\end{equation}
in $(0,T),$ where  $F_j$ is the unique facet of $W^\norm$ with $\nu_{F_j}=\nu_{S_j},$
\begin{equation}\label{def:delta_is}
\delta_j:=\sH^1(F_j)^2 \norm^o(\nu_{S_j}),\quad j=1,\ldots,n,
\end{equation}
and we set $\delta_0:=\delta_n$ and $\delta_{n+1}:=\delta_1.$
\end{itemize}
\end{definition}

Notice that \eqref{elastic_ode123} makes sense also when $\Gamma^0$ is unbounded, since in this case $S_1$ and
$S_n$ are half-lines and by our convention, $c_1=c_n=0,$ and \eqref{elastic_ode123} is taken for $i=2,\ldots,n-1.$ Also, note that by 
parallelness, $\nu_{S_j}$ and
$\theta_j$ in \eqref{elastic_ode123} do not depend on $t.$ Moreover, as $h_i$ is continuous in $[0,T),$ by
\eqref{length_changed_hhh} all lengths $\sH^1(S_i(t))$ are represented linearly in $h$ and thus a posteriori,
from \eqref{elastic_ode123} we conclude in fact $h_i'$ is continuously differentiable in $(0,T),$ and hence, by
bootstrap, $h_i\in C^0([0,T))\cap C^\infty(0,T).$

\begin{lemma}[A priori estimates]\label{lem:apriori_estocada}
Given a polygonal curve $\Gamma^0=\cup_{i=1}^n S_i^0,$ let $\{\Gamma(t)\}_{t\in[0,T]}$ be a crystalline elastic
flow starting from $\Gamma^0.$ Let
$$
\Delta_1:= \frac{1}{2}\, \min_{1\le i\le n} \frac{\sH^1(S_i^0)}{\frac{1}{|\sin \theta_i|} + |\cot\theta_i +
\cot\theta_{i+1}| + \frac{1}{|\sin \theta_{i+1}|}}
$$
and let $T'\le T$ be the smallest positive time (if any) with $\sH^1(S_i(T'))\le \frac{1}{2}\sH^1(S_i^0)$ for some
$i\in\{1,\ldots,n\}.$ Then
$$
\text{either}\quad  T'=T\quad \text{or}\quad T'\ge \frac{\Delta_1}{\Delta_2},
$$
where 
\begin{align*}
\Delta_2:=\max_{1\le i\le n}\norm^o(\nu_{S_i})\Bigg\{ \tfrac{2\sH^1(F_i)}{\sH^ 1(S_i^0)} + \tfrac{2\alpha}{\sH^1(S_i^0)}
\Big(\tfrac{ 4c_{i-1}^2 \delta_{i-1} }{\sH^1(S_{i-1}^0)^2|\sin\theta_i|} +
\tfrac{4c_{i}^2 \delta_i |\cot \theta_i+\cot \theta_{i+1}|}{\sH^1(S_i^0)^2} +
\tfrac{4c_{i+1}^2 \delta_{i+1}  }{\sH^1(S_{i+1}^0)^2|\sin\theta_{i+1}|}\Big)\Bigg\}.
\end{align*}
\end{lemma}

\noindent 
Note that $\Delta_1$ depends ony on the length of the segments and angles of $\Gamma^0,$ while  $\Delta_2$ depends only on $\alpha,$ $\norm,$ the length of the segments and angles of $\Gamma^0.$

\begin{proof}
Assume that $T'<T.$ Since $\sH^1(S_i(t))\ge \frac{1}{2}\sH^1(S_i^0)$ for $t\in [0,T']$ for all $i=1,\ldots,n,$ by
\eqref{elastic_ode123},
$|h_i'| \le \Delta_2$ in $(0,T').$ Hence $|h_i(t)|\le \Delta_2t$ for any
$t\in(0,T').$ Moreover, by the definition of $T'$ there exists $i_o$ such that
$\sH^1(S_{i_o}(T'))=\frac{1}{2}\sH^1(S_{i_o}^0).$ Then
by \eqref{length_changed_hhh} and the inequality $|h_i(T')|\le \Delta_2T'$ we have 
$$
\frac12\sH^1(S_{i_o}^0) = \sH^1(S_{i_o}(T')) \ge \sH^1(S_{i_o}^0) - 
\Big(\tfrac{1}{|\sin \theta_{i_o}|} + |\cot\theta_{i_o} + \cot\theta_{{i_o}+1}| + \tfrac{1}{|\sin \theta_{i_o+1}|}\Big) \Delta_2T'.
$$
This and the definition of $\Delta_1$ imply $T'\ge \frac{\Delta_1}{\Delta_2}.$ 
\end{proof}

\subsection{Gradient flow structure}

In this section we show that, as expected, our definition of crystalline elastic flow implies that the associated curves
decrease their crystalline elastic energy along the flow. Since we may have unbounded half-lines,
we need to localize the corresponding energy.

\begin{theorem}[Gradient flow]\label{teo:cl_elastc_gradflow}
Let $\norm$ be a crystalline anisotropy, $\{\Gamma(t)\}_{t\in[0,T)}$ be a (bounded or unbounded)  admissible crystalline elastic flow
for some $T\in(0,+\infty].$ Then for any $t\in[0,T)$ and any disc compactly containing all segments $S_i(t)$ of $\Gamma(t)$ one has
\begin{equation}\label{grad_floqasdf}
\frac{d}{dt}\cF_\alpha(\Gamma(t), D) = - \sum_{i=1}^n \frac{|h_i'(t)|^2\cH^1(D\cap S_i(t))}{\norm^o(\nu_{S_i(t)})},
\end{equation}
where $h_i\equiv0$ if $S_i$ is a half-line. In particular, the map $t\in[0,T) \mapsto \cF_\alpha(\Gamma(t),D)$ is 
nonincreasing. Finally, if $\Gamma(0)$ is bounded, one can take $D=\R^2$ in \eqref{grad_floqasdf}.
\end{theorem}

\begin{proof}
For shortness, let us drop the dependence on $t.$ 
First assume that $\Gamma^0:=\Gamma(0)$ is closed and $D=\R^2.$ By 
parallelness,
$\nu_{S_i}=\nu_{S_i^0},$ and hence  by \eqref{curva_definiton},
\begin{align*}
\cF_\alpha(\Gamma) = & \sum_{i=1}^n \int_{S_i} \left(1+ \alpha \left(
\kappa_{S_i}^\norm\right)^2\right)\norm^o(\nu_{S_i})\,d\sH^1 =
\sum_{i=1}^n \norm^o(\nu_{S_i^0}) \Big(\sH^1(S_i) + \tfrac{\alpha c_i^2 \sH^1(F_i)^2}{\sH^1(S_i)} \Big),
\end{align*}
where $c_i\in\{-1,0,1\}$ is the transition number of $S_i$ and $F_i$ is the unique facet of $W^\norm$ with
$\nu_{S_i}=\nu_{F_i}.$
Moreover, by \eqref{length_changed_hhh} 
\begin{equation}\label{ashdvbnm}
\sH^1(S_i) = \sH^1(S_i^0) - \Big( \frac{h_{i-1}}{\sin\theta_i} + h_i[\cot\theta_i+\cot\theta_{i+1}] +
\frac{h_{i+1}}{\sin\theta_{i+1}}\Big),
\end{equation}
where $\theta_i$ are the angles of $\Gamma^0.$ Thus, 
\begin{equation}\label{ahstszgbbb}
\frac{d}{dt} \cF_\alpha(\Gamma) = -\sum_{i=1}^n \norm^o(\nu_{S_i^o})\Big(\tfrac{h_{i-1}'}{\sin\theta_i} +
h_i'[\cot\theta_i+\cot\theta_{i+1}] + \tfrac{h_{i+1}'}{\sin\theta_{i+1}}\Big)\Big(1-\tfrac{\alpha
c_i^2\sH^1(F_i)^2}{\sH^1(S_i)^2}\Big)
\end{equation}
By our conventions on indexation, after relabelling the indices of the sums we can write
\begin{equation*}
\sum_{i=1}^n \norm^o(\nu_{S_i^o})\tfrac{h_{i-1}'}{\sin\theta_i}
\Big(1-\tfrac{\alpha c_i^2\sH^1(F_i)^2}{\sH^1(S_i)^2}\Big) =
\sum_{i=1}^n \norm^o(\nu_{S_{i+1}^o})\tfrac{h_i'}{\sin\theta_{i+1}}
\Big(1-\tfrac{\alpha c_{i+1}^2\sH^1(F_{i+1})^2}{\sH^1(S_{i+1})^2}\Big)
\end{equation*}
and 
\begin{equation*}
\sum_{i=1}^n \norm^o(\nu_{S_i^o})\tfrac{h_{i+1}'}{\sin\theta_{i+1}}
\Big(1-\tfrac{\alpha c_i^2\sH^1(F_i)^2}{\sH^1(S_i)^2}\Big) =
\sum_{i=1}^n \norm^o(\nu_{S_{i-1}^o})\tfrac{h_i'}{\sin\theta_{i}}
\Big(1-\tfrac{\alpha c_{i-1}^2\sH^1(F_{i-1})^2}{\sH^1(S_{i-1})^2}\Big).
\end{equation*}
We can also represent \eqref{ahstszgbbb} as 
\begin{multline*}
\frac{d}{dt} \cF_\alpha(\Gamma) = - \sum_{i=1}^n h_i' \Big( \tfrac{\norm^o(\nu_{F_{i-1}})}{\sin\theta_i} +
\norm^o(\nu_{F_i})[\cot\theta_i+\cot\theta_{i+1}] + \tfrac{\norm^o(\nu_{F_{i+1}})}{\sin\theta_{i+1}} \\
+ \alpha \Big( \tfrac{c_{i-1}^2\sH^1(F_{i-1})^2 \norm^o(\nu_{S_{i-1}})}{\sH^1(S_{i-1})^2\sin\theta_i} 
+ \tfrac{c_{i}^2\sH^1(F_{i})^2\norm^o(\nu_{S_{i}})[\cot \theta_i+\cot \theta_{i+1}]}{\sH^1(S_{i})^2} 
+ \tfrac{c_{i+1}^2\sH^1(F_{i+1})^2\norm^o(\nu_{S_{i+1}})}{\sH^1(S_{i+1})^2\sin\theta_{i+1}}\Big)\Big).
\end{multline*}
Thus, recalling the identity \eqref{ashuc7bnn} and the ODE \eqref{elastic_ode123}, the last
equality reads as
\begin{equation*}
\frac{d}{dt} \cF_\alpha(\Gamma) = - \sum_{i=1}^n \tfrac{ |h_i'|^2\sH^1(S_i) }{\norm^o(\nu_{S_i})} 
\end{equation*}
and \eqref{grad_floqasdf} follows.

Now assume that $\Gamma(0)$ is unbounded and $D$ is a disc compactly containing all segments of $\Gamma(t).$ In this case $S_1$ and $S_n$ are half-lines and hence  
\begin{align*}
\cF_\alpha(\Gamma,D) = & \norm^o(S_1^0) \sH^1(D\cap S_1)+ \sum_{i=2}^{n-1} \norm^o(\nu_{S_i^0}) \Big(\sH^1(S_i) +
\tfrac{\alpha c_i^2 \sH^1(F_i)^2}{\sH^1(S_i)} \Big) + \norm^o(S_n^0) \sH^1(D\cap S_n).
\end{align*}
In this case, for $i=1$ and $i=n$ the equality \eqref{ashdvbnm}  is represented as 
$$
\sH^1(D\cap S_1) = \sH^1(D\cap S_1^0) - \frac{h_2}{\sin\theta_2},\quad 
\sH^1(D\cap S_n) = \sH^1(D\cap S_n^0) - \frac{h_{n-1}}{\sin\theta_n}.
$$
Thus, 
\begin{multline}\label{dawq12}
\frac{d}{dt} \cF_\alpha(\Gamma,D) = -\tfrac{h_2'\norm^o(\nu_{S_2^o})}{\sin\theta_2} \\
-\sum_{i=2}^{n-1}
\norm^o(\nu_{S_i^o})\Big(\tfrac{h_{i-1}'}{\sin\theta_i} + h_i'[\cot\theta_i+\cot\theta_{i+1}] +
\tfrac{h_{i+1}'}{\sin\theta_{i+1}}\Big)\Big(1-\tfrac{\alpha c_i^2\sH^1(F_i)^2}{\sH^1(S_i)^2}\Big) -
\tfrac{h_{n-1}'\norm^o(\nu_{S_{n-1}^o})}{\sin\theta_n}.
\end{multline}
As $h_1,h_n\equiv0,$ 
\begin{equation*}
\sum_{i=2}^{n-1} \norm^o(\nu_{S_i^o})\tfrac{h_{i-1}'}{\sin\theta_i}
\Big(1-\tfrac{\alpha c_i^2\sH^1(F_i)^2}{\sH^1(S_i)^2}\Big) =
\sum_{i=2}^{n-2} \norm^o(\nu_{S_{i+1}^o})\tfrac{h_i'}{\sin\theta_{i+1}}
\Big(1-\tfrac{\alpha c_{i+1}^2\sH^1(F_{i+1})^2}{\sH^1(S_{i+1})^2}\Big)
\end{equation*}
and 
\begin{equation*}
\sum_{i=2}^{n-1} \norm^o(\nu_{S_i^o})\tfrac{h_{i+1}'}{\sin\theta_{i+1}}
\Big(1-\tfrac{\alpha c_i^2\sH^1(F_i)^2}{\sH^1(S_i)^2}\Big) =
\sum_{i=3}^{n-1} \norm^o(\nu_{S_{i-1}^o})\tfrac{h_i'}{\sin\theta_{i}}
\Big(1-\tfrac{\alpha c_{i-1}^2\sH^1(F_{i-1})^2}{\sH^1(S_{i-1})^2}\Big),
\end{equation*}
and hence, as above, \eqref{dawq12} reads as 
$$
\frac{d}{dt} \cF_\alpha(\Gamma,D) = - \sum_{i=2}^{n-1} \tfrac{|h_i'|^2 \sH^1(S_i)}{\norm^o(\nu_{S_i})}
$$
and \eqref{grad_floqasdf} follows.
\end{proof}

Integrating \eqref{grad_floqasdf} we deduce the following energy dissipation equality
\begin{equation}\label{ZgEW}
\cF_\alpha(\Gamma^0,D) =\cF_\alpha(\Gamma(t),D) + \int_0^t \tfrac{|h'(s)|^2\sH^1(S_i(s))}{\norm^o(\nu_{S_i(s)})}ds,
\end{equation}
which implies some upper and lower bounds for the length of segments.

\begin{corollary}\label{cor:segments_nonzero_curva}
Let $\norm$ be a crystalline anisotropy, $\Gamma^0:=\cup_{i=1}^nS_i^0$ be an admissible polygonal curve and
$\{\Gamma(t)\}_{t\in[0,T)}$ be a crystalline elastic flow starting from $\Gamma^0$ for some $T\in(0,+\infty].$
Then for any $T'\in (0,T)$ and $t\in [0,T']$
\begin{equation}\label{length_bnd12}
\frac{1}{c_\norm}\,\cF_\alpha(\Gamma^0,D) \ge  \sum_{i=1}^n \sH^1(D\cap S_i(t)) 
\end{equation}
and 
\begin{equation}\label{length_bnd13}
\sH^1(D\cap S_i(t)) \ge  \frac{\alpha c_\norm c_i^2\sH^1(F_i)^2}{\cF_\alpha(\Gamma^0,D)}, \quad i=1,\ldots,n,
\end{equation}
where $c_\norm$ is given by \eqref{norm_bounds_12120} and $D$ is any disc compactly containing all bounded
segments of $\Gamma(t)$ for all $t\in[0,T'],$ which can be taken $\R^2$ if $\Gamma^0$ is bounded.
\end{corollary}

The estimate \eqref{length_bnd12} follows from the equality \eqref{ZgEW} using the anisotropic length   of $S_i,$ while \eqref{length_bnd13} follows from the $\norm$-curvature part of $\cF_\alpha$ in  \eqref{ZgEW}.

\section{Existence, uniqueness and  behaviour at the maximal time}\label{sec:existence}

In this section we prove the following result.

\begin{theorem}[Crystalline elastic flow]\label{teo:existence}
Let $\norm$ be a crystalline anisotropy and $\Gamma^0:=\cup_{i=1}^n S_i^0$ be an admissible polygonal curve.
Then:
\begin{itemize}
\item[(a)] there exist a maximal time $T^\dag\in(0,+\infty]$ and a unique crystalline elastic flow
$\{\Gamma(t)\}_{t\in[0,T^\dag)}$ starting from $\Gamma^0;$

\item[(b)] $t\mapsto \Gamma(t)$ has the semigroup property, i.e., for any $t,s\ge0$ with $t+s<T^\dag,$
\begin{equation}\label{semigruops}
\Gamma(t+s;\Gamma^0) = \Gamma(t,\Gamma(s;\Gamma^0)),
\end{equation}
where $\Gamma(\cdot; \Sigma)$ is the crystalline elastic flow starting from $\Sigma;$ 

\item[(c)] if $\Gamma^0$ is bounded and $T^\dag<+\infty,$ then $t\mapsto \Gamma(t)$ is Kuratowski continuous in
$[0,T^\dag],$ the set
$$
\bigcup_{t\in[0,T^\dag]}\,\,\, \Gamma(t)
$$
is bounded and there exists an index $i\in\{1,\ldots,n\}$ with $c_i=0$ such that 
\begin{equation}\label{segment_disappear1}
\lim\limits_{t\nearrow T^\dag} \sH^1(S_i(t)) = 0.
\end{equation}
Thus, at the maximal time, at least one segment with zero $\norm$-curvature disappears. Moreover, the length of all  segments with nonzero curvature is bounded away from $0$ as $t\nearrow T^\dag.$
\end{itemize}
\end{theorem}

We expect that assertion (c) is valid also in case $\Gamma^0$ is unbounded; however, we do not study it here and leave it as a future work (see also Section \ref{sec:open_problems}). 

\begin{proof}
(a)-(b). We provide the full proof only in the case of closed curves; the assertions for unbounded curves 
can be done along the same lines. Let $\{\theta_i\}_{i=1}^n$ be the angles of $\Gamma^0$ and
let $\Delta_1$ be given by Lemma  \ref{lem:apriori_estocada}.
\smallskip

{\it Step 1: short-time existence and uniqueness.}
For $T>0,$ let 
\begin{equation}\label{zolotaya_toza1029}
\cK_T:=\{h:=(h_1,\ldots,h_n):\,\, h_i\in C[0,T],\,\, h_i(0)=0,\,\,\|h\|_\infty\le \Delta_1\}
\end{equation}
be the closed convex subset of the Banach space $(C[0,T])^n$ with the standard $L^\infty$-norm 
$$
\|h\|_\infty = \max_{1\le i\le n}\,\|h_i\|_\infty.
$$
For $h\in \cK_T$ set
$$
L_i[h](t):=\sH^1(S_i^0) - \Big( \frac{h_{i-1}(t)}{\sin\theta_i} + h_i(t)[\cot\theta_i + \cot\theta_{i+1}] +
\frac{h_{i+1}(t)}{\sin\theta_{i+1}}\Big),\quad t\in[0,T],
$$
and consider the operator
$
\cA =(A_1 ,\ldots,A_n )$ defined in $\cK_T$  as 
\begin{equation*}
A_i[h](t):=  \int_0^t \Big[-\tfrac{c_i\sH^1(F_i)}{L_i[h](s)} \\
- \tfrac{\alpha}{L_i[h](s)} \Big(\tfrac{ c_{i-1}^2\delta_{i-1} }{(L_{i-1}[h](s))^2
\sin\theta_i}+ \tfrac{c_{i}^2 \delta_i [\cot \theta_i+\cot
\theta_{i+1}]}{(L_i[h](s))^2}+\tfrac{c_{i+1}^2 \delta_{i+1} }{(L_{i+1}[h](s))^2
\sin\theta_{i+1}}\Big)\Big]\,\norm^o(\nu_{S_i^0})ds.
\end{equation*}
By the definition of $\cK_T,$ we have 
$$
L_i[h]\ge \frac{1}{2}\sH^1(S_i^0),\quad i=1,\ldots,n,\quad h\in \cK_T,
$$
and thus, 
$$
\|A_i[h]\|_\infty \le \Delta_2T,
$$
where $\Delta_2>0$ is given by Lemma \ref{lem:apriori_estocada} and depends only on $\alpha,$ $\norm,$ the angles
$\theta_i$ and reciprocals of the lengths of the segments of $\Gamma^0.$
Moreover, as $L_i[h]$ is linear in $h,$ we have also 
$$
\|A_i[h']-A_i[h'']\|_\infty \le \Delta_3 \|h'-h''\|_\infty T,
$$
where again $\Delta_3>0$ depends only on $\alpha,$ $W^\norm,$ the angles $\theta_i$ and reciprocals of the
lengths of the segments of $\Gamma^0.$ Thus, if we choose
$$
T := \min\Big\{\frac{\Delta_1}{\Delta_2},\frac{1}{2\Delta_3}\Big\},
$$
then $\cA:\cK_T\to \cK_T$ is a contraction. Theferfore it has a unique fixed point $h$ in $\cK_T.$ 
Using the equation $\cA[h]=h$ and a bootstrap argument, we immediately deduce that $h\in C^\infty[0,T].$
Moreover, as $h$ satisfies $\|h\|_\infty\le \Delta,$ by Lemma \ref{lem:construc_paral_curve} for any $t\in[0,T]$
there exists a unique polygonal closed curve $\Gamma(t)$ parallel to $\Gamma^0,$ such that
$h_i(t)=\scalarp{H(S_i^0,S_i(t))}{\nu_{S_i^0}}.$ By the definition of $\cA,$ the heights $h_i$ solve 
\eqref{elastic_ode123}. Hence, $\{\Gamma(t)\}_{t\in[0,T]}$ is a crystalline elastic flow starting from
$\Gamma^0.$
\smallskip

{\it Step 2: uniqueness.} Let us show that there is at most one crystalline elastic flow starting from an
admissible polygonal curve $\Gamma^0.$ Indeed, by contradiction, suppose that for some $T>0$ there are two flows
$\{\Gamma(t)\}_{t\in[0,T]}$ and $\{\Sigma(t)\}_{t\in[0,T]}$ starting from $\Gamma^0.$ Let $T'\in[0,T]$ be the
largest time for which $\Gamma(t)=\Sigma(t)$ for all $t\in[0,T']$ and
let $\{h_i\}$ and $\{b_i\}$ be the associated heights. Clearly, $h_i=b_i$ in $[0,T']$ for all $i=1,\ldots,n.$
Moreover, if $T'<T,$ by smoothness, there exists $\epsilon>0$ such that both
$h:=(h_1(T'+\cdot),\ldots,h_n(T'+\cdot))$ and $b:=(b_1(T'+\cdot),\ldots,b_n(T'+\cdot))$ belong  to
$\cK_\epsilon,$ see \eqref{zolotaya_toza1029} for definition. Thus, if we choose $\epsilon>0$ small enough, by
step 1, the operator $\cA$ will have a unique fixed point so that $h=b$ in $[0,\epsilon].$ This implies $h_i=b_i$
in $[0,T'+\epsilon]$ for all $i=1,\ldots,n,$ which contradicts the maximality of $T'.$
\smallskip

{\it Step 3: maximal existence time.} Let 
$$
T^\dag:=\sup\Big\{T>0:\,\,\text{there is a crystalline elastic flow $\{\Gamma(t)\}_{t\in[0,T]}$ starting from
$\Gamma^0$}\Big\}.
$$
By step 1, $T^\dag\ge \max\{\frac{\Delta_1}{\Delta_2},\frac{1}{2\Delta_3}\}.$ Moreover, if
$\{\Gamma'(t)\}_{t\in[0,T']}$ and $\{\Gamma''(t)\}_{t\in[0,T'']}$ are two flows starting from $\Gamma^0,$ then by
step 2, $\Gamma'(t)=\Gamma''(t)$ for any $0\le t\le \min\{T',T''\}.$ Therefore, there exists a
unique crystalline elastic flow $\{\Gamma(t)\}_{t\in[0,T^\dag)}$ starting from $\Gamma^0.$
\smallskip

{\it Step 4: semigroup property.} Since each $\Gamma(t)$ is parallel to $\Gamma^0,$ for any $t,s\ge0 $ with
$t+s\in[0,T^\dag)$ we have
$$
H(S_i^0,S_i(t+s)) = H(S_i^0,S_i(t)) +  H(S_i(t),S_i(t+s)),\quad i=1,\ldots,n. 
$$
This implies \eqref{semigruops}. 
\smallskip

(c) Now we study the behaviour of the flow at the maximal time  $T^\dag$ assuming that $\Gamma^0$ is bounded and
$T^\dag<+\infty.$  First suppose that there exists $\epsilon_0>0$ such that 
\begin{equation}\label{y7dufhh}
\min_{i=1,\ldots,n}\,\, \inf_{t\in[0,T^\dag)} \,\sH^1(S_i(t)) \ge \epsilon_0,
\end{equation}
and for sufficiently small $\epsilon\in(0,\epsilon_0),$ let the constants $\Delta_1^{\epsilon},$ $\Delta_2^{\epsilon}$ and $\Delta_3^{\epsilon}$ be given as in step 1 with $\Gamma(T^\dag-\epsilon)$ in place of $\Gamma^0.$ By definition, these numbers are in fact independent of $\epsilon,$ rather they depend  only on $\epsilon_0,$ $\theta_i$ and $W^\norm.$ Thus, $T_0:=\min\{\frac{\Delta_1^\epsilon}{\Delta_3^\epsilon},\frac{1}{\Delta_3^\epsilon}\}$ is independent of $\epsilon.$ In particular, we may choose $0<\epsilon<T_0/3.$ Then applying step 1 we can construct the crystalline elastic flow $\{\Gamma(t)\}_{t\in (T^\dag-\epsilon,T^\dag-\epsilon+T_0)}$ starting from
$\Gamma(T^\dag-\epsilon).$ Again by uniqueness, $\Gamma(\cdot)$ can be defined until $T^\dag-\epsilon+T_0>T^\dag,$
which contradicts the maximality of $T^\dag.$ 
This contradiction yields that there exists $i\in\{1,\ldots,n\}$ such that
\begin{equation}\label{p0of5}
\inf_{t\in[0,T^\dag)} \,\sH^1(S_i(t)) =0.
\end{equation}

Note that \eqref{p0of5} implies  $\liminf_{t\nearrow T^\dag} \sH^1(S_i(t))=0.$ Now we prove  $\limsup_{t\to T^\dag} \sH^1(S_i(t))=0,$ thereby proving \eqref{segment_disappear1}.
We first observe that by \eqref{length_bnd12} and \eqref{length_bnd13},
\begin{equation}\label{caputions}
+\infty>\frac{1}{c_\norm} \cF_\alpha(\Gamma^0) \ge \sum_{i=1}^n \sH^1(S_i(t)) = \sH^1(\Gamma(t))
\quad
\text{and} 
\quad
\sH^1(S_i(t)) \ge \frac{\alpha c_\norm c_i^2 \sH^1(F_i)^2}{\cF_\alpha(\Gamma^0)},\quad i=1,\ldots,n,
\end{equation}
for any $t\in[0,T^\dag).$ Let us investigate the behaviour of the signed heights $\{h_i\}$ from $\Gamma^0,$ solving the system of ODEs \eqref{elastic_ode123}. We need to rule out the behaviour of $h_i$ near $T^\dag$ depending on whether $c_i\ne 0$ or $c_i=0.$
\smallskip

{\it Case 1:} fix any $i\in\{1,\ldots,n\}$ with $c_i\ne0.$ By the evolution equation \eqref{elastic_ode123}, 
$$
|h_i'| \le \frac{C|c_i|}{\sH^1(S_i)} + \alpha \sum_{j\in\{i-1,i,i+1\}} \frac{C\,|c_j|}{\sH^1(S_i)\sH^1(S_j)^2}\quad\text{in $(0,T^\dag)$}
$$
for some constant $C:=C_{\norm,\theta}>0,$ where $\theta:=\{\theta_i\}$ are the set of angles. 
Now using the lower bound in \eqref{caputions} we get
\begin{equation}\label{qtwchghj}
|h_i'| \le \tilde C\quad \text{in $(0,T^\dag)$}
\end{equation}
for some  $\tilde C:=\tilde C_{\norm,\theta,\alpha,\cF_\alpha(\Gamma^0)}>0.$ Thus, $h_i\in C[0,T^\dag]\cap W^{1,+\infty}(0,T^\dag)\cap C^\infty(0,T^\dag).$

Note that the $L^\infty$-bound \eqref{qtwchghj} of $h_i'$ implies 
\begin{equation*}
|h_i(0)| - \tilde C T^\dag \le |h_i(T^\dag)| \le |h_i(0)| + \tilde C T^\dag.
\end{equation*}
Thus, $X:=\cup_{t\in[0,T^\dag)}\,S_i(t)$ is bounded. Moreover, by the first estimate of \eqref{caputions}, for $r:=\frac{2}{c_\norm}\cF_\alpha(\Gamma^0)$ we have
$$
\Gamma(t) \subset \bigcup_{x\in S_i(t)} B_{\sH^1(\Gamma(t))}(x) \subseteq \bigcup_{x\in X} B_r(x),\quad t\in [0,T^\dag).
$$
Thus, the union $\cup_{t\in[0,T^\dag)}\Gamma(t)$ is contained in the $\frac{2}{c_\norm}\cF_\alpha(\Gamma^0)$-neighborhood of $X,$ i.e., is bounded.
\smallskip

{\it Case 2:} consider any segment $S_i$ with $c_i=0.$ By the admissibility of $\Gamma(\cdot),$ the segments $S_{i-1}$ and $S_{i+1}$ are parallel and $\nu_{S_{i-1}} = \nu_{S_{i+1}}.$ In this case \eqref{elastic_ode123} takes form 
\begin{equation}\label{io0t5a}
h_i' = -\frac{\alpha \norm^o(\nu_{S_i^0})}{\sH^1(S_i)}\Big(\frac{c_{i-1}^2\delta_{i-1}}{\sH^1(S_{i-1})\sin\theta_i} +
\frac{c_{i+1}^2\delta_{i+1}}{\sH^1(S_{i+1})\sin\theta_{i+1}} \Big)\quad \text{in $[0,T^\dag).$}
\end{equation}
We distinguish the following cases.
\smallskip

{\it Case 2.1:} $c_{i-1} = c_{i+1}=0.$ In this case, $h_i\equiv0$ in $[0,T^\dag].$
\smallskip

{\it Case 2.2:} $c_{i-1} = 0 \ne c_{i+1}$ or $c_{i-1}\ne 0 = c_{i+1}.$ In this case, recalling the lower bound in \eqref{caputions} and the evolution equation \eqref{io0t5a}, we find 
$$
\frac{1}{\sH^1(S_i)} \le C''|h_i'|\quad \text{in $(0,T^\dag)$}
$$
for some constant $C'':= C_{\norm,\theta,\alpha,\cF_\alpha(\Gamma^0)}''>0.$ Thus, by the Cauchy inequality,
$$
|h_i'| \le  \frac{C''}{2}|h_i'|^2\sH^1(S_i) + \frac{1}{2C''\sH^1(S_i)} \le  \frac{C''}{2}|h_i'|^2\sH^1(S_i) + \frac{|h_i'|}{2}\quad\text{in $(0,T^\dag).$}
$$
Then for any $0<s<t<T^\dag,$
$$
|h_i(s) - h_i(t)| \le \int_s^t|h_i'|d\sigma \le C'' \int_s^t |h_i'|^2\sH^1(S_i)d\sigma.
$$
In view of \eqref{grad_floqasdf}, $|h_i'|^2\sH^1(S_i)\in L^1(0,T^\dag)$ and therefore, $h_i$ is bounded and absolutely continuous in $[0,T^\dag].$ 
\smallskip

{\it Case 2.3:} $c_{i-1},c_{i+1}\ne0.$  
As we observed in step 1,  $h_{i-1},h_{i+1} \in C[0,T^\dag].$ Consider the difference 
$$
H(S_{i-1},S_{i+1}) = h_{i+1} - h_{i-1} - H(S_{i-1}^0,S_{i+1}^0),
$$
where $H(S,T)$ is the distance vector from $S$ to $T.$ Clearly, $H(S_{i-1},S_{i+1}) \in C[0,T^\dag].$

\begin{itemize}
\item {\it Subcase 2.3.1:} if $|H(S_{i-1}(T^\dag),S_{i+1}(T^\dag))|>0,$ then by continuity and the regularity of the evolution, $|H(S_{i-1},S_{i+1})|\ge\epsilon>0$ in $(0,T^\dag).$  This means $\sH^1(S_i(t))\ge a_\epsilon > 0$ in $(0,T^\dag)$ for some $a_\epsilon>0$ and hence, as in step 1, from \eqref{io0t5a} it follows that $h_i\in C[0,T^\dag].$

\item {\it Subcase 2.3.2:} if $|H(S_{i-1}(T^\dag),S_{i+1}(T^\dag))|=0,$ then $|H(S_{i-1}(t),S_{i+1}(t))|\to0$ and $\sH^1(t)\to0$ as $t\nearrow T^\dag.$  Thus, every Kuratowski-limit of segments $S_{i-1}(t)$ and $S_{i+1}(t)$ is contained in a single straight line $L$.  In particular, $S_i(t)$ K-converges to a closed connected subset of $L;$ here possible oscillations of $h_i$ near $T^\dag$ may prevent the limit to be a singleton. 
\end{itemize}

These observations imply \eqref{segment_disappear1}. Indeed, if subcase 2.3.2 holds with some segment $S_i$ with $c_i=0$ and $c_{i-1},c_{i+1}\ne0$ at the maximal time $T^\dag$, then $\sH^1(S_i(t))\to0$ as $t\to T^\dag.$ On the other hand, if any segment $S_i$ with $c_i=0$ and $c_{i-1},c_{i+1}\ne0$ satisfy  subcase 2.3.1, then all $h_j$ are continuous up to $T^\dag.$ Moreover, by \eqref{length_changed_hhh}
$$
\sH^1(S_j) = \sH^1(S_j^0) - \Big(\frac{h_{j-1}}{\sin\theta_j} + h_j[\cot \theta_j + \cot\theta_{j+1}] + \frac{h_{j+1}}{\sin\theta_{j+1}}\Big)\in C[0,T^\dag].
$$
Now if $\min_j \sH^1(S_j(T^\dag))>0,$ then by continuity and regularity of the evolution, we would get  \eqref{y7dufhh} for some $\epsilon>0,$ which again would contradict the maximality of $T^\dag.$ This contradiction implies $\sH^1(S_{\bar j}(t))\to0$ as $t\nearrow T^\dag$ for some $\bar j.$ Since by \eqref{caputions} $\sH^1(S_i(T^\dag))>0$ for any segment $S_i$ with $c_i\ne0,$ necessarily $c_{\bar j}=0$ and \eqref{segment_disappear1} follows.

Finally, we prove that $t\mapsto \Gamma(t)$ is Kuratowski-continuous in $[0,T^\dag].$ Its continuity in $[0,T^\dag)$ follows from the regularity of the evolution. Let us show that as $t\nearrow T^\dag$ the curves $\Gamma(t)$ K-converge to a unique admissible curve $\Gamma^*,$ which we denote by $\Gamma(T^\dag).$ Indeed, let us write  $S_i(t)=[G_i(t),G_{i+1}(t)]$ for any $i$ and study the limit points of the vertices $G_i(t)$ as $t\nearrow T^\dag.$  

\begin{itemize}
\item[(a)] Let $i$ be such that both $h_i,h_{i-1}\in C[0,T^\dag].$ Then, as $t\nearrow T^\dag,$ the straight lines $L_i(t)$ and $L_{i-1}(t)$ containing $S_i(t)$ and $S_{i-1}(t),$ respectively, Kuratowski converge to the intersecting straight lines $L_i^*$ and $L_{i-1}^*.$ In this case, $G_i(t),$ which is the intersection of $L_i(t)$ and $L_{i-1}(t),$ converges to the unique intersection point $G_i^*$ of $L_i^*$ and $L_{i-1}^*.$

\item[(b)] Let $i$ be such that either $i-1$ or $i$ is as in subcase 2.3.2, where we could not claim the continuity of corresponding signed height at $T^\dag.$ If, for instance,
$h_{i-1}$ is oscillating, i.e., $c_{i-2},c_{i}\ne 0=c_{i-1},$ and $|H(S_{i-2},S_{i})|\to0$ and $\sH^1(S_{i-1})\to0$ as $t\nearrow T^\dag,$ then $G_i(t)$ Kuratowski converges to a closed connected subset $X_i$, contained in all Kuratowski limits of $S_{i-2}$ and $S_i,$ which lie on the same straight line, denoted as $L_i^*=L_{i-2}^*.$  The case of oscillating $h_i$ is analogous. In either situation we do not define $G_i^*.$
\end{itemize}

Now let $\{i_1,\ldots,i_m\}$ be all vertices $i$ at which $G_{i}^*$ is defined and consider the closed polygonal curve $\Gamma^*:=\cl{G_{i_1}^*G_{i_2}^*\ldots G_{i_m}^*G_{i_1}^*}.$ In view of (a), $[G_i^*G_{i+1}^*]$ can be a degenerate segment, and if the vertex $i$ of $\Gamma^0$ lies between the vertices $i_j$ and $i_{j+1},$ then by (b), the Kuratowski limit $X_i$ of $G_i(t)$ satisfies 
\begin{equation}\label{us78vbnn}
X_i\subset [G_{i_j}^*G_{i_{j+1}}^*].
\end{equation}
Let us check that 
\begin{equation}\label{tzasdfv4}
\Gamma^* = K \text{-} \lim_{t \nearrow T^\dag} \Gamma(t).
\end{equation}
Indeed, let $t_k\nearrow T^\dag$ and $x_k\in \Gamma(t_k)$ be such that $x_k\to x\in\R^2.$  Possibly passing to a subsequence, we may assume that there exists $i$ such that $x_k\in S_i(t_k)=[G_i(t_k),G_{i+1}(t_k)]$ for all $k.$ 
\begin{itemize}
\item If $i=i_j$ and $i+1=i_{j+1}$ for some $j,$ then $S_{i}(t_k) \overset{K}{\to} [G_{i_j}^*G_{i_{j+1}}^*]$ and hence, $x\in \Gamma^*.$ 

\item If both $G_i^*$ and $G_{i+1}^*$ are not defined, then by (b), both $X_i$ and $X_{i+1}$ belong to the same straight line $L_i^*=L_{i+1}^*.$ Thus, $S_i(t_k)$ Kuratowski converges to a subset of the  closed convex hull of $X_i$ and $X_{i+1},$ which is a segment of $L_i^*.$ Then  \eqref{us78vbnn} implies $x\in \Gamma^*.$ 

\item If $G_i^*$ is not defined and $i+1=i_j$ for some $j,$ then $S_i(t_k)$ Kuratowski converges to a subset of the  closed convex hull of $X_i$ and $G_{i_j}^*,$ which lies on the same straight line containing $X_i$ and $G_{i_j}^*.$ By \eqref{us78vbnn}, again $x\in\Gamma^*.$ 

\item Finally, if $G_{i+1}^*$ is not defined and $i=i_j$ for some $j,$ the conclusion $x\in \Gamma^*$ is as above. 
\end{itemize}
These observations yield that $\Gamma^*$ contains the Kuratowski upper limit of $\Gamma(t)$ as $t\nearrow T^\dag.$ To show that $\Gamma^*$ is the Kuratowski lower limit of $\Gamma(\cdot),$ we consider an arbitrary sequence $t_k\nearrow T^\dag$ and fix any $x\in \Gamma^*.$ We claim that there exists $x_k\in \Gamma(t_k)$ such that $x_k\to x.$ Indeed, if $x=G_{i_j}^*$ for some $j,$ then by (a), $G_{i_j}(t_k)\to x=G_{i_j}^*.$
Thus, we may assume that $x$ belongs to the relative interior of some segment $[\Gamma_{i_j}^*G_{i_{j+1}}^*]$ of $\Gamma^*.$ We observe that the union $\cup_{i=i_j}^{i_{j+1}-1}S_i(t_k)$ Kuratowski converges to $[\Gamma_{i_j}^*G_{i_{j+1}}^*].$ Indeed, $G_i^* $ is not defined for any $i_j<i<i_{j+1}-1$ and any $S_i(t_k)$ is either parallel to $[\Gamma_{i_j}^*G_{i_{j+1}}^*]$ or its length converges to $0.$ Thus, recalling \eqref{us78vbnn}, we deduce the required convergence. This convergence implies that $\Gamma^*$ is the Kuratowski lower limit and \eqref{tzasdfv4} follows. 

Finally, we claim that $\Gamma^*$ is admissible. Indeed, let $\nu_j$ be the unit normal of $[G_{i_j}^*G_{i_{j+1}}^*].$ By definition and \eqref{us78vbnn}, for any index $i_j<i<i_{j+1}-1,$ the segment $S_i$ is either parallel to $[G_{i_j}^*G_{i_{j+1}}^*] $ so that $\nu_{S_i^0}=\nu_j$ or has $c_i=0$ and its length converges to $0$ (in particular, neighboring segments have unit normal $\nu_j$). Thus, recalling the admissibility of $\Gamma(\cdot),$ we conclude that the normals $\nu_{j-1}$ and $\nu_j$ of two consecutive segments $[G_{i_{j-1}}^*G_{i_j}^*]$ and $[G_{i_{j}}^*G_{i_{j+1}}^*]$ must be adjacent outer normals to  $W^\norm.$ This implies $\Gamma^*$ is admissible. Clearly, $\Gamma^*$ is not parallel to $\Gamma^0.$
\end{proof}

\subsection{Restart of the flow}

Suppose  $\norm$ is a crystalline anisotropy, $\Gamma^0:=\cup_{i=1}^nS_i^0$ is a closed admissible polygonal curve and $\{\Gamma(t)\}_{t\in[0,T^\dag)}$ is the maximal crystalline elastic
flow starting from $\Gamma^0.$ By Theorem \ref{teo:existence} (c), if $T^\dag$ is finite, then at the maximal time only some segments with zero $\norm$-curvature disappear and 
the limiting curve $\Gamma(T^\dag)$ (defined in the Kuratowski sense) is well-defined and admissible. Thus, relabelling
the segments/half-lines of $\Gamma(T^\dag),$ we can continue the flow, applying Theorems \ref{teo:existence} to
reach another maximal time $T^\ddag.$

This observation yields the following

\begin{theorem}[Restart of the flow]\label{teo:continue_flows}
Let $\norm$ be a crystalline anisotropy and $\Gamma^0$ be a closed admissible polygonal curve consisting of $n$ segments. Then there exists a unique family $\{\Gamma(t)\}_{t\in[0,+\infty)}$ of
admissible polygonal curves and $0\le m<n$-times $0 = T_0^\dag < T_1^\dag < \ldots< T_m^\dag<T_{m+1}^\dag =
+\infty$ such that

\begin{itemize}
\item[(a)] the map $t\mapsto \Gamma(t)$ is Kuratowski continuous in $[0,+\infty)$;

\item[(b)] for any $i\in\{0,\ldots,m\}$ the family $\{\Gamma(t):\,\,t\in[T_i^\dag,T_{i+1}^\dag)\}$ is the unique
maximal elastic flow starting from $\Gamma(T_i^\dag);$

\item[(c)] for any $i\in\{0,\ldots,m-1\}$ as $t\nearrow T_{i+1}^\dag$ some segments of $\Gamma(t)$ with zero $\norm$-curvature disappear;

\item[(d)] for any $i\in\{0,\ldots,m-1\}$ as $t\nearrow T_{i+1}^\dag$ the length of each segment of $\Gamma(t)$ with nonzero $\norm$-curvature is bounded away from $0;$

\item[(e)] for any $t\ge0$ the index (in the sense of Definition \ref{def:curve_index}) of $\Gamma(t)$ is equal to the index of $\Gamma^0.$ 
\end{itemize}
\end{theorem}

\begin{proof}
Applying Theorem \ref{teo:existence} inductively we find $0 = T_0^\dag < T_1^\dag < \ldots<
T_m^\dag<T_{m+1}^\dag = +\infty$ with $0\le m<n$ and for any $i\in\{0,\ldots,m\}$ we find the unique maximal elastic flow
$\{\Gamma(t):\,\,t\in[T_i^\dag,T_{i+1}^\dag)\}$ starting from $\Gamma(T_i^\dag).$
Moreover, by \eqref{segment_disappear1}, for any $i\in\{0,\ldots,m-1\}$ as $t\nearrow T_{i+1}^\dag$ only some segment (s) 
of $\Gamma(t)$ with zero $\norm$-curvature disappear and (c)-(d) hold. Finally, the Kuratowski continuity of
$\Gamma(\cdot)$ in $[T_i^\dag,T_{i+1}^\dag]$ follows from Theorem \ref{teo:existence} (c). 
To prove (e), it suffices to note that  after vanishing some zero $\norm$-curvature segments in a loop, the loop persists.
\end{proof}

Theorem \ref{teo:continue_flows} implies that the crystalline elastic evolution starting from an admissible
$\phi$-regular polygonal curve is uniquely defined for all times.

\section{Examples}\label{sec:example}

In this section we consider some examples. The first example is very similar to what happens when $\norm$ is Euclidean.

\begin{example}[Evolution of Wulff shapes]\label{ex:evol_Wulff}
Let us check that crystalline Wulff shapes evolve self-similarly.
Let $\norm$ be a crystalline anisotropy and let $\{W_{R(t)}^{\norm}\}_{t\in[0,T)}$ be a family of Wulff shapes centered at origin $O$ of radius $R(t)>0$ for all $t.$ We assume that orientation of the boundaries of all $W_{R(t)}^\norm$ are chosen such that the  corresponding normals are outward. 
For any segment $S_i(t)$ of $W_{R(t)}^\norm,$ let $h_i(t):=\scalarp{H(S_i(0)),S_i(t)}{\nu_{S_i(0)}};$ one can readily check that 
\begin{equation}\label{ptzcv}
h_i(t) = (R(t)- R(0)) \norm^o(\nu_{S_i(0)}).
\end{equation}
Being $W_{R(t)}^\norm$ convex, all transition coefficients are equal to $1.$  Moreover, by homothety, $\sH^1(S_i(t)) = R(t)\sH^1(F_i)$ and recalling \eqref{ashuc7bnn}, the right-hand side of \eqref{elastic_ode123} takes the  form
$$
-\frac{1}{R(t)} + \frac{\alpha}{R(t)^3}.
$$
Thus, if we assume the boundaries of the  Wulff shapes $\{W_{R(t)}^{\norm}\}_{t\in[0,T)}$ is an elastic flow, by \eqref{ptzcv}, we should have 
$$
R' = -\frac{1}{R} + \frac{\alpha}{R^3}\quad\text{in $(0,T).$}
$$
Hence, if $R$ is the unique solution of the ODE 
\begin{equation*}
\begin{cases}
R' = -\frac{1}{R} + \frac{\alpha}{R^3} & \text{in $(0,T^\dag)$},\\
R(0)=R_0,
\end{cases}
\end{equation*}
which is an equivalent formulation of \eqref{elastic_ode123} for the evolution of Wulff shapes, by Theorem
\ref{teo:existence}, $\p W_{R(\cdot)}^\norm$ is the unique evolution starting from $\p W_{R_0}^\norm.$ We can also provide some qualitative properties of the flow. First of all, $T^\dag=+\infty.$ Moreover:

\begin{itemize}
\item[(a)] if $R_0^2=\alpha,$ the evolution is stationary, i.e. $R(\cdot)\equiv R_0$ in
$[0,+\infty);$

\item[(b)] if $R_0^2>\alpha,$  one can readily check that $R'<0$ and hence the evolution is
self-shrinking and $R$ strictly decreases until $\sqrt{\alpha}$ as $t\to+\infty;$

\item[(c)] if  $R_0^2<\alpha,$ $R'>0$ and hence  one can show that the evolution is
self-expanding and $R$ strictly increases until $\sqrt{\alpha}$ as $t\to+\infty.$
\end{itemize}
\end{example}

\begin{wrapfigure}[6]{r}{0.4\textwidth}
\vspace*{-9mm}
\includegraphics[width=0.38\textwidth]{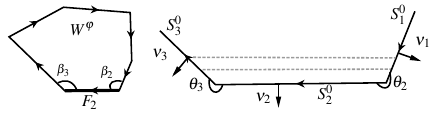}
\caption{\small The Wulff shape and the admissible polygonal curve in Example \ref{ex:translate}.}\label{fig:unbound}
\end{wrapfigure}
The observations in (b) and (c) show  that as $t\to+\infty$ the Wulff shapes $W_{R(t)}^\norm $ converge to a stationary Wulff shape. It is worth to mention that such a long-time behaviour is not specific to Wulff shapes and holds true for the globally defined evolution starting from an arbitrary admissible closed polygonal curve (see Section \ref{sec:long_time_behaviour} below).

\begin{example}[Translating-type solutions]\label{ex:translate}
Let $\norm$ be a crystalline anisotropy and let $F_2$ be a  facet of $W^\norm$ at which the sum of the interior
angles $\beta_2\in(0,\pi)$ and $\beta_3\in(0,\pi)$ of $W^\norm$ is not less than $\pi$ (see Fig.
\ref{fig:unbound}). Consider an admissible unbounded curve $\Gamma^0=S_1^0\cup S_2^0\cup S_3^0,$ as in Fig.
\ref{fig:unbound}, where $S_1^0$ and $S_3^0$ are half-lines. For shortness, let $\nu_i$ stands for the outer normal 
$S_i^0.$
As the only ``evolving'' segment is $S_2^0$ and $\Gamma^0$ is convex near $S_2^0,$ we have $c_2=1.$ Thus, the
crystalline elastic evolution equation \eqref{elastic_ode123} becomes $h_1=h_3\equiv0$ and
\begin{equation}\label{aiszfv}
\frac{h_2'}{\norm^o(S_2^0)} = -\frac{\sH^1(F_2)}{\sH^1(S_2)}- \frac{\alpha \sH^1(F_2)^2 \norm^o(\nu_2)[\cot\theta_2 +
\cot\theta_3]}{\sH^1(S_2)^3}.
\end{equation}
By admissibility, $\theta_2=2\pi-\beta_2$ and $\theta_3=2\pi-\beta_3.$ Thus, 
$$
\cot\theta_2 + \cot\theta_3 = -\cot(\beta_2)-\cot\beta_3 = \frac{\sin(2\pi-
\beta_2-\beta_3)}{\sin\beta_2\sin\beta_3}.
$$
Since $\beta_2+\beta_3\in[\pi,2\pi),$ we have $\cot\theta_2 + \cot\theta_3 >0,$ and hence, from \eqref{aiszfv} we
deduce $h_2'<0.$ Since $h_2(0)=0,$ this implies $h_2$ is decreasing and negative, which in turn yields that the
segment $S_2^0$ during the evolution translates to  infinity in the direction of $-\nu_2.$
In other words, the unique crystalline elastic flow $\Gamma(t)$ exists for all times $t>0,$ and it ``translates''
to  infinity.

This translation is more clear if $\beta_2+\beta_3 = \pi.$ In this case, the facets $F_1$ and $F_3$ of $W^\norm$ becomes parallel and hence so are $S_1^0$ and $S_3^0.$ Such curves are the analogue of the so-called translating solutions (see Definition \ref{def:translate_curves} below) or sometimes grim reapers.
\end{example}

\begin{wrapfigure}[12]{l}{0.32\textwidth}
\vspace*{-3mm}
\includegraphics[width=0.3\textwidth]{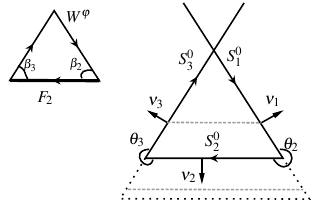}
\caption{\small A triangular Wulff shape and an admissible polygonal curve in Example \ref{ex:triangle}.}\label{fig:triang}
\end{wrapfigure}

Note that the curves in Example ``translates'' to infinity so that as time converges to infinity they disappear. The following example shows that if $\beta_2+\beta_3<\pi$ in Example \ref{ex:translate} so that  the half-lines intersect, the evolution is globally defined and converges to some stationar curve,  i.e., we obtain a similar situation as in the case of evolving Wulff shapes.

\begin{example}\label{ex:triangle}
Let $\norm$ be a crystalline anisotropy  and $F_2$ be a  facet of $W^\norm$ at which the sum of the interior angles $\beta_2,\beta_3\in(0,\pi)$ of $W^\norm$ is less than $\pi$ (see Fig. \ref{fig:triang}). 
Consider an admissible unbounded curve $\Gamma^0=S_1^0\cup S_2^0\cup S_3^0,$ as in Fig.
\ref{fig:triang}, where $S_1^0$ and $S_3^0$ are half-lines, which intersect at a single point. As in Example \ref{ex:translate},  $c_1=c_3=0$ and $c_2=1,$ and thus, $h_1=h_3\equiv0$ and
$h_2$ is a solution of \eqref{aiszfv}.  
By admissibility, $\theta_2=2\pi-\beta_2$ and $\theta_3=2\pi-\beta_3,$ and so $\cot\theta_2 + \cot\theta_3<0.$ Now:
\begin{itemize}
\item if $\sH^1(S_i^0) = \sqrt{-\alpha \sH^1(F_2)\norm^o(\nu_2)[\cot\theta_2+\cot\theta_3]},$ then $\Gamma^0$ is stationary, i.e., the constant family $\Gamma(\cdot)\equiv \Gamma^0$ is the unique solution of \eqref{elastic_ode123}.

\item if $\sH^1(S_i^0)\ne \sqrt{-\alpha \sH^1(F_2)\norm^o(\nu_2)[\cot\theta_2+\cot\theta_3]},$ then $T^\dag=+\infty$ and 
$$
\lim\limits_{t\to+\infty} 
\sH^1(S_i(t)) =  \sqrt{-\alpha \sH^1(F_2)\norm^o(\nu_2)[\cot\theta_2+\cot\theta_3]},
$$
that is, the flow $\{\Gamma(t)\}$ converges to a stationary solution as $t\to+\infty.$
\end{itemize}
\end{example}

\section{Stationary solutions}\label{sec:stationar_solutino}

Inspired from Examples \ref{ex:evol_Wulff} and \ref{ex:triangle}, in this section we introduce the notion of a stationary curve.

\begin{definition}[Stationary curves]\label{def:reg_stationar}
Let $\norm$ be a crystalline anisotropy. An admissible polygonal curve $\Gamma=\cup_i S_i$ is called
\emph{stationary} provided that
\begin{equation}\label{stationarS1}
c_i\sH^1(F_i)+ \alpha \Big(\frac{c_{i-1}^2\delta_{i-1}}{\sH^1(S_{i-1})^2\sin\theta_i}+ \frac{c_{i}^2\delta_i[\cot
\theta_i+\cot \theta_{i+1}]}{\sH^1(S_{i})^2}+\frac{c_{i+1}^2\delta_{i+1}}{\sH^1(S_{i+1})^2\sin\theta_{i+1}}\Big)
= 0
\end{equation}
for all segments $S_i,$ where $\delta_j>0$ are defined in \eqref{def:delta_is}.
\end{definition}

Clearly, if $\Gamma^0$ is stationary, $\Gamma(t):=\Gamma^0$ is the crystalline elastic flow starting from
$\Gamma^0.$ Therefore, we sometimes call $\Gamma^0$ a stationary solution.

\begin{example}
If $W^\norm$ is the square $[-1,1]^2,$ then $\sH^1(F_i)=2,$ $\norm^o(\nu)=1$ for the normal $\nu$ of any admissible polygonal curve, which can only lie in $\{\pm\be_1,\pm\be_2\},$ with angles $\theta_i\in\{\pi/2,3\pi/2\}.$ Thus, the stationarity equation \eqref{stationarS1} simplifies to
\begin{equation}\label{hazsut}
c_i + \frac{2\alpha c_{i-1}^2}{\sH^1(S_{i-1})^2\sin\theta_{i-1}} + \frac{2\alpha
c_{i+1}^2}{\sH^1(S_{i+1})^2\sin\theta_{i+1}}=0
\end{equation}
for any segment $S_i.$  
\end{example}

An important property of stationary solutions is related to their elastic energy.

\begin{proposition}
Let $\norm$ be a crystalline anisotropy and $\Gamma:=\cup_{i=1}^n S_i$ be an admissible stationary polygonal
curve.

\begin{itemize}
\item Assume that $\Gamma$ is closed. Then for any admissible polygonal curve $\bar\Gamma:=\cup_{i=1}^n\bar
S_i,$ parallel to $\Gamma,$
\begin{equation}\label{diff_energy_bdd_stat}
\cF_\alpha(\bar \Gamma) - \cF_\alpha(\Gamma) = \alpha\sum_{i=1}^{n} \frac{ c_i^2 \delta_i}{\sH^1(S_i)^2 \sH^1(\bar
S_i)} \,\Big(\sH^1(\bar S_i) - \sH^1(S_i)\Big)^2.
\end{equation}

\item Assume that $\Gamma$ is unbounded and $S_1$ and $S_n$ are half-lines. Then for any admissible polygonal
curve $\bar\Gamma:=\cup_{i=1}^n\bar S_i,$ parallel to $\Gamma$ and for any  bounded open set $D\subset\R^2$ compactly
containing all segments of $\Gamma$ and $\bar\Gamma,$
\begin{equation}\label{diff_energy_unbdd_stat}
\cF_\alpha(\bar \Gamma,D) - \cF_\alpha(\Gamma,D) = \alpha \sum_{i=2}^{n-1} \frac{ c_i^2 \delta_i}{\sH^1(S_i)^2
\sH^1(\bar S_i)} \,\Big(\sH^1(\bar S_i) - \sH^1(S_i)\Big)^2.
\end{equation}
\end{itemize}
Thus, any admissible stationary polygonal curve $\Gamma$ is a local minimizer of $\cF_\alpha$ among all
polygonal curves   parallel to $\Gamma.$ 
Moreover, if $\Gamma$ is closed, then it is a minimizer of $\cF_\alpha$ among all curves parallel to $\Gamma.$ 
\end{proposition}

\begin{proof}
Let $\Gamma:=\cup_{i=1}^n S_i$ be an admissible stationary polygonal curve and $\bar \Gamma:=\cup_{i=1}^n \bar
S_i$ be a curve  parallel to $\Gamma$ with corresponding signed heights $h_i:=\scalarp{H(S_i,\bar S_i)}{\nu_{S_i}}.$

First assume that $\Gamma$ (and hence $\bar\Gamma$) is closed and consider the difference 
\begin{align*}
\cF_\alpha(\bar \Gamma) -\cF_\alpha(\Gamma) = & \sum_{i=1}^n \norm^o(\nu_{S_i}) \Big(\sH^1(\bar S_i)-\sH^1(S_i) +
\alpha c_i^2\sH^1(F_i)^2\Big(\tfrac{1}{\sH^1(\bar S_i)}- \tfrac{1}{\sH^1(S_i)}\Big) \Big) \\
= & \sum_{i=1}^n \norm^o(\nu_{S_i}) \big(\sH^1(\bar S_i)-\sH^1(S_i)\big)\, \Big( 1 - \tfrac{\alpha
c_i^2\sH^1(F_i)^2}{\sH^1(S_i)\sH^1(\bar S_i^{})} \Big).
\end{align*}
By \eqref{length_changed_hhh} we can represent the difference as 
\begin{multline}\label{gstdzfg}
\cF_\alpha(\bar \Gamma) -\cF_\alpha(\Gamma) = -\sum_{i=1}^n 
\norm^o(\nu_{S_i}) \Big(\tfrac{h_{i-1}}{\sin\theta_i} + h_i[\cot\theta_i+\cot\theta_{i+1}] +
\tfrac{h_{i+1}}{\sin\theta_{i+1}}\Big)\,\Big(1 - \tfrac{\alpha c_i^2\sH^1(F_i)^2}{\sH^1(S_i)^2} \Big) \\
-\sum_{i=1}^n 
\norm^o(\nu_{S_i}) \Big(\tfrac{h_{i-1}}{\sin\theta_i} + h_i[\cot\theta_i+\cot\theta_{i+1}] +
\tfrac{h_{i+1}}{\sin\theta_{i+1}}\Big)\times \\
\times \Big(\tfrac{\alpha c_i^2\sH^1(F_i)^2}{\sH^1(S_i)^2}
-
\tfrac{\alpha c_i^2\sH^1(F_i)^2}{\sH^1(S_i)\Big[\sH^1(S_i)
-
\Big( \tfrac{h_{i-1}}{\sin\theta_i}
+ h_i[\cot\theta_i+\cot\theta_{i+1}] + \tfrac{h_{i+1}}{\sin\theta_{i+1}}\Big) \Big]}\Big).
\end{multline}
As in the proof of Theorem \ref{teo:cl_elastc_gradflow}, representing all sums with respect to single $h_i,$ we get
\begin{align*}
\sum_{i=1}^n 
\norm^o(\nu_{S_i}) &\Big(\tfrac{h_{i-1}}{\sin\theta_i} + h_i[\cot\theta_i+\cot\theta_{i+1}] +
\tfrac{h_{i+1}}{\sin\theta_{i+1}}\Big)\,\Big(1 - \tfrac{\alpha c_i^2\sH^1(F_i)^2}{\sH^1(S_i)^2} \Big) \\
= &\sum_{i=1}^n h_i \Big(\tfrac{\norm^o(\nu_{S_{i-1}}) }{\sin\theta_i } + \norm^o(\nu_{S_i})[\cot \theta_i+\cot
\theta_{i+1}]+ \tfrac{\norm^o(\nu_{S_{i+1}}) }{\sin\theta_{i+1} } \\
&-\tfrac{\alpha c_{i-1}^2\sH^1(F_{i-1})^2 \norm^o(\nu_{S_{i-1}})}{\sH^1(S_{i-1})^2\sin\theta_i} - \tfrac{\alpha
c_{i}^2\sH^1(F_{i})^2\norm^o(\nu_{S_{i}}) [\cot \theta_i+\cot \theta_{i+1}]}{\sH^1(S_{i})^2} -
\tfrac{\alpha c_{i+1}^2\sH^1(F_{i+1})^2\norm^o(\nu_{S_{i+1}})}{\sH^1(S_{i+1})^2\sin\theta_{i+1}}\Big) \\
= &\sum_{i=1}^n h_i \Big(-c_i\sH^1(F_i) - \tfrac{\alpha c_{i-1}^2\sH^1(F_{i-1})^2
\norm^o(\nu_{S_{i-1}})}{\sH^1(S_{i-1})^2\sin\theta_i} -
\tfrac{\alpha c_{i}^2\sH^1(F_{i})^2\norm^o(\nu_{S_{i}}) [\cot \theta_i+\cot \theta_{i+1}]}{\sH^1(S_{i})^2} -
\tfrac{\alpha c_{i+1}^2\sH^1(F_{i+1})^2\norm^o(\nu_{S_{i+1}})}{\sH^1(S_{i+1})^2\sin\theta_{i+1}}\Big),
\end{align*}
where in the last equality we used \eqref{ashuc7bnn}. By the stationarity condition each summand in the last sum is
$0.$
Thus, \eqref{gstdzfg} reads as 
\begin{equation*}
\cF_\alpha(\bar \Gamma) -\cF_\alpha(\Gamma) = 
\alpha  \sum_{i=1}^n \tfrac{c_i^2\sH^1(F_i)^2\norm^o(\nu_{S_i})}{\sH^1(S_i)^2 \sH^1(\bar S_i)}
\, 
\Big(\tfrac{h_{i-1}}{\sin\theta_i} + h_i[\cot\theta_i+\cot\theta_{i+1}] +
\tfrac{h_{i+1}}{\sin\theta_{i+1}}\Big)^2.
\end{equation*}
Now, recalling \eqref{length_changed_hhh}, we conclude \eqref{diff_energy_bdd_stat}.

Analogously, if $\Gamma$ is unbounded, then $S_1,\bar S_1$ and $S_n,\bar S_n$ are half-lines with compact
symmetric differences (in particular, $c_1=c_n=0$). Hence for any bounded open set $D\subset\R^2$ compactly
containing all $S_i$ and $\bar S_i$ with $2\le i\le n-1$ we have
\begin{equation*}
\cF_\alpha(\bar \Gamma,D) -\cF_\alpha(\Gamma,D) = \alpha 
\sum_{i=2}^{n-1} \tfrac{ c_i^2\sH^1(F_i)^2\norm^o(\nu_{S_i})}{\sH^1(S_i)^2 \sH^1(\bar S_i)}
\, 
\Big(\tfrac{h_{i-1}}{\sin\theta_i} + h_i[\cot\theta_i+\cot\theta_{i+1}] +
\tfrac{h_{i+1}}{\sin\theta_{i+1}}\Big)^2.
\end{equation*}
Now, using again \eqref{length_changed_hhh}, we conclude \eqref{diff_energy_unbdd_stat}.
\end{proof}

\begin{remark}\label{rem:station_lengith}
The identities \eqref{diff_energy_bdd_stat} and \eqref{diff_energy_unbdd_stat} show that if $\Gamma$ and
$\bar\Gamma$ are parallel and 
both stationary, then $\sH^1(S_i)=\sH^1(\bar S_i)$ whenever $c_i\ne0,$ in other
words, in two parallel stationary curves, the length of all segments with nonzero $\norm$-curvature must
coincide.
\end{remark}

\begin{example}
Let $W^\norm = [-1,1]^2$ and consider the family of admissible polygonal curves depicted in Fig. \ref{fig:stat_09}.  
\begin{figure}[htp!]
\includegraphics[width=0.85\textwidth]{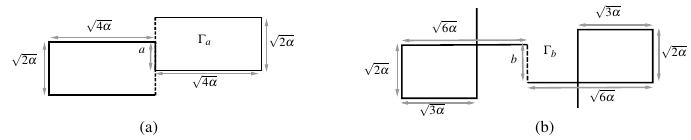}
\caption{\small Family of stationary curves. In (a), two rectangles can slide up and down with the restriction 
$a \in(0,\sqrt{2\alpha}].$  Clearly, all of such curves are stationary and parallel to each other. Also in (b), the rectangles can move up and down so that the segment of length $b$ elongates or shortens. Unlike (a), here $b\in(0,+\infty),$  and all of such curves are stationary and parallel to each other.}\label{fig:stat_09}
\end{figure}
One can readily check that the families $\{\Gamma_a\}_{a\in(0,\sqrt{2\alpha}]}$ and $\{\Gamma_b\}_{b\in(0,+\infty)}$ consist of parallel stationary curves. We refer to Theorem \ref{teo:stationary_square} fot a full classification of stationary curves in the square anisotropy. 
\end{example}

Now we prove a Liouville-Lojasiewicz (or Lojasiewicz-Simon) type inequality. 

\begin{proposition}[Lojasiewicz-Simon-type inequality, I]\label{prop:loja_simon_ineq}
Let $\norm$ be a crystalline anisotropy, $\Gamma^0:=\cup_{i=1}^nS_i^0$ be a closed admissible stationary
polygonal curve and $\Delta_1:=\Delta_1(\Gamma^0)>0$ be the constant of Lemma \ref{lem:apriori_estocada} defined
for $\Gamma:=\Gamma^0.$
Then there exist constants $C,\beta>0$  such that, given an admissible polygonal curve $\Gamma:=\cup_{i=1}^n S_i,$
parallel to $\Gamma^0$ with
\begin{equation}\label{ahstzddff}
\max_{1\le i\le n} |H(S_i^0,  S_i)|<\Delta_1, 
\end{equation}
one has 
\begin{multline}\label{loja_simo_tengis}
|\cF_\alpha( \Gamma) - \cF_\alpha(\Gamma^0)|^\beta
\le C \Big(\alpha \max_{i=1,\ldots,n} \,c_i^2\delta_i\,\sH^1(S_i)\Big)^\beta \times\\
\times\sum_{i=1}^n \Big|c_i\sH^1(F_i)+ \alpha \Big(\frac{c_{i-1}^2\delta_{i-1}}{\sH^1( S_{i-1})^2\sin\theta_i}+
\frac{c_{i}^2\delta_i[\cot \theta_i+\cot \theta_{i+1}]}{\sH^1( S_{i})^2}+\frac{c_{i+1}^2\delta_{i+1}}{\sH^1(
S_{i+1})^2\sin\theta_{i+1}}\Big)\Big|,
\end{multline}
where $\{\theta_i\}$ are the angles of $\Gamma$ and $\delta_j>0$ are defined in \eqref{def:delta_is}.
\end{proposition}

This inequality will be used in the proof of Theorem \ref{teo:conver_stationar_sol} to study the long time
behaviour of the crystalline elastic flow.

\begin{proof}
To prove the proposition we apply the classical results on semi- and subanalytic sets and functions. We refer to
\cite{BM:1988, Hironaka:1973, Loja:1995} for precise definitions and related result.

Consider the real-analytic family $g:=(g_1,\ldots,g_n)$ with  
$$
g_i(h_1,\ldots,h_n) := \frac{1}{\sH^1(S_i^0) -
\big(\frac{h_{i-1}}{\sin\theta_i} 
+
h_i[\cot\theta_i+\cot\theta_{i+1}]
+
\frac{h_{i+1}}{\sin\theta_{i+1}}\big) },\quad
(h_1,\ldots,h_n)\in(-1.5\Delta_1, 1.5\Delta_1)^n,
$$
where $h_0:=h_n$ and $h_{n+1}:=h_1.$ 
By Lemma \ref{lem:construc_paral_curve}, for any $(h_1,\ldots,h_n)\in(-\Delta_1,\Delta_1)^n$ there exists exists
a unique polygonal curve $\Gamma$ parallel to $\Gamma$ such that $h_i=\scalarp{H(S_i^0,S_i)}{\nu_{S_i}}.$
Moreover, by \eqref{length_changed_hhh}, the definition of $\Delta_1$ and the definition of $g_i,$
$$
\tfrac{\sH^1(S_i^0)}{4} < \sH^1(S_i) < \tfrac{7\sH^1(S_i^0)}{4}
\quad\text{and}\quad 
\sH^1(S_i) = \tfrac{1}{g_i(h_1,\ldots,h_n)}, 
\quad i=1,\ldots,n. 
$$
Thus, the image 
$$
O:= g((-\Delta_1,\Delta_1)^n)\subset\R^n
$$ 
of $g$ is a bounded set contained in the compact hypercube 
$
\prod_{i=1}^n[\frac{4}{7\sH^1(S_i^0)},\frac{4}{\sH^1(S_i^0)}].
$
Being $(-\Delta_1,\Delta_1)^n$ open and $g$ real-analytic, the graph of 
$g$ is semianalytic in $\R^n\times\R^n,$ and hence, being a projection onto $\R^n,$ the set $O$ is subanalytic. In particular, the
boundary, interior and closure of $O$ are also subanalytic.

Now consider the continuous function $f:\cl{O}\to\R$ of $n$ variables  
$$
f(x): = \sum_{i=1}^n \Big|c_i\sH^1(F_i)+ \alpha\Big( \tfrac{c_{i-1}^2\delta_{i-1}}{\sin\theta_i} x_{i-1}^2
+ c_{i}^2\delta_i[\cot \theta_i+\cot \theta_{i+1}]\,x_i^2+\tfrac{
c_{i+1}^2\delta_{i+1}}{\sin\theta_{i+1}}\,x_{i+1}^2\Big)\Big|,\quad x=(x_1,\ldots,x_n),
$$
where as usual $x_0:=x_n$ and $x_{n+1}:=x_1.$ Being a sum of   absolute values of real-analytic functions, $f$
is subanalytic. Thus, we can apply the Lojasiewicz inequality (see e.g. \cite[Section IV.9]{Loja:1995}) to find
positive constants $\beta'$ and $C'>0$ (depending on $O$) such that
\begin{equation}\label{loja_ineq_subanaly}
|f(x)| \ge C' \dist(x,\{f=0\})^{\beta'},\quad x\in O,
\end{equation}
where $\{f=0\}:=\{z\in \cl{O}:\,\, f(z)=0\}$ is the $0$-level set of $f.$ 

Consider any $z=(z_1,\ldots,z_n)\in \{f=0\}$ and let $z^k\in O$ be such that
$z^k\to z$ as $k\to+\infty.$ By the definition of $O,$ for each $k\ge1$ there exists
$h^k:=(h_1^k,\ldots,h_n^k)\in (-\Delta_1,\Delta_1)^n$ such that $g(h^k) = z^k.$ Up
to a not relabeled subsequence, we may suppose $h^k\to h^\infty \in[-\Delta_1,\Delta_1]^n.$ Clearly, $z=g(h^\infty),$ and as we
observed earlier, we can find a unique polygonal curve $\Gamma^\infty:=\cup_{i=1}^nS_i^\infty,$ parallel to
$\Gamma^0,$ satisfying $h_i^\infty = \scalarp{H(S_i^0,S_i^\infty)}{\nu_{S_i^0}}$ and $\sH^1(S_i^\infty) =
\frac{1}{g_i(h^\infty)} = \frac{1}{z_i}$ for all $i=1,\ldots,n.$ The last relation, the definition of $f$ and the
equation $f(z)=0$ imply that $(\frac{1}{\sH^1 (S_1^\infty)},\ldots,\frac{1}{\sH^1(S_n^\infty)})$
satisfy \eqref{stationarS1} and thus, $\Gamma^\infty$ is
stationary. Since $\Gamma^0$ is also stationary,
parallel to $\Gamma^\infty,$ by Remark \ref{rem:station_lengith} we have 
\begin{equation*}
\sH^1(S_i^0) = \sH^1(S_i^\infty)\quad\text{whenever $i\in\{1,\ldots,n\}$ and $c_i\ne0.$} 
\end{equation*}

Now, fix any $\epsilon\in(0,\Delta_1)$ and take a  polygonal curve $\Gamma,$ parallel to $\Gamma^0$ and
satisfying \eqref{ahstzddff}. By \eqref{diff_energy_bdd_stat}, we have
$$
\cF_\alpha(\Gamma) - \cF_\alpha(\Gamma^0) =\alpha \sum_{i=1}^{n} \tfrac{  c_i^2 \delta_i}{\sH^1(S_i^0)^2
\sH^1(S_i)} \,\Big(\sH^1(S_i^0) - \sH^1(S_i)\Big)^2 = \alpha 
\sum_{i=1}^n c_i^2 \delta_i\, \sH^1(S_i) \,\Big(\tfrac{1}{\sH^1(S_i^0)} - \tfrac{1}{\sH^1(S_i)}\Big)^2.
$$
Hence, 
\begin{equation}\label{kstsss}
|\cF_\alpha(\Gamma) - \cF_\alpha(\Gamma^0)| \le \alpha \max_{j}  c_j^2 \delta_j\, \sH^1(S_j)
\sum_{c_i\ne 0}\Big(\tfrac{1}{\sH^1(S_i^0)} - \tfrac{1}{\sH^1(S_i)}\Big)^2.
\end{equation}
Set
$
x=(\frac{1}{\sH^1(S_1)},\ldots,\frac{1}{\sH^1(S_n)});
$ 
by \eqref{ahstzddff} and the definition of $g,$ one has $x\in O.$ Let $ z\in \{f=0\}$ be such that
$\dist(x,\{f=0\}) = |x-  z|.$ As we observed above, $z_i = \frac{1}{\sH^1(S_i^0)}$ whenever $c_i\ne0.$ Thus,
we can represent and further estimate \eqref{kstsss} as
\begin{multline*}
|\cF_\alpha(\Gamma) - \cF_\alpha(\Gamma^0)| \le \alpha \max_{1\le i\le n}  c_i^2 \delta_i\, \sH^1(S_i)
\sum_{c_i\ne 0}|z_i-x_i|^2 \\
\le \alpha  \max_{1\le i\le n} c_i^2 \delta_i\sH^1(S_i)\, |z-x|^2=
\max_{1\le i\le n} \alpha c_i^2 \delta_i\sH^1(S_i)\,\dist(x,\{f=0\})^2.
\end{multline*}
Now, recalling \eqref{loja_ineq_subanaly}, we deduce 
$$
|\cF_\alpha(\Gamma) - \cF_\alpha(\Gamma^0)|^{\frac{\beta'}{2}} \le \Big(\alpha  \max_{1\le i\le n} c_i^2
\delta_i\sH^1(S_i)\Big)^{\frac{\beta'}{2}}\, \frac{|f(x)|}{C'},
$$
which is \eqref{loja_simo_tengis} with $\beta=\beta'/2>0$ and $C=1/C'>0.$
\end{proof}

\section{Long-time behaviour of the crystalline elastic flow}\label{sec:long_time_behaviour}

In this section we investigate the infinite time behaviour of the flow and its convergence to a stationary solution.

\subsection{Long time behaviour: regular case}

Our first aim is to prove the following theorem, assuming that no segment of the  curves disappear near infinity. 

\begin{theorem}[Long-time behaviour, I]\label{teo:conver_stationar_sol}
Let $\norm$ be a crystalline anisotropy and $\{\Gamma(t)\}_{t\ge0}$ be the maximal (in time) elastic flow, starting from a closed admissible polygonal curve, with finitely many restarts (see Theorem \ref{teo:continue_flows}).
Assume that there exist  $T\ge0$ and $a_0>0$ such that
\begin{equation}\label{min_segments00000}
\inf_{t\ge T}\,\,\min_{i}\,\sH^1(S_i(t)) \ge a_0>0. 
\end{equation}
Then there exists a stationary polygonal curve $\Gamma^\infty,$ parallel to $\Gamma(T),$ such that $\Gamma(t)$
Kuratowski converges to $ \Gamma^\infty$ as $t\to+\infty.$
\end{theorem}

\begin{proof}
We follow the arguments in the Euclidean case, see e.g. \cite{MP:2021_cvpde}. There is no loss of generality in
assuming $T=0$ and let $\Gamma^0:=\Gamma(T):=\cup_{i=1}^nS_i^0.$
Setting $h_i:=\scalarp{H(S_i^0,S_i(t))}{\nu_{S_i^0}}$ and integrating \eqref{grad_floqasdf} in $[0,t],$ we get 
\begin{equation}\label{hztagsr56df}
\cF_\alpha(\Gamma^0) - \cF_\alpha(\Gamma(t)) = \sum_{i=1}^n \int_0^t \tfrac{|h_i'(s)|^2\sH^1(S_i(s))}{\norm^o(S_i^0)}ds.
\end{equation}
Thus, 
$$
\sup_{t\ge0 } \cF_\alpha(\Gamma(t)) \le \cF_\alpha(\Gamma^0)<+\infty
$$
and hence, by \eqref{min_segments00000} and the definition of $\cF_\alpha,$ 
\begin{equation}\label{hamidaka1w23}
0<a_0 \le \sH^1(S_i(t)) \le \frac{\cF_\alpha(\Gamma^0)}{c_\norm}\quad \text{for any $i=1,\ldots,n$ and $t\ge0$}
\end{equation}
(see also Corollary \ref{cor:segments_nonzero_curva}). Letting $t\to+\infty$ in \eqref{hztagsr56df} and using
the lower bound in \eqref{hamidaka1w23}, we can find a sequence $t_j\nearrow+\infty$ such that
\begin{equation*}
\lim\limits_{j\to+\infty}\,h_i'(t_j)=0\quad\text{as $j\to+\infty$ for all $i=1,\ldots,n.$}
\end{equation*}
Now consider the closed curves $\Gamma(t_j),$ $j\ge1.$ By the upper bound in \eqref{hamidaka1w23}, it follows that 
$\sH^1(\Gamma(t_j)) \le C$ for some $C>0$ depending only on $n,$ $\cF_\alpha(\Gamma^0)$ and $ c_\norm.$
Thus, $\Gamma(t_j)$ is contained in a disc of radius $2C$ centered at some point of $\Gamma(t_j).$ Therefore,
we can find a sequence $(p_j)$ of vectors in $\R^2,$ for which the translated curves $p_j+\Gamma(t_j)$ are
contained in the disc $D_{2C}$ centered at the origin. By compactness, up to a subsequence,
$p_j+\Gamma(t_j)\overset{K}{\to} \Gamma^\infty$ for some closed set $\Gamma^\infty.$ In view of the bounds in
\eqref{hamidaka1w23}, $\Gamma^\infty$ is also an admissible polygonal curve, parallel to $\Gamma^0,$ and by the
Kuratowski convergence,
\begin{equation}\label{gdtzerr}
\lim\limits_{j\to+\infty}\,\,\max_{1\le i\le n} |H(S_i^\infty,p_j+S_i(t_j))| = 0.
\end{equation}
In particular, 
\begin{equation}\label{ud8immm}
\lim\limits_{j\to+\infty}\,\sH^1(S_i(t_j)) = \sH^1(S_i^\infty),\quad i=1,\ldots,n.
\end{equation}

Now consider the translations $\tilde \Gamma^j(\cdot):=p_j+\Gamma(\cdot).$ Clearly, 
\begin{equation}\label{hstdzfgg}
\tilde h_i^j(t):=\scalarp{H(S_i^0,\tilde S_i^j(t))}{\nu_{S_i^0}} = h_i(t) + \scalarp{p_j}{\nu_{S_i^0}},\quad
i=1,\ldots,n.
\end{equation}
Since the crystalline elastic flow $\{\Gamma(\cdot)\}$ is invariant under translations and $p_j$ is independent of
time,
\begin{equation}\label{aus6fbnb}
\frac{d}{dt} \tilde h_i^j = - \tfrac{\Theta_i^j}{\sH^1(\tilde S_i^j)}:= - \norm^o(\nu_{S_i^0})\Big(\tfrac{c_i\sH^1(F_i)}{\sH^ 1(\tilde S_i^j)} +
\tfrac{\alpha}{\sH^1(\tilde S_i^j)} \Big[\tfrac{c_{i-1}^2\delta_{i-1}}{\sH^1(\tilde S_{i-1}^j)^2\sin\theta_i}+
\tfrac{c_{i}^2\delta_i[\cot \theta_i+\cot \theta_{i+1}]}{\sH^1(\tilde
S_i^j)^2}+\tfrac{c_{i+1}^2\delta_{i+1}}{\sH^1(\tilde S_{i+1}^j)^2\sin\theta_{i+1}}\Big]\Big)
\end{equation}
in $(0,+\infty),$ where as usual 
$$
\delta_i:=\sH^1(F_i)^2 \norm^o(\nu_{S_i}),\quad i=1,\ldots,n.
$$
In particular, recalling $\sH^1(\tilde S_i^j(t_j)) = \sH^1(S_i(t_j)),$ the lower bound in \eqref{hamidaka1w23}
and the relations \eqref{ud8immm}, and letting $j\to+\infty$ in \eqref{aus6fbnb} evaluated at $t_j,$ we deduce
$$
c_i\sH^1(F_i) + \alpha\Big[\tfrac{c_{i-1}^2\delta_{i-1}}{\sH^1(S_{i-1}^\infty)^2\sin\theta_i}+
\tfrac{c_{i}^2\delta_i[\cot \theta_i+\cot
\theta_{i+1}]}{\sH^1(S_i^\infty)^2}+\tfrac{c_{i+1}^2\delta_{i+1}}{\sH^1(S_{i+1}^\infty)^2\sin\theta_{i+1}}\Big] =
0,\quad i=1,\ldots,n.
$$
Thus, $\Gamma^\infty$ is stationary.
\medskip 

For each $j\ge1,$ let $I_j\subset(0,+\infty) $ be the set of all $t>0$ for which 
$$
\max_{1\le i\le n} |H(S_i^\infty, \tilde S_i^j(t)|<\Delta_1,
$$
where $\Delta_1:=\Delta_1(\Gamma^\infty)$ is given by Lemma \ref{lem:apriori_estocada} with
$\Gamma=\Gamma^\infty.$ In view of \eqref{gdtzerr}, $I_j$ is nonempty (it contains at least $t_j$) and by the
continuity of $\tilde h_i^j,$ it is an open set. Let $C,\beta>0$ be given by Proposition
\ref{prop:loja_simon_ineq} applied with $\Gamma^\infty$ replacing $\Gamma$ and consider the function
$$
\ell(t):=|\cF_\alpha(\tilde \Gamma^j(t)) - \cF_\alpha(\Gamma^\infty)|^\sigma,\quad t\in I_j,
$$
for some $\sigma\in(0,1)$ to be chosen later.  Note that by the Lojasiewicz-Simon inequality \eqref{loja_simo_tengis}, $\ell(t_j)\to0$ as $j\to+\infty. $
Moreover,
\begin{align*}
\ell'(t) = & \sigma |\cF_\alpha(\tilde \Gamma^j(t)) -
\cF_\alpha(\Gamma^\infty)|^{\sigma-1}\,\frac{d}{dt}\cF_\alpha(\tilde \Gamma^j(t)) = -\sigma
\ell(t)^{\frac{\sigma-1}{\sigma}}\sum_{i=1}^n\Big |\frac{d}{dt}{\tilde h_i^j}(t)\Big|^2\sH^1(\tilde S_i^j(t)).
\end{align*}
Thus, using \eqref{hamidaka1w23}, the convexity of $p\mapsto |p|^2$ and the equation \eqref{aus6fbnb} we find
\begin{align*}
-\ell'(t) \ge \frac{\sigma a_0\ell(t)^{\frac{\sigma-1}{\sigma}} }{n}\, \Big(\sum_{i=1}^n \Big|\frac{d}{dt}{\tilde h_i^j}(t)\Big|\Big)^2 = \frac{\sigma a_0\ell(t)^{\frac{\sigma-1}{\sigma}} }{n}\, \sum_{i=1}^n \Big|\frac{d}{dt}{\tilde h_i^j}(t)\Big| \,\sum_{i=1}^n
\Big|\tfrac{\Theta_i^j(t)}{\sH^1(\tilde S_i^j)}\Big|
\end{align*}
for the $\Theta_i^j$ defined in \eqref{aus6fbnb}. 
By \eqref{loja_simo_tengis} and the upper bound in \eqref{hamidaka1w23}, the last sum can be bounded from below by 
$$
\frac{\ell(t)^{\sigma \beta}}{C \cF_\alpha(\Gamma^0)\, (\max_i \alpha c_i^2\delta_i \sH^1(\tilde S_i^j))^\beta},
$$
and therefore, 
\begin{align*}
-\ell'(t) \ge 
 \frac{\sigma a_0\ell(t)^{\frac{\sigma-1}{\sigma} + \sigma\beta} }{C n\cF_\alpha(\Gamma^0)\, (\max_i \alpha c_i^2\delta_i \sH^1(\tilde S_i^j))^\beta}\, \sum_{i=1}^n \Big|\frac{d}{dt}{\tilde h_i^j}(t)\Big| \,\sum_{i=1}^n
\Big|\tfrac{\Theta_i^j(t)}{\sH^1(\tilde S_i^j)}\Big|
\end{align*}
Now we choose $\sigma$ as 
the unique positive solution of the equation 
$$
\beta\sigma^2+\sigma-1 = 0.
$$
Then using once more the relation $\sH^1(\tilde S_i^j) = \sH^1(S_i)$ and the upper bound in \eqref{hamidaka1w23}, we obtain
\begin{align}\label{hatagsdf}
-\ell'(t) \ge C_1 \sum_{i=1}^n \Big|\frac{d}{dt}{\tilde h_i^j}(t)\Big|
\end{align}
for some $C_1>0$ depending only on $\sigma,$ $n,$ $\alpha,$ $\cF_\alpha(\Gamma^0)$ and $a_0.$

Now, consider any finite interval $(l,r)\subset I_j.$ As in \cite[Eq. 4.7]{MP:2021_cvpde}, by \eqref{hatagsdf} we
have
\begin{equation}\label{fahstdzvb}
\sum_{i=1}^n|\tilde h_i^j(r) - \tilde h_i^j(l)| \le \int_{l}^{r} \sum_{i=1}^n \Big|\frac{d}{dt}\tilde
h_i^j(t)\Big| dt
\le -\frac{1}{C_1} \int_{l}^{r} \ell'(t)dt \le \frac{\ell(l)}{C_1} = \frac{1}{C_1}|\cF_\alpha(\tilde \Gamma^j(l))
- \cF_\alpha(\Gamma^\infty)|^{\sigma}.
\end{equation} 

Fix any  $\epsilon\in(0,2^{-10}\Delta_1)$ and let $j\ge1$ be so large that 
$$
\max_{1\le i\le n} |H(S_i^\infty, \tilde S_i^j(t_j))|<\epsilon\quad\text{and}\quad |\cF_\alpha(\tilde
\Gamma^j(l)) - \cF_\alpha(\Gamma^\infty)|^{\sigma}<C_1\epsilon.
$$
Let us choose $l=t_j$ in \eqref{fahstdzvb} and let $\bar r\in(t_j,+\infty]$ be the supremum of all $r$ for which
$(l,r)\subset I_j.$ We claim that $\bar r=+\infty.$ Indeed, if $\bar r<+\infty,$ by the continuity of $t\mapsto
|H(S_i^\infty,\tilde S_i^j(t))|,$ we would have
\begin{equation}\label{rrrtzdv}
\max_{1\le i\le n} |H(S_i^\infty, \tilde S_i^j(\bar r))| = \Delta_1.
\end{equation}
On the other hand, by the choice of $\epsilon$ and the relation \eqref{fahstdzvb} applied in $(t_j,\bar r)$ we
would have
$$
|H(S_i^\infty, \tilde S_i^j(\bar r))| \le |H(S_i^\infty, \tilde S_i^j(t_j))| + |H(\tilde S_i^j(t_j), \tilde
S_i^j(\bar r))| <\epsilon + |\tilde h_i^j(t_j) - \tilde h_i^j(\bar r)|<2\epsilon.
$$
Thus, the equality in \eqref{rrrtzdv} is impossible and $\bar r=+\infty.$

In view of \eqref{hstdzfgg} and \eqref{fahstdzvb} as well as the choice of $\epsilon,$ for any $t>t_j$ we have
\begin{equation*}
\sum_{i=1}^n|h_i(t) -  h_i(t_j)| = \sum_{i=1}^n|\tilde h_i^j(t) - \tilde h_i^j(t_j)| \le \epsilon.
\end{equation*}
This estimate shows that the families $\{h_i(t)\}_{t>0},$ $i=1,\ldots,n,$ are fundamental as $t\to+\infty.$ Then there exists $\bar h_i:=\lim_{t\to+\infty} h_i(t)$ for any $i.$
In particular, applying \eqref{hstdzfgg} at $t=t_j$ and letting $j\to+\infty$ we find that the limit of
$\scalarp{p_j}{\nu_{S_i^0}}$ exists for all $i.$ Since $\Gamma^0$ is closed, it has at least two nonparallel
normals. Thus, the sequence $p_j$ also converges to some $p^\infty$ as $j\to+\infty.$ Finally, we claim that
$\Gamma(t)\overset{K}{\to} -p^\infty + \Gamma^\infty$ as $t\to+\infty.$
Indeed, for any $i\in\{1,\ldots,n\}$ and $t>t_j,$
\begin{multline*}
|H(-p^\infty + S_i^\infty, S_i(t))| \le |H(-p_j+S_i^\infty,-p^\infty+S_i^\infty| + |H(S_i^\infty, p_j+S_i(t_j))|
+ |H(p_j+S_i(t_j), p_j+S_i(t))| \\
\le |p_j-p^\infty| + |H(S_i^\infty, p_j+S_i(t_j))| + |h_i(t_j)-\bar h_i(t)| \to0
\end{multline*}
as $t\to+\infty$ and $j\to+\infty.$ Since all segments have length away from zero, the segments of $\Gamma(t)$ Kuratowski converges to the corresponding segments of $-p^\infty +\Gamma^\infty,$ which implies the claim.
\end{proof}

A special case of elastic flows satisfying \eqref{min_segments00000} for all times is the evolution of convex curves: we say an admissible polygonal curve $\Gamma$ is convex if $c_i\ne0$ for any segment $S_i$ of $\Gamma.$

\begin{corollary}[Convex evolution]\label{cor:convex_evolution}
Let $\norm$ be a crystalline anisotropy and let $\Gamma^0$ be a closed convex admissible polygonal curve. Let $\{\Gamma(t)\}_{t\in[0,T^\dag)}$ be the unique elastic flow starting from $\Gamma^0.$ Then $T^\dag=+\infty$ and there exists $p\in\R^2$ such that 
$$
\text{$K$-}\lim\limits_{t\to+\infty}\Gamma(t) = \p W_{\sqrt\alpha}^\norm(p),
$$
where $W_{\sqrt\alpha}^\norm(p)$ is the Wulff shape of radius $\sqrt{\alpha},$ centered at $p,$ which is a stationary curve by Example \ref{ex:evol_Wulff}.
\end{corollary}

\begin{proof}
Since $c_i\ne 0$ for any segments $S_i$ of $\Gamma(\cdot),$  by \eqref{length_bnd13}
$$
\min_{i} \sH^1(S_i(t)) \ge \min_i  \tfrac{\alpha c_\norm c_i^2\sH^1(F_i)^2}{\cF_\alpha(\Gamma^0)},\quad t\in[0,T^\dag).
$$
Thus, by Theorem \ref{teo:existence}, $T^\dag = +\infty$ and thus,  \eqref{min_segments00000} follows. Moreover, By Example \ref{ex:evol_Wulff}, $\Sigma:=\p W_{\sqrt\alpha}^\norm$ is stationary, and one can readily check that if the image of another polygonal curve $\Sigma'$ with index $m\ne0$ is $\Sigma$ (that is $\Sigma'$ covers $\Sigma$ $m$ times), then by Remark \ref{rem:station_lengith},  every stationary curve, whose segments are parallel to those of $\Sigma',$ must be the $m$-cover of a Wulff shape of radius $\sqrt{\alpha}.$ Since $\Gamma^0$ is closed, admissible and convex, its loops are constructed by a polygons with the same angles as the Wulff shapes (parallel to Wulff shapes) and  never disappear. By Theorem \ref{teo:conver_stationar_sol}, $\Gamma(t)\overset{K}{\to}\Gamma^\infty$ for some stationary curve $\Gamma^\infty,$ whose
index $m$ is the same as $\Gamma^0;$ and segments are parallel to those of $m$-cover of a Wulff shape, which implies it is itself a $m$-cover of a Wulff shape.
\end{proof}

\subsection{Long time behaviour: irregular case}

The proof of Theorem \ref{teo:conver_stationar_sol} heavily relies on the lower bound assumption
\eqref{min_segments00000}. The aim of this section is to prove the following long-time behaviour of the regular crystalline elastic flow, where we drop \eqref{min_segments00000}.

\begin{theorem}[Long-time behaviour, II]\label{teo:long_time_general}
Let $\norm$ be a crystalline anisotropy and $\{\Gamma(t)\}_{t\ge0}$ be the unique crystalline elastic flow starting from a closed admissible polygonal curve $\Gamma(0)$ with possible finitely many restarts.
Assume that there exists $T>0$ such that 
$$
\inf_{[T, T+m]}\min_{i=1,\ldots,n}\sH^1(S_i(t))>0\quad\text{for any $m>0$}
\qquad
\text{and}
\qquad 
\liminf\limits_{t\to+\infty}\,\min_{i=1,\ldots,n}\sH^1(S_i(t))=0.
$$
Then there exists a closed  admissible polygonal curve $\Gamma^\infty$ such that $\Gamma(t)\overset{K}{\to}\Gamma^\infty$ as $t\to+\infty.$ Moreover, $\Gamma^\infty$ is represented as a union of (possibly degenerate, i.e., zero-length) segments $\{S_i^\infty\}_{i=1}^n$ with 
$$
S_i(t)\overset{K}\to S_i^\infty\quad \text{and}\quad \sH^1(S_i(t))\to\sH^1(S_i^\infty)\quad\text{as $t\to+\infty$}\quad \text{for any $i=1,\ldots,n.$}
$$
Furthermore, $S_i^\infty$ is degenerate only if $c_i=0$ and 
if $\{\theta_i\}$ is the set of angles of $\Gamma(T),$ then 
\begin{equation}\label{gen_statS1}
c_i\sH^1(F_i)+ \alpha \Big(\frac{c_{i-1}^2\delta_{i-1}}{\sH^1(S_{i-1}^\infty)^2\sin\theta_i}+ \frac{c_{i}^2\delta_i[\cot
\theta_i+\cot \theta_{i+1}]}{\sH^1(S_{i}^\infty)^2}+\frac{c_{i+1}^2\delta_{i+1}}{\sH^1(S_{i+1}^\infty)^2\sin\theta_{i+1}}\Big)
= 0
\end{equation}
for any nondegenerate segment $S_i^\infty,$ where $\delta_j$ are  defined in \eqref{def:delta_is}. 
\end{theorem}

Comparing \eqref{gen_statS1} with the analogous condition of Definition \ref{def:reg_stationar}, we can refer to $\Gamma^\infty$ as a \emph{generalized stationary curve}.

The  remaining  of the section is devoted to the proof of the theorem. We follow the arguments of Theorem \ref{teo:conver_stationar_sol}, carefully inspecting the situations where the lower bound assumption \eqref{min_segments00000} is addressed. Without loss of generality, we assume $T=0$ and $\Gamma^0:=\Gamma(T).$
\smallskip

{\it Step 1: definition of $\Gamma^\infty.$} 
As we observed earlier (see \eqref{ZgEW} and Corollary \ref{cor:segments_nonzero_curva}),
\begin{equation}\label{dfhvvbb}
\cF_\alpha(\Gamma^0) = \cF_\alpha(\Gamma(t)) + \sum_{i=1}^n \int_0^t \frac{|h_i'(s)|^2\sH^1(S_i(s))}{\norm^o(\nu_{S_i^0})}\,ds,
\end{equation}
and hence,
\begin{equation}\label{dusgatdc}
\sH^1(\Gamma(t)) = \sum_{i=1}^n \sH^1(S_i(t)) \le \frac{1}{c_\norm}\cF_\alpha(\Gamma(t)) \le \frac{1}{c_\norm}\cF_\alpha(\Gamma^0),\quad t\ge0,
\end{equation}
and 
\begin{equation*}
\inf_{t\ge0} \sH^1(S_i(t)) \ge \frac{\alpha c_\norm c_i^2\sH^1 (F_i)^2}{\cF_\alpha(\Gamma^0)},\quad i=1,\ldots,n.
\end{equation*}
Thus, we can choose a sequence $t_k\nearrow +\infty$ such that $h_i'(t_k)\to0$ as $k\to+\infty$ for all $i$ with $c_i\ne0$ and a sequence $(p_k)_k\subset\R^2$ such that 
$0\in p_k+\Gamma(t_k).$ In particular, each $p_k+\Gamma(t_k)$ stays in the disc $D$ of radius $2\cF_\alpha(\Gamma^0)/c_\norm $ (centered at the  origin) and hence, by the Kuratowski compactness of  connected compact sets \cite{Falconer:1985}, up to a not relabelled subsequence, $p_k+\Gamma(t_k) \overset{K}{\to} \Gamma^\infty$ as $k\to+\infty$ for some compact set $\Gamma^\infty\subset D.$
Since each $p_k+\Gamma(t_k)$ is parallel to $\Gamma^0,$ we can readily check that $\Gamma^\infty$ is an admissible polygonal curve, not necessarily parallel to $\Gamma^\infty,$ consisting of a union of $n$ (some of which possibly degenerate) segments $\{S_i\}$ with 
$$
p_k+S_i(t_k)\overset{K}{\to} S_i^\infty \quad\text{and}\quad \sH^1(S_i(t_k))\to \sH^1(S_i^\infty)\quad\text{as}\quad k\to+\infty\quad \text{for any $i=1,\ldots,n.$}
$$
Thus, applying the evolution equation \eqref{elastic_ode123} with $t=t_k$ and letting $k\to+\infty,$ we deduce  
\begin{equation}\label{anatomosha_la}
c_i\sH^1(F_i) + \alpha\Big[\tfrac{c_{i-1}^2\delta_{i-1}}{\sH^1(S_{i-1}^\infty)^2\sin\theta_i}+
\tfrac{c_{i}^2\delta_i[\cot \theta_i+\cot
\theta_{i+1}]}{\sH^1(S_i^\infty)^2}+\tfrac{c_{i+1}^2\delta_{i+1}}{\sH^1(S_{i+1}^\infty)^2\sin\theta_{i+1}}\Big] =
0
\end{equation}
for any $i\in\{1,\ldots,n\}$ with $\sH^1(S_i^\infty)>0.$ Thus, 
$\Gamma^\infty$ is a generalized stationary curve.
\smallskip

{\it Step 2: properties of generalized stationary curves.} 
Let us call any solution $\bar \Gamma=\cup_{i=1}^n\bar S_i$ of \eqref{anatomosha_la} (applied with nondegenerate segments $\bar S_i$) a generalized stationary curve. We define its energy as usual,
$$
\cF_\alpha(\bar \Gamma) = \sum_{i=1}^n \int_{\bar S_i} \norm^o(\nu_{\bar S_i}) (1+\alpha \,[\kappa_{\bar S_i}^\norm]^2)d\sH^1 =
\sum_{i=1}^n \norm^o(\nu_{\bar S_i}) \Big(\sH^1(\bar S_i) + \frac{\alpha c_i^2\sH^1(F_i)^2}{\sH^1(\bar S_i)} \Big)
$$
which is well-defined also for degenerate segments, setting it as zero energy.

Generalized stationary curves $\Gamma^\infty$ have the same energy dissipation property as standard stationary curves (in the sense of Definition \ref{def:reg_stationar}).

\begin{proposition}\label{prop:stationer_gener}
Let $\bar \Gamma:=\cup_{i=1}^n \bar S_i$ be any generalized stationary curve  solving the system \eqref{anatomosha_la} (for instance, $\bar \Gamma=\Gamma^\infty$). Then for any polygonal curve $\Gamma:=\cup_{i=1}^n 
S_i$ parallel to $\bar\Gamma,$
\begin{equation}\label{diff_ener_genene}
\cF_\alpha(\Gamma) - \cF_\alpha(\bar \Gamma) = \alpha \sum_{i=1}^{n} \frac{ c_i^2 \delta_i}{\sH^1(\bar S_i)^2 \sH^1(
S_i)} \,\Big(\sH^1(S_i) - \sH^1(\bar S_i)\Big)^2.
\end{equation}
Moreover, $\sH^1(\bar S_i) = \sH^1(S_i^\infty)$ for all segments with $c_i\ne0,$ where $\Gamma^\infty=\cup_iS_i^\infty$ is the curve  obtained in step 1.
\end{proposition}

The proof runs along the same lines of Proposition \ref{prop:stationer_gener}; the degeneracy of segments does not create a problem here because of the presence of $c_i$ in the numerator. Moreover, in obtaining the identity \eqref{diff_ener_genene}, we can use both $\bar \Gamma$ and $\Gamma^\infty$ in place of $\Gamma,$  which yields $\cF_\alpha(\Gamma^\infty) = \cF_\alpha(\bar \Gamma),$ and thus $\sH^1(\bar S_i) = \sH^1(S_i^\infty)$ whenever $c_i\ne0.$
\smallskip

{\it Step 3: a Lojasiewicz-Simon-type  inequality.} In this step we establish an analogue of Proposition \ref{prop:loja_simon_ineq}. 
To this aim, let $\bar \Gamma:=\cup_{i=1}^n\bar S_i$ be any generalized stationary
curve,  solving \eqref{anatomosha_la}. For any admissible curve $\Gamma=\cup_{i=1}^nS_i,$ parallel to $\bar\Gamma,$ let us write 
$
H(S_i,\bar S_i)
$
to define the distance vector from $S_i$ to $\bar S_i$ in case $\bar S_i$ is nondegenerate, and the distance vector from $S_i$ to the straight line $\ell_i,$ passing through $\bar S_i$ and parallel to $S_i$ in case $\bar S_i$ is degenerate. We also set as usual $h_i:=\scalarp{H(S_i,\bar S_i)}{\nu_{S_i}}.$

\begin{proposition}[Lojasiewicz-Simon-type inequality, II]\label{prop:loja_simon_ineqII}
Let 
$\bar \Gamma:=\cup_{i=1}^n\bar S_i$ be a generalized stationary curve and fix any $\epsilon>0$ satisfying
$$
\epsilon< 2^{-10}\min_{\sH^1(\bar S_i)>0} \frac{\sH^1(\bar S_i)}{\frac{1}{|\sin\theta_i|} + |\cot\theta_i+\cot\theta_{i+1}|+\frac{1}{|\sin\theta_{i+1}|}}.
$$ 
There exist constants $C,\beta>0$ such that, given any polygonal curve $\Gamma:=\cup_{i=1}^n S_i,$ parallel to $\bar\Gamma$ with
\begin{equation}\label{maosfb}
\min_{c_i\ne0} \sH^1(S_i) \ge \epsilon\quad\text{and}\quad 
\max_{i\in J} |H(S_i,  \bar S_i)| < \epsilon,
\end{equation}
where 
$$
J:=\{i\in\{1,\ldots,n\}:\,\,  \text{either $c_i\ne 0$ or $c_i=0$ and $c_{i-1}^2+c_{i+1}^2\le 1$}\},
$$
one has 
\begin{multline}\label{loja_simoII}
|\cF_\alpha(\Gamma) - \cF_\alpha(\bar \Gamma)|^\beta
\le C \Big(\alpha \max_{i=1,\ldots,n} \, c_i^2\delta_i\,\sH^1(S_i)\Big)^\beta \times\\
\times\sum_{\sH^1(\bar S_i)>0} \Big|c_i\sH^1(F_i)+ \alpha \Big(\frac{c_{i-1}^2\delta_{i-1}}{\sH^1( S_{i-1})^2\sin\theta_i}+
\frac{c_{i}^2\delta_i[\cot \theta_i+\cot \theta_{i+1}]}{\sH^1( S_{i})^2}+\frac{c_{i+1}^2\delta_{i+1}}{\sH^1(
S_{i+1})^2\sin\theta_{i+1}}\Big)\Big|.
\end{multline}
\end{proposition}

\begin{proof}
Let $U$ be the collection of all $h:=(h_1,\ldots,h_n)$ with $h_i\in(-\epsilon,\epsilon),$ $i\in J,$ for which there exists a unique associated admissible polygonal curve $\Gamma,$ parallel to $\bar \Gamma,$ satisfying 
\begin{equation*}
\min_{c_i\ne0} \sH^1(S_i) \ge\epsilon\quad\text{and}\quad  h_i=\scalarp{H(S_i,\bar S_i)}{\nu_{S_i}}\quad \text{for all $i=1,\ldots,n.$}
\end{equation*}
Notice that $U$ is a bounded open set. 
Indeed, the boundedness of $\bar\Gamma$ and admissibility and parallelness 
conditions force the coordinates $h_i$ of $h\in U$ with $i\notin J$ 
to belong to the interval $(-2\sH^1(\bar\Gamma),2\sH^1(\bar\Gamma)).$ 
Moreover, 
for any $\tilde h\in U$ and associated polygonal curve $\tilde \Gamma,$ 
applying Lemma \ref{lem:construc_paral_curve} we can find  
$\eta>0$ such that for any $h\in\R^n$ with $|\tilde h-h|<\eta$ 
(i.e., $h\in B_\eta(\tilde h)$) there exists a unique polygonal 
curve $\Gamma$ parallel to $\tilde \Gamma$ (and hence to $\bar\Gamma$), satisfying $h_i-\tilde h_i:=\scalarp{H(S_i,\tilde S_i)}{\nu_{S_i}}.$ Since $\nu_{S_i} = \nu_{\tilde S_i}$ for any $i$ and $\tilde h\in U,$ it follows $h_i=\scalarp{H(S_i,\bar S_i)}{\nu_{S_i}}.$ Moreover, since $\tilde h\in (-\epsilon,\epsilon)^n,$ possibly decreasing $\eta,$ we may assume $h\in (-\epsilon,\epsilon)^n,$ and hence, $h\in U.$ Thus, $B_{\eta}(\tilde h)\subset U,$ i.e., $U$ is open. 

Consider the real-analytic map $g:=(g_1,\ldots,g_n)$ defined in $U$ as
$$
g_i(h_1,\ldots,h_n) := \frac{l_i}{\sH^1(S_i)} = \frac{l_i}{ \sH^1(\bar S_i) -
\Big(\frac{h_{i-1}}{\sin\theta_i}+h_i[\cot\theta_i+\cot\theta_{i+1}]+\frac{h_{i+1}}{\sin\theta_{i+1}}\Big) },\quad
h=(h_1,\ldots,h_n)\in U,
$$
where $l_i=1$ if $\sH^1(\bar S_i)>0$ and $l_i=0$ if $\sH^1(\bar S_i)=0.$ 

Let us show that $g$ extends real analytically to $\cl{U}.$ Indeed, fix any $h\in\cl{U}$ and consider any sequence $U\ni h^k \to h.$ Now if we take an index $i$ with $c_i\ne0,$ then by the definition of $U,$ 
$$
\sH^1(\bar S_i) - \Big(
\frac{h_{i-1}^k}{\sin\theta_i} 
+
h_i^k[\cot\theta_i+\cot\theta_{i+1}]
+
\frac{h_{i+1}^k}{\sin\theta_{i+1}}\Big) =
\sH^1(S_i^k)\ge \epsilon, 
$$
where $\{S_j^k\}$ are the segments of the curve $\Gamma^k,$ associated to $h^k.$ Thus, we can uniquely extend $g_i$ real-analytically to a small neighborhood of $h.$ On the other hand, let $i$ be such that $c_i=0$ with $\sH^1(\bar S_i)>0.$ By admissibility, $S_i^k$ has zero $\norm$-curvature and the segments $S_{i-1}^k$ and $S_{i+1}^k$ are parallel. Since $\bar\Gamma$ is admissible, $|H(\bar S_{i-1},S_{i+1} )| \ge \tilde c \sH^1(\bar S_i)$ for some $\tilde c>0$ depending only on the  angles $\{\theta_j\}.$ Moreover, by the definition of $J,$  both $i-1,i+1\in J.$ In particular, by the second assumption in \eqref{maosfb},   $h_{i-1},h_{i+1}\in(-\epsilon,\epsilon)$ and therefore, recalling the smallness of $\epsilon$ depending only on the lengths of the nondegenerate segments of $\bar\Gamma$ and the angles of $\Gamma^0,$ we conclude 
$$
\sH^1(S_i^k) \ge |H(S_{i-1}^k,S_{i+1}^k)| \ge |H(\bar S_{i-1},\bar S_{i+1})| - |h_{i-1}| - |h_{i+1}|  \ge \tilde c \sH^1(\bar S_i) - |h_{i-1}| - |h_{i+1}| \ge \epsilon.
$$
Thus, again $g_i$ is extended real-analytically to a neighborhood of $h.$
Finally, in case $c_i=0$ and $\sH^1(\bar S_i)=0,$ $g_i\equiv0,$ which is real-analytic in $\R^n.$

Notice that for any $h\in\cl{U}$ we can define a unique admissible curve $\Gamma:=\cup_{i=1}^nS_i$ (for instance, defined by a Kuratowski limit of the curves $\Gamma^k,$ associated to an approximating sequence $U\ni h^k\to h$) satisfying 
$$
h_i=\scalarp{H(S_i,\bar S_i)}{\nu_{S_i}}:=\lim_{k\to+\infty} \scalarp{H(S_i^k,\bar S_i)}{\nu_{S_i^k}}
$$
with $\nu_{S_i} = \nu_{S_i^k}$ for all $i$ and $k.$ In particular, $\bar h=(0,\ldots,0),$ associated to our $\bar\Gamma,$ belongs to the closure of $U.$

Since $U$ is relatively compact, by \cite[Section IV.2]{Loja:1995}, both $O:=g(U)$ and $\cl{O}=g(\cl{U})$ are subanalytic bounded sets. 
Now, as in the proof of Proposition \ref{prop:loja_simon_ineq}, consider the continuous function $f:\cl{O}\to\R$ of $n$ variables  
$$
f(x_1,\ldots,x_n) := \sum_{i=1}^n l_i \Big|c_i\sH^1(F_i)+ \alpha \Big(\tfrac{c_{i-1}^2\delta_{i-1}}{\sin\theta_i} x_{i-1}^2
+ c_{i}^2\delta_i[\cot \theta_i+\cot \theta_{i+1}]\,x_i^2+\tfrac{
c_{i+1}^2\delta_{i+1}}{\sin\theta_{i+1}}\,x_{i+1}^2\Big)\Big|,
$$
where as usual $x_0:=x_n$ and $x_{n+1}:=x_1.$ Being a sum of  absolute values of real-analytic functions, $f$
is subanalytic. Thus, we can apply the Lojasiewicz inequality in \cite[Section IV.9]{Loja:1995} to find
positive constants $\beta'$ and $C'$ (depending on $O$ and thus, on $\epsilon$) such that
\begin{equation}\label{loja_rengeo}
|f(x)| \ge C' \dist(x,\{f=0\})^{\beta'}.
\end{equation}
Consider any $z=(z_1,\ldots,z_n)\in \{f=0\}$ and let $z^k\in O$ be such that $z^k\to z$ as $k\to+\infty.$ For each $k\ge1$ take  $h^k:=(h_1^k,\ldots,h_n^k)\in U$ such that $g(h^k) = z^k.$ By relative compactness of $U,$ up to a not relabeled subsequence, $h^k\to \tilde h\in \cl{U}.$ Clearly, $z=g(\tilde h)$ and as we
observed earlier, we can find a unique associated polygonal curve $\tilde \Gamma:=\cup_{i=1}^n\tilde S_i$ satisfying $\tilde h_i= \scalarp{H(\tilde S_i,\bar S_i)}{\nu_{\tilde S_i}}$ 
and $\sH^1(\tilde S_i) =
\frac{1}{g_i(\tilde h)} = \frac{1}{z_i}$ whenever $\sH^1(\bar S_i)>0,$ i.e., $l_i=1.$
The last relation, the definition of $f$ and the equation $f(z)=0$ imply that $(\frac{l_1}{\sH^1 (\tilde S_1)},\ldots,\frac{l_n}{\sH^1(\tilde S_n)})$ satisfy \eqref{anatomosha_la} and thus, $\tilde \Gamma$ is a generalized 
stationary curve. 
Since $\bar \Gamma$ is also stationary, by Proposition \ref{prop:stationer_gener},  
\begin{equation}\label{iu78fh7}
\sH^1(\bar S_i) = \sH^1(\tilde S_i) = \frac{1}{z_i} \quad\text{whenever $i\in\{1,\ldots,n\}$ with $c_i\ne0.$} 
\end{equation}

Now, fix any $h\in U$ and associated $\Gamma,$ parallel to $\Gamma^0.$ Clearly, $\Gamma$ satisfies \eqref{maosfb}. By \eqref{diff_ener_genene}, 
\begin{equation}\label{oakdu7bn}
|\cF_\alpha(\Gamma) - \cF_\alpha(\bar \Gamma)| = \alpha
\sum_{i=1}^n  c_i^2 \delta_i\, \sH^1(S_i) \,\Big(\tfrac{1}{\sH^1(\bar S_i)} - \tfrac{1}{\sH^1(S_i)}\Big)^2 \le \alpha  \max_{1\le i\le n} c_i^2 \delta_i\, \sH^1(S_i)
\sum_{c_i\ne 0}\Big(\tfrac{1}{\sH^1(\bar S_i)} - \tfrac{1}{\sH^1(S_i)}\Big)^2.
\end{equation}
Set 
$
x:=(\frac{l_1}{\sH^1(S_1)},\ldots,\frac{l_n}{\sH^1(S_n)});
$
by the definition of $g,$ one has $x\in O.$ Let $\bar z\in \{f=0\}$ be such that
$\dist(x,\{f=0\}) = |x-\bar z|.$ 
As we observed above (see also \eqref{iu78fh7}), $z_i = \frac{1}{\sH^1(\bar S_i)}$ whenever $c_i\ne0.$ Thus,
we can represent and further estimate \eqref{oakdu7bn} as
\begin{equation*}
|\cF_\alpha(\Gamma) - \cF_\alpha(\bar \Gamma)| \le  \alpha \max_{1\le i\le n} c_i^2 \delta_i\, \sH^1(S_i)
\sum_{c_i\ne 0}|z_i-x_i|^2  \le \alpha
\max_{1\le i\le n}  c_i^2 \delta_i\sH^1(S_i)\,\dist(x,\{f=0\})^2.
\end{equation*}
Now recalling \eqref{loja_rengeo} we deduce 
$$
|\cF_\alpha(\Gamma) - \cF_\alpha(\bar \Gamma)|^{\frac{\beta'}{2}} \le \Big(\alpha \max_{1\le i\le n}  c_i^2
\delta_i\sH^1(S_i)\Big)^{\frac{\beta'}{2}}\, \frac{|f(x)|}{C'},
$$
which is \eqref{loja_simo_tengis} with $\beta=\beta'/2>0$ and $C=1/C'>0.$
This concludes the proof of Proposition \ref{prop:loja_simon_ineqII}.
\end{proof}

{\it Step 4: conclusion of the proof of Theorem \ref{teo:long_time_general}.} 
We follow the arguments of Theorem \ref{teo:conver_stationar_sol}, but some care is required as we do not have uniform lower bound of \eqref{min_segments00000}. 
For any $k$ let 
$$
\tilde \Gamma^k(t):=p_k+\Gamma(t)\quad\text{and}\quad 
\tilde h_i^k := \scalarp{H(S_i^0,\tilde S_i^k(t))}{\nu_{S_i^0}},\quad i=1,\ldots,n.
$$ 
Clearly,  $\tilde h_i^k = h_i + \scalarp{p_k}{\nu_{S_i^0}}$ also solves the same evolution equation \eqref{elastic_ode123} as $h_i.$ 

Let $C,\epsilon,\beta>0$ and the set $J$ be given by Proposition
\ref{prop:loja_simon_ineqII} applied with $\bar \Gamma:=\Gamma^\infty.$ 
There is no loss of generality in assuming
\begin{equation}\label{ervertwe} 
\frac{1}{c_\norm}\,\cF_\alpha(\Gamma^0)
\sH^1(\tilde S_i^k(t)) = \sH^1(S_i^k(t)) \ge \frac{\alpha c_\norm c_i^2\sH^1(F_i)^2}{\cF_\alpha(\Gamma^0)} >\epsilon\quad\text{for any $t\ge0$}
\end{equation}
for any $i$ with $c_i\ne0.$ 
For any $k\ge1$ let $I_k\subset(0,+\infty)$ be the set of all $t$ satisfying 
$$
\max_{i\in J}\,|H(\tilde S_i^k(t),S_i^\infty)|<\epsilon.
$$
By the definition of $\Gamma^\infty,$ the set $I_k$ contains $t_k$ for all  sufficiently large $k.$ As $\tilde h_i^k$ is continuous, $I_k$ is open. 

For $\sigma>0$ satisfying $\sigma^2\beta+\sigma=1,$ consider the function
$$
\ell(t):=|\cF_\alpha(\tilde \Gamma^k(t)) - \cF_\alpha(\Gamma^\infty)|^{\sigma},\quad t\in I_k.
$$
Note that by the Lojasiewicz-Simon inequality \eqref{loja_simo_tengis}, $\ell(t_k)\to0$ as $k\to+\infty. $
Moreover, by the energy dissipation equality \eqref{dfhvvbb} applied with $\{\tilde \Gamma^k(t)\}$ we have 
\begin{align*}
\ell'(t) = & -\sigma 
\ell(t)^{-\frac{\sigma-1}{\sigma}}\,\sum_{i=1}^n\Big |\frac{d}{dt}{\tilde h_i^k}(t)\Big|^2\sH^1(\tilde S_i^k(t)).
\end{align*}
By obvious estimates,
$$
\sum_{i=1}^n\Big |\frac{d}{dt}{\tilde h_i^k}\Big|^2\sH^1(\tilde S_i^k) = 
\sum_{i=1}^n\Big |\frac{d}{dt}{\tilde h_i^k}\sH^1(\tilde S_i^k)\Big|^2\tfrac{1}{\sH^1(\tilde S_i^k)} 
\ge \frac{1}{n\,\max_i \sH^1(\tilde S_i^k)} \Big(\sum_{i=1} ^n \Big|\frac{d}{dt}{\tilde h_i^k}\Big|\sH^1(\tilde S_i^k)\Big)^2
$$
and by \eqref{elastic_ode123} applied with $\tilde h_i^k,$ the relations \eqref{ervertwe}  and the Lojasiewicz-Simon inequality \eqref{loja_simoII},
$$
C\Big(\max_{1\le i\le n} \alpha c_i^2 \delta_i \sH^1(\tilde S_i^k(t))\Big)^\beta  \sum_{i=1}^n \Big|\frac{d}{dt}\tilde h_i^k(t)\Big|\sH^1(\tilde S_i^k(t)) \ge |\cF_\alpha(\tilde \Gamma^k(t)) - \cF_\alpha(\Gamma^\infty)|^\beta = \ell(t)^{\sigma\beta}. 
$$
Thus, using \eqref{dusgatdc} as $\sH^1(S_i(t)) \le \cF_\alpha(\Gamma^0)/c_\norm$ and the definition of $\sigma$ we conclude 
\begin{equation}\label{domashniy_kaputt}
-\ell'(t) \ge C_1\sum_{i=1}^n \Big|\frac{d}{dt}\tilde h_i^k(t)\Big|\sH^1(\tilde S_i^k(t)),\quad t\in I_k,
\end{equation}
for some constant $C_1>0$ depending only on $\norm,$ $n,$ $\alpha,$ $\cF_\alpha(\Gamma^0)$ and $\sigma.$

Fix  $\gamma>0$ (to be chosen shortly) and $\eta\in(0,\gamma\epsilon),$ 
and let $k\ge1$ be so large that 
\begin{equation*}
\max_{1\le i\le n} |H(\tilde S_i^k(t_k),S_i^\infty)|<\eta \quad\text{and}\quad |\cF_\alpha(\tilde
\Gamma^k(t_k)) - \cF_\alpha(\Gamma^\infty)|^{\sigma}< \eta.
\end{equation*}
Let $\bar r>t_k$ be the supremum of all $r>t_k$ such that $(t_k,r)\subset I_k.$ We claim that $\bar r=+\infty.$  Indeed, if $\bar r<+\infty,$ by the continuity of $t\mapsto
|H(S_i^\infty,\tilde S_i^j(t))|,$ we would have
\begin{equation}\label{opo8jte5}
|H(\tilde S_i^k(\bar r), S_i^\infty)| = \epsilon\quad\text{for some $i\in J.$}
\end{equation}
Thus, from \eqref{domashniy_kaputt} and \eqref{ervertwe} we get 
$$
|\tilde h_i^k(\bar r) - \tilde h_i^k(t_k)| \le \int_{t_k}^{\bar r} \Big|\frac{d}{dt} \tilde h_i^k\Big|dt \le - \frac{1}{C_1\epsilon}  \int_{t_k}^{\bar r} \ell'(t)dt \le \frac{|\cF_\alpha(\tilde \Gamma^k(t_k)) - \cF_\alpha(\Gamma^\infty)|^\sigma}{C_1\epsilon } <\frac{\eta}{C_1\epsilon}.
$$
On the other hand, by the choice of $\eta$ and the relation \eqref{opo8jte5}  
\begin{equation}\label{alles_hofer_alles}
|H(\tilde S_i^k(\bar r),S_i^\infty)| \le |H( \tilde S_i^k(t_k), S_i^\infty)| + |H(\tilde S_i^k(t_k), \tilde
S_i^k(\bar r))| <\eta + |\tilde h_i^k(t_k) - \tilde h_i^k(\bar r)|< \Big(1+\tfrac{1}{C_1\epsilon}\Big)\eta<\epsilon 
\end{equation}
provided for instance $\gamma<(1+\tfrac{1}{C_1\epsilon})^{-1}.$ 

On the other hand, if $c_i=0$ with $c_{i-1}^2+c_{i+1}^2\le1$ (i.e., either $c_{i-1}=c_{i+1}=0$ or $c_{i-1}=0\ne c_{i+1}$ or $c_{i-1}\ne0=c_{i+1}$), then by the evolution equation \eqref{elastic_ode123}, $\tilde h_i^k$ satisfies 
$$
\frac{d}{dt}\tilde h_i^k = 
\begin{cases}
0 &  \text{if $c_{i-1}=c_{i+1+0},$}\\
-\frac{\alpha c_{i-1}\delta_{i-1}}{\sH^1(\tilde S_i^k) \sH^1 (\tilde S_{i-1}^k)^2\sin\theta_i} & \text{if $c_{i+1}=0\ne c_{i-1},$} \\
-\frac{\alpha c_{i+1}\delta_{i+1}}{\sH^1(\tilde S_i^k) \sH^1 (\tilde S_{i+1}^k)^2\sin\theta_{i+1}} & \text{if $c_{i-1}=0\ne c_{i+1}.$} 
\end{cases}
$$
Thus, either $\tilde h_i^k\equiv0,$ or recalling \eqref{ervertwe} we conclude 
\begin{equation}\label{ahcy6vbn}
\frac{1}{\sH^1(\tilde S_i^k)} \le \tilde C_2 \Big|\frac{d}{dt}\tilde h_i^k\Big|
\end{equation}
for some $\tilde C_2>0$ depending only on $\alpha,$ $\norm$ and  $\cF_\alpha(\Gamma^0).$
Then by Cauchy inequality and \eqref{ahcy6vbn},
$$
\Big|\frac{d}{dt}\tilde h_i^k\Big| \le \tilde C_2\Big|\frac{d}{dt}\tilde h_i^k\Big|^2 \sH^1(\tilde S_i^k) + \frac{1}{4\tilde C_2\sH^1(\tilde S_i^k)} \le \tilde C_2\Big|\frac{d}{dt}\tilde h_i^k\Big|^2 \sH^1(\tilde S_i^k) + \frac{1}{4}\Big|\frac{d}{dt}\tilde h_i^k\Big|.
$$
Therefore, 
\begin{equation}\label{zertifikat}
|\tilde h_i^k(\bar r) - \tilde h_i^k(t_k) | \le \int_{t_k}^{\bar r} \Big|\frac{d}{dt}\tilde h_i^k\Big| \le \frac{4\tilde C_2}3 \int_{t_k}^{+\infty} \Big|\frac{d}{dt}\tilde h_i^k\Big|^2 \sH^1(\tilde S_i^k)\,ds.
\end{equation}
Since $\tilde \Gamma(t)$ is a constant  translation of $\Gamma(t),$ recalling the energy dissipation equality \eqref{dfhvvbb} we conclude
$$
\int_{t_k}^{+\infty} \Big|\frac{d}{dt}\tilde h_i^k\Big|^2 \sH^1(\tilde S_i^k)\,ds = \int_{t_k}^{+\infty} \Big|\frac{d}{dt} h_i\Big|^2 \sH^1(S_i)\,ds <\eta
$$
for all large enough $k,$ depending only on $\eta.$ Then for such $k,$
\eqref{zertifikat} implies 
$$
|\tilde h_i^k(\bar r) - \tilde h_i^k(t_k) | \le \frac{4\tilde C_2}{3}\eta,
$$
and thus, as in \eqref{alles_hofer_alles} we get 
$$
|H(\tilde S_i^k(\bar r),S_i^\infty)| < \eta + |\tilde h_i^k(t_k) - \tilde h_i^k(\bar r)|< \Big(1+\tfrac{4\tilde C_2}{3}\Big)\eta<\epsilon 
$$
provided for instance $\gamma<(1+\tfrac{4\tilde C_2}{3})^{-1}.$  These contradictions imply that \eqref{opo8jte5} is impossible, and hence $\bar r=+\infty.$

These observations show that 
in view of \eqref{hstdzfgg} and \eqref{fahstdzvb} as well as the choice of $\eta$ and $k,$ for any $t>t_k$ we have
\begin{equation*}
\sum_{i\in J}\, |h_i(t) -  h_i(t_k)| = \sum_{i\in J}|\tilde h_i^k(t) - \tilde h_i^k(t_k)| \le \eta.
\end{equation*}
This estimate shows that the families $\{h_i(t)\}_{t>0},$ $i\in J,$ are fundamental as $t\to+\infty.$ Hence, there exists $\bar h_i:=\lim_{t\to+\infty} h_i(t)$ for any $i\in J.$
In particular, applying \eqref{hstdzfgg} at $t=t_k$ and letting $k\to+\infty$ we find that the limit of
$\scalarp{p_k}{\nu_{S_i^0}}$ exists for all $i\in J.$ Since $\Gamma^0$ is closed, it has at least two sides with nonzero $\norm$-curvature and nonparallel normals. Thus, the sequence $p_k$ also converges to some $p^\infty$ as $k\to+\infty.$ Finally, we show that
\begin{equation*}
\Gamma(t)\overset{K}{\to} -p^\infty + \Gamma^\infty\quad \text{as}\quad t\to+\infty. 
\end{equation*}
Indeed, for any $i\in J$ and $t>t_k,$
\begin{multline*}
|H(-p^\infty + S_i^\infty, S_i(t))| \le |H(-p_k+S_i^\infty,-p^\infty+S_i^\infty| + |H(S_i^\infty, p_k+S_i(t_k))|
+ |H(p_k+S_i(t_k), p_k+S_i(t))| \\
\le |p_k-p^\infty| + |H(S_i^\infty, p_k+S_i(t_k))| + |h_i(t_k)-\bar h_i(t)| \to0
\end{multline*}
as $t\to+\infty$ and $k\to+\infty.$

Now fix any index $i\notin J.$ Then $c_i=0$ and $|c_{i-1}|=|c_{i+1}|=1.$ In particular, $i-1,i+1\in J$ and hence, $h_{i-1}(t)\to \bar h_{i-1}$ and $h_{i+1}(t)\to \bar h_{i+1}$ as $t\to+\infty.$ As in  subcase 2.3 in the proof of Theorem \ref{teo:existence},  $S_{i-1}(t)$ and $S_{i+1}(t)$ are parallel, $\nu_{S_{i-1}(t)} = \nu_{S_{i+1}(t)} = \nu_{S_{i-1}^0} = \nu_{S_{i+1}^0},$ and we distinguish two cases:

\begin{itemize}
\item $\bar h_{i+1}\ne \bar h_{i-1} + \scalarp{H(S_{i-1}^0,S_{i+1}^0)}{\nu_{S_{i-1}^0}}.$  Then
\begin{align*}
\sH^1(S_i(t)) \ge & |H(S_{i-1}(t),S_{i+1}(t))| = |h_{i+1}(t) - h_{i-1}(t) - H(S_{i-1}(t),S_{i+1}(t))| \\
> & \frac{1}{4}\Big|\bar h_{i+1} -\bar h_{i-1} - \scalarp{H(S_{i-1}^0,S_{i+1}^0)}{\nu_{S_{i-1}^0}}\Big|=:\tilde \epsilon
\end{align*}
provided $t>0$ is large enough.  Since $\sH^1(S_i) = \sH^1(\tilde S_i^k),$ recalling \eqref{domashniy_kaputt}, as above we can show that $\{\tilde h_i^k(t)\}$ (and hence $\{h_t(t)\}$) is fundamental as $t\to+\infty.$ In particular, $h_i(t)\to \bar h_i$ as $t\to+\infty$ for some $\bar h_i\in\R.$

\item $\bar h_{i+1} = \bar h_{i-1} + \scalarp{H(S_{i-1}^0,S_{i+1}^0)}{\nu_{S_{i-1}^0}}.$ In this case the segment $S_i(t)$ vanishes as $t\to+\infty$ and the segments $S_{i-1}(t)$ and $S_{i+1}(t)$ Kuratowski converge to a subset of the same straight line $L_i$ satisfying $\bar h_{i+1} = \scalarp{H(S_{i+1}^0,L_i)}{\nu_{S_i^0}}.$ 
\end{itemize}

As we have seen at the end of the proof of Theorem \ref{teo:existence} (c), these observations already suffice to conclude that $\Gamma(t)\overset{K}{\to} \bar\Gamma$ as $t\to+\infty$ for some closed set $\bar\Gamma\subset\R^2.$ Since $\Gamma(t_k)\overset{K}{\to} -p^\infty+\Gamma^\infty,$ it follows that $\bar\Gamma =  -p^\infty+\Gamma^\infty.$

This completes the proof of Theorem \ref{teo:long_time_general}. \hfill \qed 

\section{Special solutions in case of square anisotropy}
\label{sec:square_anisotropy}

In this section we assume that $W^\norm$ is the square $[-1,1]^2$ and classify some special solutions of the  crystalline elastic flow. 

\subsection{Classification of stationary curves}

\begin{theorem}[Stationary curves in the square anisotropy]\label{teo:stationary_square}
Let 
$\Gamma:=\cup_{i=1}^n S_i$ be a  stationary polygonal curve. 

\begin{itemize}
\item {\bf Unbounded case.} Suppose that $S_1$ and $S_n$ are half-lines. Then up to translations, horizontal and
vertical reflections and  rotations by a multiple of $90^o,$ $\Gamma$ can be:
\begin{itemize}
\item a staircase, i.e., $n\ge2$ is any integer, the angles of $\Gamma$ alternates, e.g., $\theta_i=\pi/2$ for all even indices, while $\theta_i=3\pi/2$ for all odd indices, and the segments have arbitrary positive lengths (see Fig. \ref{fig:unb_zigzag});
%
%
\begin{figure}[htp!]
\includegraphics[height=1.0cm, width=0.7\textwidth]{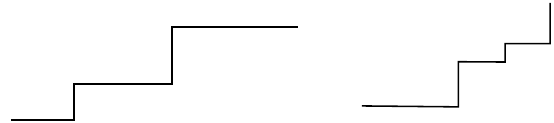}
\caption{\small Staircases.}\label{fig:unb_zigzag}
\end{figure}

\item a right-angle chain, i.e., is a union of $m\ge1$ right-angles of sidelength $\sqrt{2\alpha},$ formed by
$3m-1$ segments and two half-lines (i.e., $n=3m+1$); here $\sH^1(S_{3i+2})=\sH^1(S_{3i+3})= \sqrt{2\alpha}$ for all $0\le i\le
m-1,$ and the segments $S_{3i+1}$ with $0<i<m$ have arbitrary length and zero $\norm$-curvature, i.e.
$c_{3i+1}=0$ (see Fig. \ref{fig:unb_square});
\begin{figure}[htp!]
\includegraphics[width=0.7\textwidth]{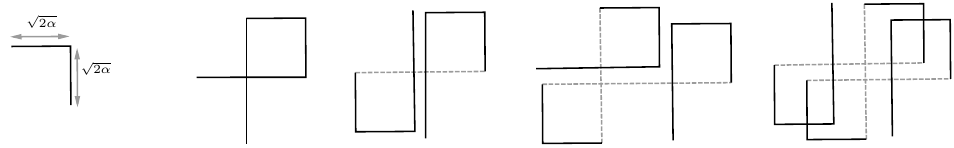}
\caption{\small A right-angle and unbounded right-angle chains with $m=1,2,3,4.$ Zero $\norm$-curvature segments are
depicted by dashed segments.}\label{fig:unb_square}
\end{figure} 

\item a double-right-angle chain, i.e., it is a union of $m\ge1$ double-right-angles consisting of a union of two
horizontal segments of length $a,b>0$ with
$$
\frac{1}{a^2}+\frac{1}{b^2}=\frac{1}{2\alpha}
$$
and one vertical segment of length $\sqrt{2\alpha},$ formed by $4m-1$ segments and two half-lines (i.e., $n=4m+1$); here
$\sH^1(S_{4i+3})=\sqrt{2\alpha}$ for all $0\le i\le m-1,$
$\sH^1(S_{8i+2})=\sH^1(S_{8i})=a$ for all $0\le i\le \intpart{m/2},$ $\sH^1(S_{8i+4})=\sH^1(S_{8i+6})=b$ for all
possible $i\ge0,$
and all segments $S_{4i+1}$ with $0<i<m$ have arbitrary length and zero $\norm$-curvature, i.e., $c_{4i+1}=0$
(see Fig. \ref{fig:unb_trap});
\begin{figure}[htp!]
\includegraphics[width=0.75\textwidth]{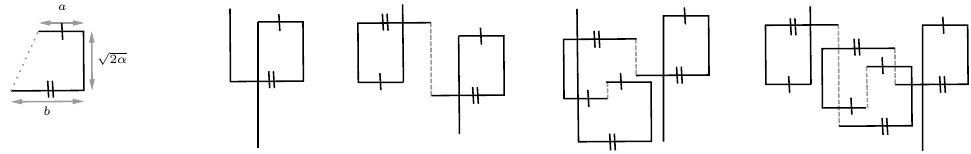}
\caption{\small A double-right-angle and unbounded double-right-angle chains with $m=1,2,3,4.$ Zero $\norm$-curvature
segments are depicted by vertical dashed segments.}\label{fig:unb_trap}
\end{figure}
\end{itemize}

\item {\bf Closed case.} Up to translations,   horizontal and vertical reflections and  rotations by a
multiple of $90^o,$ $\Gamma$ can be:
\begin{itemize}
\item a right-angle chain, i.e., it is a union of $2m\ge2$ right-angles of sidelength $\sqrt{2\alpha},$ formed by
$n=6m$ segments; here $\sH^1(S_{3i+1})=\sH^1(S_{3i+2})= \sqrt{2\alpha}$ for all $0\le i\le 2m-1,$ and the
segments $S_{3i}$ with $1\le i\le 2m$ have zero $\norm$-curvature, i.e. $c_{3i}=0$ (see Fig.
\ref{fig:bdd_square}). Morever, if $m=1,$ $\sH^1(S_3)=\sH^1(S_6)=2\sqrt{2\alpha},$ while if $m>1,$ all zero
$\norm$-curvature segments $S_{3i}$ with $1\le i<2m$ can have arbitrarily length and $S_{6m}$ is
uniquely defined;
\begin{figure}[htp!]
\includegraphics[width=0.6\textwidth]{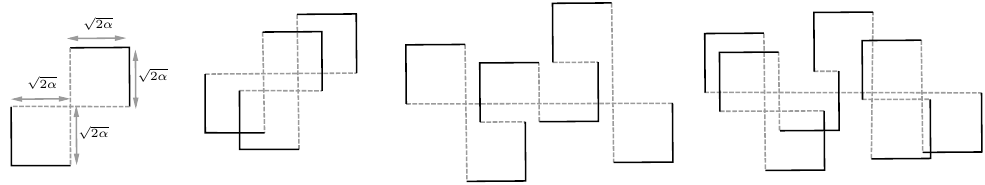}
\caption{\small Closed right-angle chains with $m=1,2,3,4.$ Zero $\norm$-curvature segments are depicted by dashed segments.}\label{fig:bdd_square}
\end{figure} 

\item a double-right-angle chain, i.e., it is a union of $2m\ge2$ double-right-angles consisting of a union of
two horizontal segments of length $\sqrt{4\alpha}$ and one vertical segment of length $\sqrt{2\alpha},$ formed by
$n=8m$ segments; here $\sH^1(S_{2i})=\sqrt{4\alpha}$ for all $1\le i\le 4m,$ $\sH^1(S_{4i+3})=\sqrt{2\alpha}$ for
all $0\le i< 2m,$ and all segments $S_{4i+1}$ with $0\le i<2m$ have zero $\norm$-curvature, i.e., $c_{4i+1}=0$
(see Fig. \ref{fig:bdd_recto}).
Moreover, $\Gamma$ lies in the strip $[-\sqrt{4\alpha},\sqrt{4\alpha}]\times\R,$ all vertical segments with 
transition number $c\in\{-1,0,1\}$ are located on the vertical line $\{c\sqrt{4\alpha},\}\times\R.$ In
particular, $m$ double-right-angles are contained in $[-\sqrt{4\alpha},0]\times\R,$ while the remaining $m$ are
contained in $[0,\sqrt{4\alpha}]\times\R.$
\begin{figure}[htp!]
\includegraphics[width=0.85\textwidth]{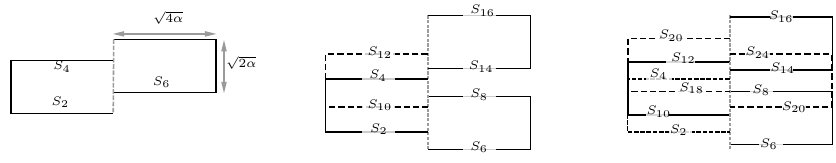}
\caption{\small Closed double-right-angle chains with $m=1,2,3.$ Zero $\norm$-curvature segments are depicted by
vertical dashed segments.}\label{fig:bdd_recto}
\end{figure}

\item a square of sidelength $\sqrt{4\alpha},$ i.e., a Wulff shape of radius $\sqrt{\alpha}$ (see Fig.
\ref{fig:squa1900})

\begin{figure}[htp!]
\includegraphics[width=0.2\textwidth]{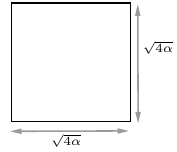}
\caption{\small A Wulff shape of radius $\sqrt{\alpha}.$ }\label{fig:squa1900}
\end{figure}

\end{itemize}

\end{itemize}

\end{theorem}

We have already seen in Example \ref{ex:evol_Wulff} that any Wulff shape of radius $\sqrt{\alpha}$ is a
stationary curve. Moreover, one can readily check that closed right-angle and double-right-angle chains have
$0$-index. Thus, the only stationary curve with a nonzero index is a square.

\begin{proof}
If $c_i=1,$ then $c_{i-1}$ and $c_{i+1}$ cannot be $0$ simultaneously. This prevents holes or hills in $\Gamma,$
see Fig. \ref{fig:hillhole} (a).
\begin{figure}[htp!]
\includegraphics[width=0.75\textwidth]{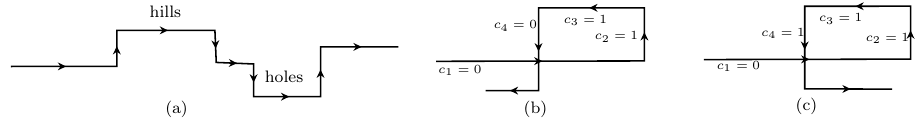}
\caption{\small} \label{fig:hillhole}
\end{figure}

First, assume that $c_1=c_2=0.$ Then by \eqref{hazsut}, $c_3=0.$ Repeating the same argument, we conclude $c_i=0$
for all $1\le i\le n,$ i.e., all segments have zero $\norm$-curvature. In this case, the curve $\Gamma^0$ is an
unbounded staircase, i.e, $n$ is any integer, $S_1$ and $S_n$ are half-lines, and angles $\theta_i$ of $\Gamma$
are alternatively $\pi/2$ and $3\pi/2,$ see Fig. \ref{fig:unb_zigzag}. Clearly, the length of the segments can
be arbitrary and any horizontal or vertical reflections, and   $\pm90^o$-rotations of such curves are also
stationary.

Next, assume that $c_1=0$ and $c_2=1,$ i.e., either $S_1$ is a half-line or $\Gamma$ is neither convex nor
concave near $S_1,$ but locally convex near $S_2.$ Then by admissibility and \eqref{hazsut} (applied with $i=2$), it is locally convex
also near $S_3,$ so that $c_3=1$ and $\sH^1(S_3) = \sqrt{2\alpha}$ (see Fig. \ref{fig:hillhole} (b) and (c)). Now there are two cases.

{\it Case 1: $c_4=0.$} In this case, necessarily, $c_5=c_6=-1,$ and thus, $\sH^1(S_5)=\sH^1(S_6)=\sqrt{2\alpha}.$
Then again by \eqref{hazsut}, $c_7=0,$ and we continue until we reach  $S_n.$ In view of this observation and
an induction argument, $c_{3i+1}=0$ for all $0\le i$ with $3i+1\le n$ and the segments $S_{3i+2}$ and $S_{3i+3}$
form a right-angle. Moreover, each segment of $0$ transition number is joined to two segments, one with positive and
the other with negative transition numbers.

\begin{itemize}
\item Assume that $\Gamma$ is unbounded, i.e., $S_1$ and $S_n$ are half-lines. As $c_1=c_n=0,$ we have
necessarily $n=3m+1$ for some $m\ge1.$ In this case, $\Gamma$ is an unbounded right-angle chain (see Fig.
\ref{fig:unb_square}).

\item Assume that $\Gamma$ is bounded. Let us group the segments as $(S_{3i+1},S_{3i+2},S_{3i+3})$ for each
$i\ge0,$ where $S_{n+k}=S_k.$ Then the triplet $(c_{3i+1},c_{3i+2},c_{3i+3})$ is either
$(0,1,1)$ or $(0,-1,1).$
Since $S_n,$ $S_1$ and $S_2$ are three consecutive segments of $\Gamma,$ and $c_1=0$ and $c_2=1,$ by
\eqref{hazsut} we have $c_n=-1.$ This shows $c_{n-1}=-1$ and $c_{n-2}=0,$ i.e., $n$ must be divisible by $3.$
Moreover, the triplets $(c_{3i+1},c_{3i+2},c_{3i+3})$ starts with $(0,1,1)$ and alternates with $(0,-1,1).$ As we
have seen, the last triplet is $(0,-1,1),$ and thus, the number of triplets must be even, i.e., $n$ is also even. This implies $n=6m$ for some $m\ge1$ and $\Gamma$ is a union of $2m$ right-angles. Moreover, all segments
with nonzero transition number have length $\sqrt{2\alpha},$ the segments $S_{6i+2},S_{6i+3}$ for $0\le i<m,$
forming $m$ right-angles, have positive $\norm$-curvature, equal to $\frac{1}{\sqrt{\alpha}},$ while the segments
$S_{6i+5},S_{6i+6}$ for $0\le i<m,$ forming the remaining $m$ right-angles, have negative $\norm$-curvature,
equal to $-\frac{1}{\sqrt{\alpha}}.$ Furthermore, all, but one, segments with zero transition number can have
arbitrary length  and
the exceptional segment is necessary to close the curve and make it admissible, see Fig. \ref{fig:bdd_square}.
\end{itemize}

{\it Case 2: $c_4=1.$} In this case, by \eqref{hazsut}, the segments $S_2$ and $S_4$ must satisfy 
\begin{equation}\label{hcvbb}
\frac{1}{\sH^1(S_2)^2} + \frac{1}{\sH^1(S_4)^2}=\frac{1}{2\alpha}
\end{equation}
so that both lengths are larger than $\sqrt{2\alpha}.$ Since $\sH^1(S_3)=\sqrt{2\alpha}$ and $c_4=1,$ again by
\eqref{hazsut} we conclude $c_5=0.$ Clearly, $c_6\ne0$ and thus $\Gamma$ must be concave near $S_6,$ i.e., in
\eqref{hazsut} we have $c_6=-1$ and $\theta_6= \pi/2$ so that $\sH^1(S_6)=\sH^1(S_4).$ Then the same observation
above yields $c_7=-1,$ $\sH^1(S_7)=\sqrt{2\alpha},$ the segments $S_8$ and $S_6$ satisfies the same relation in
\eqref{hcvbb} so that $\sH^1(S_8)=\sH^1(S_2).$ Then again $c_9=0$ and now it is the turn of three segments with
positive transition number and so on. By induction, we can show $c_{4i+1}=0,$ $\sH^1(S_{4i+3})=\sqrt{2\alpha},$
$\sH^1(S_{8i+2})=\sH^1(8i+8)=a$ and $\sH^1(S_{8i+4})=\sH^1(S_{8i+6})=b$ for all meaningful $i\ge1$ (i.e., those indices not exceeding $n$) and for some $a,b>\sqrt{2\alpha}$ with
$$
\frac{1}{a^2}  + \frac{1}{b^2} = \frac{1}{2\alpha}.
$$
Rotating $\Gamma$ by $\pm90^o$ we may assume that $S_1$ is vertical and directed upwards.

\begin{itemize}
\item Assume that $\Gamma$ is unbounded, i.e., $S_1$ and $S_n$ are half-lines. As $c_1=c_n=0,$ we have $n=4m+1$
for some $m\ge1.$ In this case, $\Gamma$ is a union of $m$ double-right-angles formed by two horizontal segments
of length $a$ and $b$ and one vertical segment of length $\sqrt{2\alpha},$ see Fig. \ref{fig:unb_trap}.

\item Assume that $\Gamma$ is bounded. As $S_n,S_1,S_2,S_3$ are consecutive segments, we have $c_n=-1$ and
$\sH^1(S_n)=\sH^1(S_2)=a.$ As in the previous case, grouping $(S_{4i+1},S_{4i+2},S_{4i+3},S_{4i+4})$ for $i\ge0$
and observing that corresponding quartets $(c_{4i+1},c_{4i+2},c_{4i+3},c_{4i+4})$ form an alternating series of
$(0,1,1,1)$ and $(0,-1,-1,-1),$ we deduce $n$ must be divisible by $8,$ i.e., $N=8m$ for some $m\ge1.$

We claim that $a=b=\sqrt{4\alpha}.$ Let us show that if $a\ne b,$ then $\Gamma$ cannot be not closed by means of
finite number of double-right-angles. Indeed, as $S_1$ is vertical and directed upwards (along with the vector
$(0,1)$), all segments $S_{4i+1}$ with zero transition number are also vertical and directed upwards. Moreover,
horizontal segments of length $a$ are directed to the right (along with $(1,0)$), while all horizontal segments
of length $b$ are directed to the left (along with $(-1,0)$). Thus, after passing each double-right-angle, we
move $|a-b|$ units along the horizontal axis, namely, to the right of $S_1$ if $a>b$ and to the left of $S_1$ if
$a<b.$ Thus, in $2m$ steps we reach  a point $Z$ with horizontal coordinate equal to $2n|a-b|,$ which is
nonzero by assumption $a\ne b.$ However, as $\Gamma$ is closed, $Z$ must be the starting point of $S_1,$ a
contradiction. Thus, $a=b.$

This equality shows that if $S_1$ lies on the vertical axis, all vertical lines with transition number $c\in
\{-1,0,1\}$ are contained in the vertical line $\{c\sqrt{4\alpha}\}\times\R.$ Moreover, $m$ double-right-angles
lie in the strip $[0,\sqrt{4\alpha}]\times\R$ and the remaining $m$ lie
in $[-\sqrt{4\alpha},0]\times\R.$
\end{itemize}

Now assume that all $c_i=1,$ i.e., $\Gamma$ is a bounded convex curve. Let $\sH^1(S_1)=a_1$ and
$\sH^1(S_2)=a_2.$ Then by \eqref{hazsut} there exists $a_3,a_4>0$ with
$$
\frac{1}{a_1^2}+\frac{1}{a_3^2} =\frac{1}{2\alpha}\quad\text{and}\quad \frac{1}{a_2^2}+\frac{1}{a_4^2}
=\frac{1}{2\alpha}
$$
such that $\sH^1(S_3)=a_3$ and $\sH^1(S_4)=a_4.$ Using \eqref{hazsut} and an induction argument, we can show
$\sH^1(S_{4i+j})=a_j$ for $j=1,2,3,4$ and meaningful $i\ge1.$ If $a_1\ne a_2,$ as in the case of
double-right-angles, we can
show that $\Gamma$ cannot be closed. Thus, $a_1=a_3$ and hence $a_2=a_4.$ Then by definition
$a_j=\sqrt{4\alpha},$ i.e., $\Gamma$ is a square (Wulff shape) of sidelength $\sqrt{4\alpha}.$
\end{proof}

\subsection{Translating solutions}

In this section we are interested in grim reaper-type solutions of \eqref{elastic_ode123}. We start with a general definition.

\begin{definition}\label{def:translate_curves}
Let $\norm$ be a crystalline anisotropy. An admissible polygonal  curve $\Gamma$ is called \emph{translating} provided that there exist a vector $\eta\in\S^1$ (called translation direction) and constant $\lambda>0$ (called translation velocity) such that the set of the signed heights from the segments/half-lines of
$$
\Gamma(t):=\lambda t \eta + \Gamma,\quad t\ge0,
$$
is the solution of the system \eqref{elastic_ode123}.
\end{definition}
 
Some comments are immediate:

\begin{itemize}
\item In view of Theorems \ref{teo:conver_stationar_sol} and \ref{teo:long_time_general}, any bounded polygonal curve cannot be translating.

\item Let $\Gamma$ be translating in the direction $\eta.$ Then the half-lines of $\Gamma$ should be parallel. We expect that if the half-lines are opposite directed, then for the evolution of segments, analogous long-time behaviours as in Theorems \ref{teo:conver_stationar_sol} and \ref{teo:long_time_general} should hold. When the half-lines are co-directed, 
one can readily check that $W^\norm$ should have two facets parallel   to $\eta,$  in which case, the segments $S_i$ of $\Gamma$ parallel to $\eta$ does not translate, i.e., $h_i\equiv0.$
 
\end{itemize}

Now we study translating solutions in the case of when the Wulff shape is  $W^\norm:=[-1,1]^2.$ Let us start with some examples.

\begin{figure}[htp!]
\includegraphics[width=0.8\textwidth]{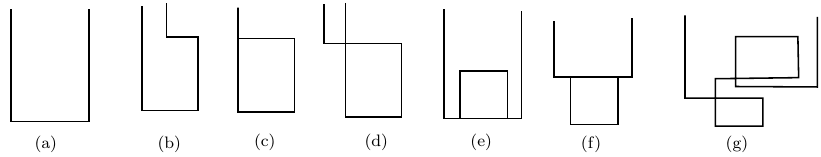}
\caption{\small Curves in Examples \ref{ex:conva}, \ref{ex:chalpa} and \ref{ex:non_transa}.}
\label{fig:translated}
\end{figure}

\begin{example}\label{ex:conva}
As we mentioned in Example \ref{ex:translate}, every unbounded admissible polygonal curve $\Gamma=S_1\cup S_2\cup S_3$ with angles $\theta_2=\theta_3=\pi/2$  (Fig. \ref{fig:translated} (a)) is translating in the direction of the half-lines.
Let us add two more segments to $\Gamma,$ i.e., let $\Gamma=\cup_{i=1}^5S_i$ be an unbounded admissible polygonal curve with angles $\theta_2=\theta_3=\theta_4=\pi/2$ and $\theta_5=3\pi/2.$  
Then $\Gamma$ is translating with velocity $\lambda>0$ if and only if $\lambda\in(0,\frac{2}{\sqrt{2\alpha}})$ and 
\begin{equation}\label{szdtvvb}
\sH^1(S_2) = \sqrt{2\alpha},\qquad \sH^1(S_3) = \sqrt{\frac{2\alpha}{1 - \frac{\lambda\sqrt{2\alpha}}{2}}},\qquad \sH^1(S_4) = 2- \lambda\sqrt{2\alpha}.
\end{equation}
Indeed, in view of the assumption on $\Gamma,$ the evolution equation \eqref{elastic_ode123} for segments  together with the translation assumption imply 
\begin{align*}
h_2' = \frac{2}{l_2} \Big(1 - \tfrac{2\alpha}{l_3^2}\Big) = \lambda,\qquad 
h_3' =\frac{2}{l_3} \Big(1 - \tfrac{2\alpha}{l_2^2}\Big) =0, \qquad 
h_4' = \frac{2}{l_4} \Big( - \tfrac{2\alpha}{l_3^2}\Big) = -\lambda,
\end{align*}
where $l_i:=\sH^1(S_i).$ Here $h_4'=-\lambda$ is because the normal to $S_4$ is $-\be_2.$  One can readily check that this equation admits a solution if and only if $\lambda\in(0,\frac{2}{\sqrt{2\alpha}})$ and in this case $(l_2,l_3,l_4)$ are uniquely given by \eqref{szdtvvb}. Note that 
if $\lambda+1>\frac{2}{\sqrt{2\alpha}},$ then $\Gamma$ is embedded (Fig. \ref{fig:translated} (b)), if $\lambda+1=\frac{2}{\sqrt{2\alpha}},$ then the half-lines of $\Gamma$ lie on the same straight line (Fig. \ref{fig:translated} (c)) and 
if $\lambda+1<\frac{2}{\sqrt{2\alpha}},$ then $\Gamma$ has a self-intersection (Fig. \ref{fig:translated} (d)).
\end{example}

Let us consider curves with a rectangle

\begin{example}\label{ex:chalpa}
(a) Assume that $\Gamma= \cup_{i=1}^6S_i$ is an unbounded admissible polygonal curve with angles $\theta_i=\pi/2$ (Fig. \ref{fig:translated} (e)). Then $\Gamma$ translates in the direction of half-lines with velocity $\lambda>0$ if and only if there exists $a\in (\sqrt{2\alpha},\sqrt{4\alpha})$ such that 
$$
\lambda = \Big(\tfrac{1}{2\alpha}\Big[\tfrac{1}{4(1-\frac{2\alpha}{a^2}) ^2 } + \tfrac{1}{4(\frac{4\alpha}{a^2} - 1) ^2 }  \Big]^{-1}\Big)^{1/2}
$$
and the lengths of the segments of $\Gamma$ are uniquely determined as 
\begin{gather*}
\sH^1(S_2) = \sH^1(S_6) = \tfrac{2}{\lambda} \Big(1 - \tfrac{2\alpha}{a^2}\Big),
\qquad 
\sH^1(S_3)=\sH^1(S_5) = a,
\qquad 
\sH^1(S_4) = \tfrac{2}{\lambda} \Big(\tfrac{4\alpha}{a^2} - 1\Big).
\end{gather*}
Indeed, the corresponding system of ODEs \eqref{elastic_ode123} is represented as 
\begin{align*}
\begin{cases}
h_2' = \tfrac{2}{l_2} \Big(1 - \tfrac{2\alpha}{l_3^2}\Big) = \lambda,\\
h_3' = \tfrac{2}{l_3} \Big(1 - 2\alpha \Big(\tfrac{1}{l_2^2} + \tfrac{1}{l_4^2}\Big)\Big) = 0,\\
h_4' = \tfrac{2}{l_4}\Big(1 - 2\alpha \Big(\tfrac{1}{l_3^2} + \tfrac{1}{l_5^2}\Big)\Big) = -\lambda,\\
h_5' = \tfrac{2}{l_4}\Big(1 - 2\alpha \Big(\tfrac{1}{l_4^2} + \tfrac{1}{l_6^2}\Big)\Big) = 0,\\
h_6' = \tfrac{2}{l_6}\Big(1 - \tfrac{2\alpha}{l_5^2}\Big) = \lambda,
\end{cases}
\end{align*}
where $l_i:=\sH^1(S_i).$
Comparing the equations for $h_3'$ and $h_5',$ we immediately find $l_2=l_6.$ Then the equations for $h_2'$ and $h_6'$ imply $l_3=l_5.$ Thus, from the equations for $h_2'$ and $h_4'$ as well as the positivity of $\lambda,$ we get $a:=l_3\in(\sqrt{2\alpha}, \sqrt{4\alpha}).$ Thus, given such $a,$ 
$$
l_2   = l_6 = \tfrac{2}{\lambda} \Big(1 - \tfrac{2\alpha}{a^2}\Big),\quad 
l_4 = \tfrac{2}{\lambda} \Big(\tfrac{4\alpha}{a^2} - 1\Big),
$$
and thus,
$$
\frac{1}{2\alpha} = \frac{1}{l_2^2} + \frac{1}{l_4^2} = \frac{\lambda^2}{4(1-\frac{2\alpha}{a^2})^2} +  \frac{\lambda^2}{4(\frac{4\alpha}{a^2} - 1)^2}.
$$
This equation admits a unique positive solution in $\lambda,$ and the assertion follows. Notice that $\Gamma$ has a vertical axial symmetry. 
\smallskip

(b) Assume that $\Gamma= \cup_{i=1}^6S_i$ is an admissible polygonal curve with angles $\theta_2=\theta_6=\pi/2$ and $\theta_3=\theta_4=\theta_5 = 3\pi/2$ (Fig. \ref{fig:translated} (f)).  As in (a), from \eqref{elastic_ode123} and the translation condition we deduce that $\Gamma$ is translating if and only if its velocity and length of segments satisfy the system 
$$
l_4 = \sqrt{2\alpha}, \qquad \tfrac{4\alpha}{l_2l_3^2} = \lambda,\qquad \tfrac{4\alpha}{l_6l_5^2} = \lambda, \qquad \tfrac{2}{l_4}\Big(1- 2\alpha\Big(\tfrac{1}{l_3^2}+ \tfrac{1}{l_5^2}\Big)\Big) = \lambda.
$$
Thus, if we fix $a:=l_2,$ then  
\begin{equation}\label{ahstczvb}
l_2 = a,\qquad 
l_3 = \sqrt{\frac{4\alpha}{\lambda a}},\quad l_4=\sqrt{2\alpha},\qquad 
l_5 = \sqrt{\frac{4\alpha}{2- \lambda\sqrt{2\alpha} - \lambda a}},
\qquad 
l_6= \frac{2-\lambda\sqrt{2\alpha} - \lambda a}{\lambda}.
\end{equation}
To have $l_5,l_6>0$ we should have $2-\lambda\sqrt{2\alpha} - \lambda a>0.$ This implies $\lambda$ must satisfy $0<\lambda<\frac{2}{a + \sqrt{2\alpha}}.$ In conclusion, $\Gamma$ is translating with velocity $\lambda$ if and only if there exists $a>0$ such that  $0<\lambda<\frac{2}{a + \sqrt{2\alpha}}.$ In this case, the lengths of the segments of $\Gamma$ are uniquely given by \eqref{ahstczvb}. Notice that, unlike (a), $\Gamma$ not necessarily admits a vertical symmetry.
\end{example}

Now consider an example with two rectangles, which is not translating.

\begin{example}\label{ex:non_transa}
Assume that $\Gamma=\cup_{i=1}^{10}S_i$ is an unbounded admissible polygonal curve with angles $\theta_2=\theta_7=\theta_8=\theta_9=\theta_{10}=\pi/2$ and 
$\theta_3=\theta_4=\theta_5=\theta_6=3\pi/2$ (Fig. \ref{fig:translated} (g)); thus, $\Gamma$ has two rectangles. Let us show that such $\Gamma$ cannot be translating. Indeed, by  \eqref{elastic_ode123} and the translating assumption
$$
h_7' = \tfrac{2}{l_7}\Big(1- \tfrac{1}{l_8^2}\Big) =0,\qquad 
h_9' = \tfrac{2}{l_9}\Big(1- 2\alpha\Big(\tfrac{1}{l_8^2} + \tfrac{1}{l_{10}^2}\Big)\Big) = 0.
$$
Thus,  by the first equation, $l_8=\sqrt{2\alpha},$ but from the second one,
$$
0 = 1 - \frac{2\alpha}{l_8^2} - \frac{2\alpha}{l_{10}^2} = - \frac{2\alpha}{l_{10}^2} <0,
$$
a contradiction.
\end{example}

These examples show that in general a complete classification of all possible translating solutions (as is done in the stationary case) is not so easy. Still, we can provide a classification under extra assumptions. 

Let us study translating polygonal curve.  admissible with $W^\norm=[-1,1];$ note that a curve is convex in this case if and only if all segments have the same nonzero transition constant. Moreover, since each rectangle contributes with four segments, by induction we can show that if $\Gamma = \cup_{i=1}^n S_i$ is convex and has two parallel co-directed half-lines, then $n=4m+3$ for some $m\ge0,$ which represents the number of rectangles.  Note that the curve $\Gamma$ with a single segment in Example \ref{ex:conva} and with a single rectangle in Example \ref{ex:chalpa} (a) are the only such translating convex curves for $m=0$ and $m=1.$ Below we study translating convex curves with $m\ge2.$

\begin{theorem}[Translating convex curves]
Let $W^\norm=[-1,1]^2$ and let $\Gamma=\cup_{i=1}^nS_i$ with $n=4m+3$ for some $m\ge2$ be a translating convex curve with angles $\theta_i=\pi/2$ (so that $c_1=c_n=0$ and $c_i=-1$ for all $2\le i\le n-1$).  Then there exists a constant $a>0$ satisfying 
\begin{equation}\label{ocpsvsbsn}
\frac{1}{2\alpha} <a < \frac{m+1}{2m\alpha}  
\end{equation}
such that the length of the segments and the translation velocity of $\Gamma$ are uniquely defined by $a.$  In particular, $\Gamma$ is uniquely defined (up to a translation) and has an axial symmetry. 
\end{theorem}

The explicit form of the length and velocity will be given in the course of the proof.
 
\begin{proof}
Without loss of generality assume that the translation direction is $\be_2.$ Since the vertical segments do not translate, $h_i\equiv0$ for all odd indices $i.$ For shortness, we write $l_i:=\sH^1(S_i).$
Then the evolution equation \eqref{elastic_ode123} is represented as 
\begin{align}\label{piyarike_pehlenazar}
\begin{cases}
h_2' = \tfrac{2}{l_2} \Big(1 - \tfrac{2\alpha}{l_3^2}\Big) = \lambda,\\
h_i' = \tfrac{2}{l_i} \Big(1 - 2\alpha \Big(\tfrac{1}{l_{i-1}^2} + \tfrac{1}{l_{i+1}^2}\Big)\Big) = 
\begin{cases}
0 & i=3,5,\ldots,4m+1,\\ 
-\lambda & i=4,8,\ldots,4m,\\
\lambda & i=6,10,\ldots,4m-2,
\end{cases}
\\
h_{4m+2}' = \tfrac{2}{l_{4m+2}}\Big(1 - \tfrac{2\alpha}{l_{4m+1}^2}\Big) = \lambda.
\end{cases}
\end{align}
In view of equations for $h_i'$ with odd indices,
\begin{equation}\label{ahstczv3}
\frac{1}{2\alpha} = \frac{1}{l_{4i-2}^2} + \frac{1}{l_{4i}^2} = \frac{1}{l_{4i}^2} + \frac{1}{l_{4i+2}^2} ,\quad i=1,\ldots,m,
\end{equation}
and hence, 
$$
l_2=l_6=\ldots = l_{4m+2}\quad \text{and}\quad l_4=l_8=\ldots = l_{4m}.
$$
Thus, from the equations for $h_2'$ and $h_{4m+2}'$ we deduce $l_3=l_{4m+1}.$ Moreover, from the equations for $h_{4i+2}'$ and $h_{4i}',$ we find
\begin{equation*}
a:=\frac{1}{l_3^2} + \frac{1}{l_5^2} = \frac{1}{l_7^2} + \frac{1}{l_9^2} = \ldots = 
\frac{1}{l_{4m-1}^2} + \frac{1}{l_{4m+1}^2} 
\end{equation*}
and 
\begin{equation*}
b:=\frac{1}{l_5^2} + \frac{1}{l_7^2} = \frac{1}{l_9^2} + \frac{1}{l_{11}^2} = \ldots = 
\frac{1}{l_{4m-3}^2} + \frac{1}{l_{4m-1}^2},
\end{equation*}
where by \eqref{piyarike_pehlenazar} and the positivity of $\lambda,$ the numbers $a$ and $b$ satisfy 
\begin{equation}\label{pehdtfzv}
a>\frac{1}{2\alpha}>b.
\end{equation}
Recalling $l_3=l_{4m+1},$ from these equalities we deduce 
\begin{equation}\label{ahstzc6b}
l_3 = l_{4m+1},\qquad l_5 = l_{4m-1},\quad \ldots, \quad l_{2m+1} = l_{2m+3}.
\end{equation}

Assume first $m$ is even, i.e., $m=2k$ for some $k\ge1,$ and setting $x_i:=\frac{1}{l_i^2}$ for the moment, consider the system of equations
\begin{gather}\label{sytstsgddd}
\begin{cases}
x_3+x_5 = x_7+x_9= \ldots = x_{4k-1}+x_{4k+1} = a,\\
x_5 + x_7 = x_9+x_{11} = \ldots = x_{4k+1} + x_{4k+3} =b.
\end{cases} 
\end{gather}
From the last equality in \eqref{ahstzc6b}, $x_{4k+1} = x_{4k+3},$ and hence, $x_{4k+1} = b/2.$ Then from the system \eqref{sytstsgddd} we find 
$$
\begin{cases}
x_{3+4i} = (k-i)a-\frac{2k-2i-1}{2}b & i=0,1,\ldots,k-1,k, \\[2mm]
x_{1+4i} = \frac{2k-2i+1}{2}b - (k-i)a & i=1,\ldots,k-1,k. 
\end{cases}
$$
Since $x_i=1/l_i^2,$ these numbers should be positive. By assumption \eqref{pehdtfzv}, we already have $x_{3+4i}>0, $ but to have $x_{1+4i}>0$ we should assume 
$$
b > \frac{2k-2i}{2k-2i+1}a \quad \text{for all $i=1,\ldots,k,$}
$$
which implies 
\begin{equation}\label{pehdtfzv2}
b>\frac{2k-2}{2k-1}a = \frac{m-2}{m-1}a.
\end{equation}
Thus, assuming $a$ and $b$ satisfy  \eqref{pehdtfzv} and \eqref{pehdtfzv2}, we can uniquely represent the length of segments with odd indices by $a$ and $b$ as 
$$
\begin{cases}
l_{3+4i} = l_{8k-4i+1} = \frac{1}{\sqrt{(k-i)a-\frac{2k-2i-1}{2}b}} & \quad \text{for $i=0,1,\ldots,k-1,k,$} \\[2mm]
l_{1+4i} = l_{8k-4i+3} = \frac{1}{\sqrt{\frac{2k-2i+1}{2}b - (k-i)a}} & \quad  \text{for $i=1,\ldots,k-1,k.$}
\end{cases}
$$
Then by \eqref{piyarike_pehlenazar} 
$$
l_4 = l_8 = \ldots = l_{4m} = \frac{2(2a\alpha - 1)}{\lambda},\qquad 
l_2 = l_6 = \ldots = l_{4m+2} = \frac{2(1-2b\alpha )}{\lambda},
$$
where the translation velocity $\lambda>0$ can be obtained as the unique solution of 
$$
\frac{1}{2\alpha} = \frac{\lambda^2}{4(2a\alpha-1)^2} + \frac{\lambda^2}{4(1-2b\alpha)^2},
$$
which comes from \eqref{ahstczv3}.

Finally, let us find a relation between $a$ and $b.$ From the first equation in \eqref{sytstsgddd} we have 
$$
1- 2\alpha \Big(ka-\frac{2k-1}{2}b\Big) = \frac{\lambda}{2}\,\frac{2(1-2b\alpha)}{\lambda},
$$
which simplifies to 
$$
b = \frac{m}{m+1}\,a.
$$
Notice that this choice of $b$ satisfies \eqref{pehdtfzv2}, but to obtain \eqref{pehdtfzv} we should ask additionally $a<\frac{m+1}{2m\alpha}$. Using these relations we can uniquely identify the length of the segments of $\Gamma$ and its translation velocity $\lambda$ in terms of $a.$ Since the angles of $\Gamma$ are known, these lengths uniquely define $\Gamma$ (up to a translations).
\medskip

Now assume $m$ is odd, i.e., $m=2k+1$ for some $k\ge1.$
In this case, the system \eqref{sytstsgddd} becomes  
\begin{gather*}
\begin{cases}
x_3+x_5 = x_7+x_9= \ldots = x_{4k+3}+x_{4k+5} = a,\\
x_5 + x_7 = x_9+x_{11} = \ldots = x_{4k+1} + x_{4k+3} =b,
\end{cases} 
\end{gather*}
where by the last equality in \eqref{ahstzc6b}, $x_{4k+3} = x_{4k+5}.$ Thus, as above $b$ satisfies 
\begin{equation}\label{pehdtfzv3}
b>\frac{2k-1}{2k}a = \frac{m-2}{m-1}a
\end{equation}
and
$$
\begin{cases}
l_{3+4i} = l_{8k-4i+5} = \frac{1}{\sqrt{ \frac{2k-2i+1}{2}a - (k-i)b } } & \quad\text{for $i=0,1,\ldots,k-1,k,$} \\[2mm]
l_{5+4i} = l_{8k-4i+3} = \frac{1}{\sqrt{ (k-i)b-\frac{2k-2i-1}{2}a }} & \quad \text{for $i=0,\ldots,k-1,k.$}
\end{cases}
$$
The definitions of $l_i$ with even indices and the translation velocity $\lambda$ do not change, and the relation between $a$ and $b$ reads as 
$$
1 - 2\alpha \Big(\frac{2k +1}{2}a - kb\Big) = \frac{\lambda}{2}\,\frac{2(1-2b\alpha)}{\lambda},
$$
which simplifies again to 
$$
b= \frac{m}{m+1}\,a,
$$
and if we assume \eqref{ocpsvsbsn}, $b$ satisfies both \eqref{pehdtfzv} and \eqref{pehdtfzv3}. Then as above, using $a$ we can define $\Gamma$ and $\lambda$ uniquely (up to a translation).

The explicit formulas for the length imply that the segments of $\Gamma$ are arranged symmetrically with respect to the median line of the strip bounded by the two half-lines of $\Gamma.$ Hence, $\Gamma$ itself is axially  symmetric.
\end{proof}

Inspired from Example \ref{ex:non_transa} we can perform a similar classification in a slightly more general situation. 
For simplicity, let us call an unbounded curve $\Gamma:=\cup_{i=1}^n S_i$ \emph{nice} if $n=4m+3$ for some $m\ge0$ and $\theta_{4j+3}=\theta_{4j+4}=\theta_{4j+5}=\theta_{4j+6}\in \{\pi/2,3\pi/2\}$ for any $j=0,1,\ldots,m-1.$  Notice that $c_i\ne 0$ for all indices $1<i<n$ such that $i$ is either odd or divisible by $4.$ 
Let us call a rectangle $R$ of $\Gamma$ \emph{convex} resp. \emph{concave} if all the angles $\Gamma$ at the vertices of $R$ are $\pi/2$ resp. $3\pi/2.$ Let us   also call the rectangle consisting of segments $S_{4j-1},S_{4j},S_{4j+1}$ and $S_{4j+2}$  
the \emph{$j$-th rectangle} of $\Gamma.$ 

\begin{theorem}[A possible structure of nice translating curves]
Let $W^\norm=[-1,1]^2$ and  $\Gamma=\cup_{i=1}^nS_i$ be a  nice translating curve.

\begin{itemize}
\item Assume that $\theta_2=\pi/2.$ Then $\Gamma$ is convex.

\item Assume that $\theta_2=3\pi/2.$ Then $\Gamma$ consists of a union of (possibly alternating) chain of convex and concave rectangles. Moreover, two convex rectangles cannot be consecutive and  $\theta_{n-1} = 3\pi/2.$ In particular, the first and last rectangles of $\Gamma$ are concave.
\end{itemize}
\end{theorem}

\begin{proof}
Assume that $\theta_2=\pi/2.$ Then by the  niceness assumption, $\theta_3=\theta_4=\theta_5=\pi/2.$ Assume that $\theta_6=3\pi/2$ so that $c_6=0.$ Then as in Example \ref{ex:non_transa}, using the ODE for $h_3'$ and $h_5'$ we find 
$$
h_3' = \tfrac{2}{l_3}\Big(1- 2\alpha\Big(\tfrac{1}{l_2^2} + \tfrac{1}{l_{4}^2}\Big)\Big) = 0,\qquad 
h_5' = \tfrac{2}{l_5}\Big(1- \tfrac{1}{l_4^2}\Big) =0.
$$
Thus,  
$$
0 = 1 - \frac{2\alpha}{l_2^2} - \frac{2\alpha}{l_{4}^2} = - \frac{2\alpha}{l_{2}^2} <0,
$$
a contradiction and $\theta_6=\pi/2.$ Using this argument and niceness inductively, we conclude all angles are $\pi/2,$ and hence, $\Gamma$ is concave.

Assume that $\theta_2=3\pi/2.$  Suppose that $\theta_{4i-2} = 3\pi/2$ and the $i$-th rectangle has angles $\theta_{4i-1}=\ldots=\theta_{4i+2}=\pi/2.$ If $\theta_{4i+3}=\pi/2,$ then the $(i-1)$-th and $i$-th rectangles are as in Example \ref{ex:non_transa}: using the ODE for $h_{4i-1}'$ and $h_{4i+1}'$ we get a contradiction as above. This contradiction together with the assumption $\theta_2=3\pi/2$ show that $\Gamma$ starts with a concave rectangle, each chain of consecutive concave   rectangles may end with a single convex rectangle and the last rectangle is also concave.
\end{proof}

\section{Final remarks in the unbounded case}\label{sec:open_problems}

In this final section we state some conjectures related to the crystalline elastic evolution of unbounded curves with a crystalline anisotropy $\norm$ in $\R^2.$

\begin{conjecture}
Let $\Gamma^0:=\cup_{i=1}^nS_i^0$ be an unbounded  admissible polygonal curve and $\{\Gamma(t)\}_{t\in[0,T^\dag)}$ be the unique maximal crystalline elastic flow starting from $\Gamma^0.$ If $T^\dag<+\infty,$ then: 

\begin{itemize}
\item[(a)] $t\mapsto \Gamma(t)$ is Kuratowski continuous in $[0,T^\dag],$

\item[(b)] the set 
$$
\bigcup_{t\in[0,T^\dag]} \bigcup_{i=2}^{n-1} S_i(t)
$$
is bounded,

\item[(c)] there exists an index $i\in\{2,\ldots,n-1\}$ such that 
$$
\lim\limits_{t\nearrow T^\dag}\,\,\sH^1(S_i(t)) = 0.
$$
\end{itemize}
\end{conjecture}

We expect the proof of this conjecture runs along the lines of Theorem \ref{teo:existence}, however, due to  unboundedness, we cannot drop $D$ in Corollary \ref{cor:segments_nonzero_curva} and hence, to prevent a translation to infinity in a finite time, we need some sort of rescaling. If the conjecture is true, we can restart the flow after singularity inductively and, as in Theorem \ref{teo:continue_flows},  we can define a unique globally defined  Kuratowski continuous elastic flow $\{\Gamma(t)\}_{t\ge0}.$ 

Another problem is related to the long-time behaviour of flow.

\begin{conjecture}
Let $\Gamma^0:=\cup_{i=1}^nS_i^0$ be an unbounded admissible polygonal curve and $\{\Gamma(t)\}_{t\ge0}$ be the globally defined elastic flow with finitely many restarts. 

\begin{itemize}
\item[(a)] Assume that half-lines $S_1^0$ and $S_n^0$ are not parallel. Then all bounded segments of $\Gamma(t)$ diverge to infinity as $t\to+\infty.$

\item[(b)] Assume that half-lines $S_1^0$ and $S_n^0$ are  parallel and ``co-directed'', i.e., $\nu_{S_1^0}+\nu_{S_2^0}=0.$ Then there exists a family $\{p_t\}\subset\R^2$ of vectors such that, as $t\to+\infty,$ the translated curves $p_t+\Gamma(t)$ converge to a unbounded translating solution $\Gamma^\infty.$ 

\item[(c)] Assume that half-lines $S_1^0$ and $S_n^0$ are  parallel and opposite-directed, i.e., $\nu_{S_1^0}=\nu_{S_2^0}.$ Then there exists a stationary curve $\Gamma^\infty$ such that $\Gamma(t)\overset{K}{\to} \Gamma^\infty$ as $t\to+\infty.$

\end{itemize}
\end{conjecture}

\subsection{Data availability}

The paper has no associated data.

\subsection*{Acknowledgment}
G. Bellettini  and M. Novaga acknowledge support from GNAMPA of INdAM. Sh. Kholmatov acknowledges support from the Austrian Science Fund (FWF) Stand-Alone project P 33716.

\end{document}